\documentclass{article} 
\usepackage[utf8]{inputenc} 
\usepackage{geometry} 
\usepackage{graphicx} 
\usepackage{booktabs} 
\usepackage{array} 
\usepackage{paralist} 
\usepackage{verbatim} 
\usepackage{subfig} 
\usepackage{amsmath, amssymb, layout, amsthm, amscd, appendix}
\usepackage{fancyhdr} 
\usepackage{hyperref}
\usepackage{xfrac, faktor}
\usepackage{comment}
\usepackage{pictexwd,dcpic}
\usepackage{chngcntr}
\counterwithin{figure}{subsection}
\newtheorem{theorem}{Theorem}[subsection]
\newtheorem{definition}[theorem]{Definition}
\newtheorem{lemma}[theorem]{Lemma}
\newtheorem{proposition}[theorem]{Proposition}
\newtheorem{corollary}[theorem]{Corollary}

\theoremstyle{remark}
\newtheorem{remark}[theorem]{Remark}
\theoremstyle{definition}
\newtheorem{example}[theorem]{Example}
\numberwithin{equation}{subsection}

\newcommand\numberthis{\addtocounter{equation}{1}\tag{\theequation}}
\usepackage{sectsty}
\allsectionsfont{\sffamily\mdseries\upshape} 

\usepackage[nottoc,notlof,notlot]{tocbibind} 
\usepackage[titles,subfigure]{tocloft} 

\title{Smoothness of Kuranishi atlases on Gromov-Witten moduli spaces}
\author{Robert Castellano}

\date{}

\begin{document}
\maketitle
\abstract{Kuranishi atlases were introduced by McDuff and Wehrheim to build a virtual fundamental class on moduli spaces of $J$-holomorphic curves and resolve some of the challenges in this field. This paper considers Gromov-Witten moduli spaces and shows they admit a smooth enough Kuranishi atlas to be able to define a Gromov-Witten virtual fundamental class in any virtual dimension. The key step for this result is the proof of a stronger gluing theorem.}
\tableofcontents

\section{Introduction}
The aim of this paper is to address the issue of smoothness for Kuranishi atlases in the context of Gromov-Witten moduli spaces. In \cite{mwtop, mwfund, mwiso} McDuff and Wehrheim build the theory of Kuranishi atlases on a compact metrizable space $X$. The prototype of the space $X$ for which this theory was developed is a compactified moduli space of $J$-holomorphic curves. In these papers, McDuff and Wehrheim develop the abstract framework for the theory of Kuranishi atlases and construct a virtual fundamental class for a space $X$ admitting a \textit{smooth} Kuranishi atlas (see Theorem \ref{theoremb}). The technique used is that of finite dimensional reduction and is based on the work by Fukaya-Ono \cite{fo} and Fukaya-Oh-Ohta-Ono \cite{fooo}. In \cite{notes} McDuff describes the construction of a Kuranishi atlas for the Gromov-Witten moduli space of closed genus zero curves. However, using the ``weak" gluing theorem of \cite{JHOL}, one does not obtain a smooth Kuranishi atlas, but rather a \textit{stratified smooth} Kuranishi atlas. In \cite{notes} McDuff describes how to adapt the theory of Kuranishi atlases to stratified smooth atlases and construct a virtual fundamental class in the case the space has virtual dimension zero.

\subsection{Statement of main results}
The purpose of this paper is to show that the genus zero Gromov-Witten moduli space admits a stronger structure than a stratified smooth Kuranishi atlas called a \textit{$C^1$ stratified smooth} ($C^1$ SS) Kuranishi atlas. This structure will be defined in Section \ref{c1ss} and it is shown in \cite{axiomme} that a $C^1$ SS structure is sufficient to obtain a virtual fundamental class for any virtual dimension.

\begin{theorem}\label{admitthm}
Let $(M^{2n},\omega,J)$ be a $2n$-dimensional symplectic manifold with tame almost complex structure $J$. Let $\overline{\mathcal{M}}_{0,k}(A,J)$ be the compact space of nodal $J$-holomorphic genus zero stable maps in class $A$ with $k$ marked points modulo reparametrization. Let $d=2n+2c_1(A)+2k-6$. Then $X:=\overline{\mathcal{M}}_{0,k}(A,J)$ has an oriented, $d$-dimensional, weak $C^1$ SS atlas $\mathcal{K}$.
\end{theorem}

The charts $U$ of $\mathcal{K}$ will lie over Deligne-Mumford space $\overline{\mathcal{M}}_{0,k}$ in the sense that there is a forgetful map
\[\pi:U\rightarrow \overline{\mathcal{M}}_{0,k}.\]
The construction of $U$ involves ``reparametrizing the gluing variables." As a preliminary result, we will give a prototype of this procedure of reparametrization on the Deligne-Mumford space $\overline{\mathcal{M}}_{0,k}$. We will show  $\overline{\mathcal{M}}_{0,k}$ admits a $C^1$ SS structure compatible with the atlas $\mathcal{K}$. The main result is the following.

\begin{proposition}\label{dm1}
Genus zero Deligne-Mumford space $\overline{\mathcal{M}}_{0,k}$ admits the structure of a $C^1$ SS manifold, denoted $\overline{\mathcal{M}}_{0,k}^{new}$, that is compatible with the $C^1$ SS atlas of Theorem \ref{admitthm} in the following sense. Let $\mathcal{K}$ denote the Kuranishi atlas on the Gromov-Witten moduli space $\overline{\mathcal{M}}_{0,k}(A,J)$ from Theorem \ref{admitthm} and $U$ a domain of $\mathcal{K}$. Then the forgetful map
\[\pi:U\rightarrow \overline{\mathcal{M}}_{0,k}^{new}\]
is $C^1$ SS.
\end{proposition}

It should be noted that while Theorem \ref{admitthm} and Proposition \ref{dm1} are expected to hold in higher genus as well, our constructions use the fact that $\overline{\mathcal{M}}_{0,k}$ is a smooth manifold with coordinates given by cross ratio functions. See Section \ref{smoothstructuresection} and Remark \ref{choices} for the specific uses of this structure.

This is a technical paper whose main component is strengthening the gluing theorem of \cite[Chapter 10]{JHOL} to a stronger ``$C^1$ gluing theorem." Using this stronger gluing theorem one can show that the spaces and structural maps constituting a Kuranishi atlas are $C^1$ stratified smooth. The precise statement of the $C^1$ gluing theorem is contained in Section \ref{c1gluing}; it roughly says that the gluing map is $C^1$ with respect to the gluing parameter.

Constructing a virtual fundamental cycle for a Gromov-Witten moduli space of any virtual dimension enables more applications. In the subsequent paper \cite{axiomme}, the author uses Kuranishi atlases to define genus zero Gromov-Witten invariants with homological constraints and show they satisfy the Gromov-Witten axioms of Kontsevich and Manin \cite{axioms}.

\begin{remark}
The question of how to construct a virtual fundamental class for moduli spaces of pseudoholomorphic curves has a long history, going back to the pioneering work of Fukaya-Ono \cite{fo}, Li-Tian \cite{litian}, Ruan \cite{ruanvir}, and Siebert \cite{siebert} among others. One can refer to the introduction of \cite{mwfund} for an overview of this history and the main issues of this problem. Recent discussions  have led to much new work in this field, for example \cite{fooo12, pardon, lisheng}. These papers all use finite dimensional reductions and traditional analytic techniques, rather than  the new scale Banach spaces and polyfolds of Hofer-Wysocki-Zehnder \cite{hwz1}. The McDuff-Wehrheim approach of Kuranishi atlases aims to construct the virtual fundamental class in a traditional way (i.e. as the zero set of a transverse perturbation) using the minimum amount of analysis and without relying on previous works. Their papers \cite{mwtop, mwfund, mwiso} address the topological aspects of this problem; the present paper completes the Kuranishi atlas approach by completing the necessary analysis. Using the minimum amount of analysis, we work in a stratified $C^1$ setting, rather than a $C^\infty$ setting. Fukaya-Oh-Ohta-Ono \cite{fooo12} also construct the virtual fundamental class as the zero set of a perturbed (multi)section, but uses $C^\infty$ charts and  Kuranishi structures as in \cite{fo}. In contrast,  Li and Sheng \cite{lisheng}, following the work of Ruan \cite{ruanvir}, define Gromov-Witten invariants by integrating differential forms over a so-called \textit{virtual neighborhood}, aiming to show that the lower strata can be ignored because the gluing maps decay  exponentially. Finally, Pardon \cite{pardon} works in the topological $(C^0)$ setting, constructing the virtual fundamental class as a cohomology class via a sheaf theoretic approach. This makes it unnecessary for him to address the compatibility issues between different local charts or with the isotropy group action, topics that are addressed in this paper in Section \ref{c1sscharts}.\hfill$\Diamond$
\end{remark}

Section \ref{background} contains the necessary background on Kuranishi atlases and stratified spaces. Section \ref{dm} studies genus zero Deligne-Mumford space and defines the $C^1$ SS structure of Proposition \ref{dm1}. This section can be understood without knowledge of Kuranishi atlases. In Sections \ref{coordfree} and \ref{chartsincoords} we will describe how to build the charts for the Kuranishi atlas on $X = \overline{\mathcal{M}}_{0,k}(A,J)$ described in Theorem \ref{admitthm}. Section \ref{c1sscharts} will show that these charts constitute a $C^1$ SS atlas assuming an appropriate gluing theorem. Proposition \ref{dm1} will also be proved in Section \ref{c1sscharts}. The technical gluing results are contained in Section \ref{c1gluing}. This section contains results useful beyond the context of Kuranishi atlases. Section \ref{c1gluing} can be treated as an appendix; the results of Section \ref{c1gluing} needed in the rest of the paper are stated in Section \ref{gwcharts}.\\

\noindent\textbf{Acknowledgements}: The author would like to thank his advisor Dusa McDuff for suggesting this project and giving countless comments on various versions of this paper. This paper would not be possible without these many useful conversations. Thanks also to Kenji Fukaya who pointed out an error in the original formulation of the $C^1$ gluing theorem.

\section{Background in Kuranishi atlases and stratified spaces}\label{background}
This section gives basic definitions and results regarding Kuranishi atlases and stratified spaces.
\subsection{Kuranishi atlases}\label{atlases}
The definitions of Kuranishi atlases were given in \cite{mwtop, mwfund, mwiso}. For a briefer survey (but more expansive than this paper) one can refer to \cite{notes}.

\begin{definition}\label{chart}
Let $F\subset X$ be a nonempty open subset. A \textbf{Kuranishi chart} for $X$ with footprint $F$ is a tuple $\mathbf{K} = (U,E,\Gamma,s,\psi)$ consisting of
\begin{itemize}
\item The \textbf{domain} $U$, which is a finite dimensional manifold.
\item The \textbf{obstruction space} $E$, which is a finite dimensional vector space.
\item The \textbf{isotropy group} $\Gamma$, which is a finite group acting smoothly on $U$ and linearly on $E$.
\item The \textbf{section} $s:U\rightarrow E$ which is a smooth $\Gamma$-equivariant map.
\item The \textbf{footprint map} $\psi:s^{-1}(0)\rightarrow X$ which is a continuous map that induces a homeomorphism
\[\underline{\psi}:\faktor{s^{-1}(0)}{\Gamma}\rightarrow F\]
where $F\subset X$ is an open set called the \textbf{footprint}.
\end{itemize}
The \textbf{dimension} of $\mathbf{K}$ is $\dim{\mathbf{K}} = \dim U - \dim E$.\\
Given a domain $U$ with isotropy group $\Gamma$, we will denote the quotient $\underline{U}:=\faktor{U}{\Gamma}$.
\end{definition}

Suppose that we have a finite collection of Kuranishi charts $(\mathbf{K}_i)_{i=1,\ldots,N}$ such that for each $I\subset\{1,\ldots,N\}$ satisfying $F_I:=\bigcap_{i\in I}F_i\neq\emptyset$, we have a Kuranishi charts $\mathbf{K}_I$ with
\begin{itemize}
\item Footprint $F_I$,
\item Isotropy group $\Gamma_I:=\prod_{i\in I}\Gamma_i$,
\item Obstruction space $E_I:=\prod_{i\in I}E_i$ on which $\Gamma_I$ acts with the product action.
\end{itemize}
Such charts $\mathbf{K}_I$ are known as \textbf{sum charts}. Thus, for $I\subset J$ we have a natural splitting $\Gamma_J\cong\Gamma_I\times\Gamma_{J\setminus I}$ and the induced projection $\rho_{IJ}^\Gamma:\Gamma_J\rightarrow\Gamma_I$. Moreover, we have a natural inclusion $\widehat{\phi}_{IJ}:E_I\rightarrow E_J$ which is equivariant with respect to the inclusion $\Gamma_I\hookrightarrow\Gamma_J$. Thus we can consider $E_I$ as a subset of $E_J$.

\begin{definition}
Let $X$ be a compact metrizable space.
\begin{itemize}
\item A \textbf{covering family of basic charts} for $X$ is a finite collection $(\mathbf{K}_i)_{i=1,\cdots,N}$ of Kuranishi charts for $X$ whose footprints cover $X$.
\item \textbf{Transition data} for $(\mathbf{K}_i)_{i=1,\cdots,N}$ is a collection of sum charts $(\mathbf{K}_J)_{J\subset\mathcal{I}_\mathcal{K},|J|\geq 2}$, and coordinate changes $(\widehat{\Phi}_{IJ})_{I,J\in\mathcal{I}_\mathcal{K}, I\subsetneq J}$ satisfying:
    \begin{enumerate}
    \item $\mathcal{I}_\mathcal{K} = \left\{I\subset\{1,\ldots,N\} | \bigcap_{i\in I}F_i \neq 0\right\}$.
    \item $\widehat{\Phi}_{IJ}$ is a coordinate change $\mathbf{K}_I\rightarrow\mathbf{K}_J$. For the precise definition of a coordinate change, see \cite{mwiso}. A coordinate change consists of
        \begin{enumerate}[(i)]
        \item A choice of domain $\underline{U}_{IJ}\subset \underline{U}_I$ such that $\mathbf{K}_I|_{\underline{U}_{IJ}}$ has footprint $F_J$.
        \item A $\Gamma_J$-invariant submanifold $\widetilde{U}_{IJ}\subset U_J$ on which $\Gamma_{J\setminus I}$ acts freely. Let $\widetilde{\phi}_{IJ}$ denote the $\Gamma_J$-equivariant inclusion.
        \item A \textbf{group covering} $(\widetilde{U}_{IJ},\Gamma_J,\rho_{IJ},\rho_{IJ}^\Gamma)$ of $(U_{IJ},\Gamma_I)$ where $U_{IJ}:=\pi^{-1}(\underline{U}_{IJ})\subset U_I$, meaning that $\rho_{IJ}:\widetilde{U}_{IJ}\rightarrow U_{IJ}$ is the quotient map $\widetilde{U}_{IJ}\rightarrow\faktor{\widetilde{U}_{IJ}}{\ker\rho_{IJ}^\Gamma}$ composed with a diffeomorphism $\faktor{\widetilde{U}_{IJ}}{\ker\rho_{IJ}^\Gamma}\cong U_{IJ}$ that is equivariant with respect to the induced $\Gamma_I$ action on both spaces. In particular, this implies that $\underline{\rho}:\underline{\widetilde{U}}\rightarrow\underline{U}$ is a homeomorphism (see \cite[Lemma 2.1.5]{mwiso}).
        \item The linear equivariant injection $\widehat{\phi}_{IJ}:E_I\rightarrow E_J$ as above.
    \item The inclusions $\widetilde{\phi}_{IJ},\widehat{\phi}_{IJ}$ and the covering $\rho_{IJ}$ intertwine the sections and the footprint maps in the sense that
        \begin{align*}
        s_J\circ\widetilde{\phi}_{IJ} &= \widehat{\phi}_{IJ}\circ s_I\circ\rho_{IJ}\quad& &\textnormal{on }\widetilde{U}_{IJ}\\
        \psi_J\circ\widetilde{\phi}_{IJ} &= \psi_I\circ\rho_{IJ}& &\textnormal{on }s_J^{-1}(0)\cap\widetilde{U}_{IJ} = \rho_{IJ}^{-1}(s_I^{-1}(0)).
        \end{align*}
    \item The map $s_{IJ}:=s_I\circ\rho_{IJ}$ is required to satisfy an \textbf{index condition}. This ensures that any two charts that are related by a coordinate change have the same dimension.
        \end{enumerate}
    \end{enumerate}
\end{itemize}
\end{definition}
One needs to express a way in which coordinate changes are compatible. This is described by cocycle conditions. Charts $\mathbf{K}_\alpha$ where $\alpha=I,J,K$ with $I\subset J\subset K$ a \textbf{weak cocycle condition} / \textbf{cocycle condition} / \textbf{strong cocycle condition} describe to what extent the composition $\widehat{\Phi}_{JK}\circ\widehat{\Phi}_{IJ}$ and the coordinate change $\widehat{\Phi}_{IK}$ agree. These conditions require them to agree on increasingly large subsets. See \cite{mwiso} for a complete discussion of cocycle conditions.

In constructions, one can achieve a weak cocycle condition, while a strong cocycle condition is what is required for the construction of a virtual fundamental class.
\begin{definition}
A \textbf{weak Kuranishi atlas of dimension d} on a compact metrizable space $X$ is transition data
\[\mathcal{K} = (\mathbf{K}_I,\Phi_{IJ})_{I,J\in\mathcal{I}_\mathcal{K},I\subsetneq J}\]
for a covering family $(\mathbf{K}_i)_{i=1,\ldots,N}$ of dimension d for $X$ that satisfies a weak cocycle condition. Similarly a \textbf{(strong) atlas} is required to satisfy a (strong) cocycle condition.
\end{definition}
The main theorem regarding Kuranishi atlases is the following.
\begin{theorem}\label{theoremb}
Let $\mathcal{K}$ be a oriented, weak, d dimensional smooth Kuranishi atlas on a compact metrizable space $X$. Then $\mathcal{K}$ determines a cobordism class of oriented, compact weighted branched topological manifolds, and an element $[X]_\mathcal{K}^{vir}\in \check{H}_d(X;\mathbb{Q})$. Both depend only on the oriented cobordism class of $\mathcal{K}$. Here $\check{H}_{*}$ denotes $\check{C}$ech homology.
\end{theorem}
In the case of trivial isotropy groups, this is \cite[Theorem B]{mwfund}. For nontrivial isotropy groups see \cite[Theorem A]{mwiso}.

The process of constructing the virtual fundamental cycle $[X]_{\mathcal{K}}^{vir}$ takes place in several steps. First, a weak atlas is \emph{tamed}, which is a procedure that shrinks the domain and implies desirable topological properties. Next, a \emph{reduction} of the atlas is constructed. In the reduced atlas one can construct an appropriate perturbation section from which one builds a compact weighted branched manifold. The virtual fundamental cycle is then built from this manifold. These processes will not be dealt with in this paper. Their role in the $C^1$ stratified smooth setting is addressed in \cite{axiomme}.

This paper will be primarily interested in the case $X = \overline{\mathcal{M}}_{0,k}(A,J)$, the compact space of nodal $J$-holomorphic genus zero stable maps on a symplectic manifold in homology class $A$ with $k$ marked points modulo reparametrization, where $J$ is a tame almost complex structure. Sections \ref{coordfree} and \ref{chartsincoords} will construct a Kuranishi atlas on this space $X$ and Section \ref{c1sscharts} will show it satisfies the necessary conditions.

\subsection{Stratified smooth spaces}\label{c1ss}
In this section we will give an overview of stratified spaces and their context in Kuranishi atlases. The following definitions are taken from \cite{notes}.
\begin{definition}
A pair $(X,\mathcal{T})$ of a topological space $X$ and a finite partially ordered set $\mathcal{T}$ is called a \textbf{stratified space with strata} $(X_T)_{T\in\mathcal{T}}$ if the following conditions hold:
\begin{enumerate}[(i)]
\item The space $X$ is a disjoint union of the strata.
\item The closure $\overline{X_T}$ is contained in $\bigcup_{S\leq T}X_S$.
\end{enumerate}
\end{definition}
\begin{example}\label{stratex}
$\mathbb{R}^k\times\mathbb{C}^{\underline{n}}$ carries a stratification in the following way:
\begin{enumerate}
\item The set $\mathcal{T}$ is the power set of $\{1,\cdots,n\}$.
\item The stratum $(\mathbb{R}^k\times\mathbb{C}^{\underline{n}})_T$ is given by
\[(\mathbb{R}^k\times\mathbb{C}^{\underline{n}})_T = \{(x,\mathbf{a})\in\mathbb{R}^k\times\mathbb{C}^n ~|~ \mathbf{a} = (a_1,\cdots,a_n), a_i\neq0\iff i\in T\}.\]
\end{enumerate}
\end{example}
\begin{example}
The genus zero Deligne-Mumford space $\overline{\mathcal{M}}_{0,k}$ admits a stratification given by
\begin{equation}\label{dmstrat}
\overline{\mathcal{M}}_{0,k} = \bigsqcup_{0\leq p \leq k-3} \mathcal{M}_p
\end{equation}
where $\mathcal{M}_p$ is the moduli space of curves with exactly $p$ nodes. This is the stratification we will use.

The genus zero Gromov-Witten moduli space $\overline{\mathcal{M}}_{0,k}(A,J)$ also admits a stratification similar to (\ref{dmstrat}):

\[\overline{\mathcal{M}}_{0,k}(A,J) = \bigsqcup_{p} \mathcal{M}_p(A,J).\]

Deligne-Mumford space $\overline{\mathcal{M}}_{0,k}$ also admits another stratification given by
\[\overline{\mathcal{M}}_{0,k} = \bigsqcup_{\substack{k\textnormal{-labelled}\\\textnormal{trees } T}} \mathcal{M}_T\]
where $\mathcal{M}_T$ is the moduli space of curves modelled on the labelled tree $T$.
\end{example}
\begin{definition}
A \textbf{stratified continuous} (resp. $\mathbf{C^1}$) map $f:(X,\mathcal{T})\rightarrow (Y,\mathcal{S})$ between stratified spaces (resp. $C^1$ manifolds) is a continuous (resp. $C^1$) map $f:X\rightarrow Y$ such that
\begin{enumerate}[(i)]
\item $f$ maps strata to strata and hence induces a map $f_{*}:\mathcal{T}\rightarrow\mathcal{S}$ such that $f(X_T)\subset Y_{f_{*}T}$.
\item The map $f_{*}$ preserves order, that is $T\leq S\Rightarrow f_{*}T \leq f_{*}S$.\footnote{This is a different definition than that in \cite{notes} where it is required that $f_{*}$ preserves \textit{strict} order. We use this definition so that relevant maps such as the evaluation map $ev:U_{I}\rightarrow M$ are stratified $C^1$.}
\end{enumerate}
\end{definition}
\begin{definition}
Let $f:U\rightarrow \mathbb{R}^k\times\mathbb{C}^{\underline{m}}$ be a stratified $C^1$ map defined on an open subset $U\subset \mathbb{R}^\ell\times \mathbb{C}^{\underline{n}}$. We call $f$ a \textbf{$\mathbf{C^1}$ stratified smooth ($\mathbf{C^1}$ SS) map} if it restricts to a smooth map $U_T=U\cap(\mathbb{R}^\ell\times\mathbb{C}^{\underline{n}})_T\rightarrow (\mathbb{R}^k\times\mathbb{C}^{\underline{m}})_{f_*T}$ on each stratum $T$.
\end{definition}

It is easy to see that the composition of two $C^1$ SS maps is $C^1$ SS. Hence one can define a \textbf{$\mathbf{C^1}$ SS manifold} as having local charts of the form $\mathbb{R}^k\times\mathbb{C}^{\underline{n}}$ as in Example \ref{stratex} and requiring that transition functions are $C^1$ SS maps. Also of use later will be the notion of a subset $Y$ being a \textbf{$C^1$ SS submanifold} of a $C^1$ SS manifold $X$. Here we say that $Y\subset X$ is a $C^1$ SS submanifold if $Y$ is a $C^1$ SS manifold with its $C^1$ SS structure induced from $X$.

One can also define the notion of a \textbf{$\mathbf{C^1}$ SS Kuranishi atlas} on a stratified space $X$ by requiring all conditions to hold in the $C^1$ SS category. In particular, the footprint map should be stratified continuous. Having all conditions hold in the $C^1$ SS category is clearly weaker than requiring they all hold in the $C^\infty$ category. In Section \ref{c1sscharts}, we will prove Theorem \ref{admitthm} which states that $X = \overline{\mathcal{M}}_{0,k}(A,J)$ admits a $C^1$ SS atlas. The analog of Theorem \ref{theoremb} for $C^1$ SS Kuranishi atlases is proved in \cite{axiomme}.

\section{Reparametrizing genus zero Deligne-Mumford space}\label{dm}
This section will construct $\overline{\mathcal{M}}_{0,k}$ using cross ratios and then use this construction to define the $C^1$ SS structure of Proposition \ref{dm1}.

\subsection{Genus zero Deligne-Mumford space via cross ratios}\label{crcoordinates}
This section describes $\overline{\mathcal{M}}_{0,k}$, the genus zero Deligne-Mumford space with $k$ marked points, in terms of cross ratios. This will be the basis of the construction in Section \ref{smoothstructuresection}. Then in Example \ref{m05} we give a key explicit example that will be the basis for proofs in Section \ref{smoothstructuresection}. For a more detailed exposition on the construction of $\overline{\mathcal{M}}_{0,k}$, see \cite[Appendix D]{JHOL}. The fact that $\overline{{\mathcal{M}}}_{0,k}$ is a manifold with coordinates given by well understood cross ratio functions is one of the reasons this paper restricts to the genus zero setting.

Given four points $z_0,z_1,z_2,z_3$ on $S^2$ such that no more than two points are equal, we define the cross ratio $w_{0123}=w(z_0,z_1,z_2,z_3)\in S^2$ by
\begin{equation}\label{crossratio}
w_{0123} = \frac{(z_1-z_2)(z_3-z_0)}{(z_0-z_1)(z_2-z_3)}.
\end{equation}
This map is invariant under M\"{o}bius transformations and is normalized so that
\[w(0,1,\infty,\lambda) = \lambda.\]
The cross ratio map also satisfies the following symmetry relations.
\begin{equation}\label{crsymmetry}
w_{jikl} = w_{ijlk} = 1-w_{ijkl},\qquad w_{ikjl} = \frac{w_{ijkl}}{w_{ijkl}-1}.
\end{equation}
Furthermore, given five points, different cross ratios satisfy the following recursion relation.
\begin{equation}\label{crrecursion}
w_{1234} = \frac{w_{0124} -1}{w_{0124} - w_{0123}}.
\end{equation}
The cross ratio map extends to a map on $\overline{\mathcal{M}}_{0,4}$ as the usual cross ratio map on smooth domains and the composition of collapsing one component with the usual cross ratio on nodal domains. Note that the cross ratios of nodal domains are $0,1,$ or $\infty$ and the cross ratio of four distinct points on a smooth domain is not $0,1,$ or $\infty$. We can further extend the cross ratio to a map on $\overline{\mathcal{M}}_{0,k}$ via the forgetful maps $\overline{\mathcal{M}}_{0,k}\rightarrow\overline{\mathcal{M}}_{0,4}$.

The set of all cross ratios $\{w_{ijkl}:\overline{\mathcal{M}}_{0,k}\rightarrow S^2\}$ gives rise to a map
\[\overline{\mathcal{M}}_{0,k}\rightarrow(S^2)^N\]
where $N=k(k-1)(k-2)(k-3)$. Knudsen \cite{knudsen} proves that this map is injective and its image is a smooth submanifold. Hence $\overline{\mathcal{M}}_{0,k}$ inherits a smooth structure from $(S^2)^N$ via cross ratio maps. This proof is based on an explicit construction of coordinate charts. Given a point $\mathbf{w}_0$ in the image of this map, Knudsen chooses $k-3$ cross ratios from which all other cross ratios are a smooth function of these $k-3$ cross ratios; this is done using the symmetry and recursion relations (\ref{crsymmetry}) and (\ref{crrecursion}). The $k-3$ cross ratios are chosen to satisfy the following conditions.
\begin{enumerate}[(a)]
\item\label{cr1} For each edge, a cross ratio $w_{ijkl}$ is chosen such that $w_{ijkl}$ is $0$ at $\mathbf{w}_0$ and is such that $w_{ijkl}(\mathbf{w})=0$ if and only if $\mathbf{w}$ is modelled on a tree with this edge.
\item\label{cr2} The remaining cross ratios determine, up to complex automorphism, the position of the special points. These cross ratios will not be $0,1$ or $\infty$ in a neighborhood of $\mathbf{w}_0$.
\end{enumerate}
The coordinates of type (\ref{cr1}) are called \textit{gluing parameters} and will be denoted with variables $\mathbf{a}$ in this section and Section \ref{chartsincoords}. The coordinates of type (\ref{cr2}) are denoted with variables $\mathbf{b}$.

\begin{example}[Coordinates on $\overline{\mathcal{M}}_{0,5}$]\label{m05}
In this example we will explicitly describe coordinates on $\overline{\mathcal{M}}_{0,5}$. There are three possible tree configurations in $\overline{\mathcal{M}}_{0,5}$; we will describe coordinates in a neighborhood of each configuration in terms of cross ratios.
\begin{figure}
\begin{center}
\includegraphics[scale=0.65]{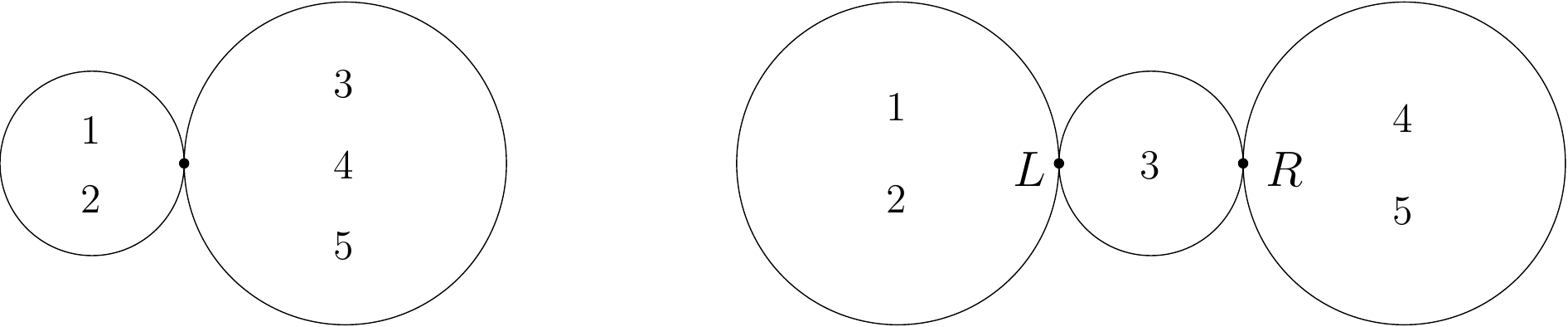}
\caption{Configuration 1 (left) and Configuration 2 (right)}\label{m05fig}
\end{center}
\end{figure}\\

\noindent\underline{Configuration 1}: This tree structure is two vertices with marked points as in Configuration $1$ in Figure \ref{m05fig}. There are five distinct cross ratios: $w_{1453}, w_{2453}, w_{1342}, w_{1352}, w_{1452}$ by which we mean all others are obtained from these via the symmetry relations (\ref{crsymmetry}). For simplicity, we define variables
\begin{equation}\label{crvariables}
x:=w_{1453},\quad y:=w_{2453},\quad z:=w_{1342},\quad u:=w_{1352},\quad v:=w_{1452}.
\end{equation}
According to the scheme given by conditions (\ref{cr1}) and (\ref{cr2}) above, the variables $z,u,v$ are gluing parameters (and are $0$ for this tree structure), while $x,y$ are $\mathbf{b}$ variables (and are away from $0,1$ and $\infty$ in a neighborhood of this tree structure). Therefore, coordinates in this configuration are given by a choice of one of $x,y$ and one of $z,u,v$. We can use the recursion relation (\ref{crrecursion}) and symmetry relations (\ref{crsymmetry}) to find the relations between these variables. For example, (\ref{crrecursion}) implies that
\[w_{2453} = \frac{w_{1243}-1}{w_{1243} - w_{1245}}.\]
Then applying (\ref{crsymmetry}), we see that
\begin{align*}
w_{2453} &= \frac{-w_{1234}}{1-w_{1234} - \frac{w_{1425}}{w_{1425}-1}}\\
        &=\frac{-\frac{w_{1324}}{w_{1324}-1}}{1-\frac{w_{1324}}{w_{1324}-1}-\frac{w_{1425}}{w_{1425}-1}}\\
        &=\frac{\frac{w_{1342}-1}{-w_{1342}}}{1+\frac{w_{1342}-1}{-w_{1342}}-\frac{w_{1425}}{w_{1425}-1}}.
\end{align*}
Simplifying and using the variables from $(\ref{crvariables})$ we get the relation
\begin{equation}\label{m0512}
y=\frac{v-vz}{v-vz+z}.
\end{equation}
By similar calculations, we see
\begin{align}
\label{m0511}z&=\frac{xv-v}{xv-x}\\
\label{m0513}x&=\frac{z-u}{zu-u}\\
\label{m0514}u&=\frac{x+y}{xy-x}\\
\label{m0515}v&=xu.
\end{align}
One can use these relations to express any variable as a function of a chosen set of coordinates. For example, choosing coordinates $x,z$ we see that
\begin{align}
\nonumber y &= \frac{v-vz}{v-vz+z}\quad & \textnormal{using (\ref{m0512})}\\
\nonumber  & =\frac{\frac{xz}{xz-x+1}-\frac{xz}{xz-x+1}z}{\frac{xz}{xz-x+1}-\frac{xz}{xz-x+1}z+z}\quad & \textnormal{using (\ref{m0511})}\\
\label{m05y} &=x-xz&\\
\label{m05u}u&=\frac{z}{zx-x+1}\quad&\textnormal{using (\ref{m0513})}\\
\label{m05v}v&=\frac{xz}{xz-x+1}\quad&\textnormal{using (\ref{m0511})}
\end{align}

\noindent\underline{Configuration 2}: This tree structure consists of three vertices, two edges $L$ and $R$, and marked points as in Configuration 2 in Figure \ref{m05fig}. As in the previous case, there are five cross ratios, $w_{1453}, w_{2453}, w_{1342}, w_{1352}, w_{1452}$ from which all others are obtained from these via the symmetry relations (\ref{crsymmetry}). Using the same notation as (\ref{crvariables}) and the scheme given by conditions (\ref{cr1}) and (\ref{cr2}), $x,y,z,$ and $u$ are all $\mathbf{a}$ variables. All are $0$ for this configuration; when the node $L$ is resolved $x$ and $y$ become nonzero and when the node $R$ is resolved $z$ and $u$ become nonzero. Therefore, the choice of $x$ or $y$ and the choice of $z$ or $u$ gives coordinates in a neighborhood of this configuration. Relations (\ref{m0511}), (\ref{m0512}), (\ref{m0513}), (\ref{m0514}) continue to hold. Note that $v$ is $0$ for this configuration and continues to be $0$ when only one node is resolved. Thus, it cannot be chosen as a coordinate.

Using (\ref{m0511}), (\ref{m0512}), (\ref{m0513}),( \ref{m0514}), one can express any variable as a chosen set of coordinates. For example, just as in Configuration 1, choosing coordinates $x,z$ the variables $y,u,v$ are given by (\ref{m05y}), (\ref{m05u}), (\ref{m05v}).\\

\noindent\underline{Configuration 3}: This tree configuration has one vertex, that is the domain is smooth. In the notation of (\ref{crvariables}), $x,y,z,u,v$ are all $\mathbf{b}$ variables and are not $0,1,$ or $\infty$. The choice of any two gives coordinates.
\end{example}

\subsection{Changing the smooth structure on $\overline{\mathcal{M}}_{0,k}$}\label{smoothstructuresection}
We will now use the description of $\overline{\mathcal{M}}_{0,k}$ via cross ratios to define a new (but diffeomorphic) smooth structure on $\overline{\mathcal{M}}_{0,k}$. Modifying the smooth structure on $\overline{\mathcal{M}}_{0,k}$ will serve as a prototype of giving the charts $U$ of our Gromov-Witten Kuranishi atlas a $C^1$ SS structure. The $C^1$ SS structure on $\overline{\mathcal{M}}_{0,k}$ will be that which is required for Proposition \ref{dm1}.

We will first define a new smooth structure on $S^2$. Consider $S^2=\mathbb{C}\cup\{\infty\}$. We will modify the structure around $0,1,\infty\in S^2$. Let $O\subset \mathbb{C}$ be a neighborhood of $0$ and $\phi:O\rightarrow S^2$ be a conformal chart around $0\in S^2$. Let $p>2$ and $\psi$ be the ``reparametrization" map
\begin{align}
\nonumber \psi:\mathbb{C} &\rightarrow \mathbb{C}\\
z &\mapsto |z|^{p-1}z. \label{psireparam}
\end{align}
The map $\psi$ fixes $0$ and away from $0$ in polar coordinates $\psi$ is of the form
\[\psi(r,\theta) = (r^p,\theta).\]
The following are the key properties of $\psi$ that we will use:
\begin{itemize}
\item $\psi$ is multiplicative, that is
\begin{equation}\label{multprop}
\psi(zw) = \psi(z)\psi(w).
\end{equation}
\item $\psi$ fixes $0$.
\item $\psi$ is $C^1$ at $0$ and $C^\infty$ everywhere else.
\item $\psi^{-1}$ is $C^\infty$ away from $0$, but not differentiable at $0$.
\end{itemize}

Charts around $0\in S^2$ for our new smooth structure will be given by
\begin{equation}\label{newcharts}
\widetilde{\phi} = \phi\circ\psi.
\end{equation}
where $\phi$ is a conformal chart as above. Charts around $1$ and $\infty$ are done in a similar manner by composing with a M\"{o}bius transformation. Charts around other points remain unchanged. The next lemma states that the charts $\widetilde{\phi}$ define a $C^1$ SS structure on $S^2$.

\begin{lemma}\label{s2c1}
Let $O\subset \mathbb{C}$ be a neighborhood of $0$ and $\phi:O\rightarrow S^2$ be a conformal chart around $0\in S^2$. Define corresponding charts $\widetilde{\phi} = \phi\circ\psi$ by precomposing with the reparametrization map $\psi$ from $(\ref{psireparam})$. Define charts $\widetilde{\phi}$ around $1,\infty\in S^2$ by composing with M\"{o}bius transformations. Define charts not containing $0,1,\infty$ using charts from the standard smooth structure on $S^2$. Then these charts define a $C^1$ SS manifold structure on $S^2$ where $S^2$ has stratification $X_0 = \{0,1,\infty\}, X_1= S^2\setminus \{0,1,\infty\}$.
\end{lemma}
\begin{proof}
To check that these charts define a $C^1$ SS structure, we must check that the induced transition maps are $C^1$ SS local diffeomorphisms. It is clear that the transition maps for charts not containing $\{0,1,\infty\}$ are $C^\infty$ local diffeomorphisms. Therefore, it suffices to consider charts containing these points. Without loss of generality we consider $\phi_1,\phi_2:O\rightarrow S^2$ two conformal charts around $0\in S^2$ for the usual smooth structure on $S^2$. Let $\widetilde{\phi}_1,\widetilde{\phi}_2$ be the corresponding new charts. Associated to $\phi_1,\phi_2$ is a transition map $T = \phi_2^{-1}\circ\phi_1:O\rightarrow O$. Associated to $\widetilde{\phi}_1,\widetilde{\phi}_2$ is a transition map
\[\widetilde{T}:O \rightarrow O\]
\[\widetilde{T} = \widetilde{\phi}_2^{-1}\circ\widetilde{\phi}_1 = \psi^{-1}\circ T \circ \psi.\]
See Figure \ref{comm} for a schematic of these maps.

\begin{figure}
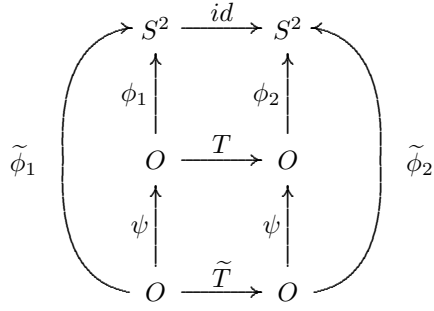

\begin{center}
\mbox{}
\begindc{\commdiag}[5]
\obj(0,20){$S^2$}
\obj(10,20){$S^2$}
\obj(0,10){$O$}
\obj(10,10){$O$}
\obj(0,0){$O$}
\obj(10,0){$O$}
\mor(0,20)(10,20){$id$}
\mor(0,10)(0,20){$\phi_1$}
\mor(10,10)(10,20){$\phi_2$}
\mor(0,10)(10,10){$T$}
\mor(0,0)(0,10){$\psi$}
\mor(10,0)(10,10){$\psi$}
\mor(0,0)(10,0){$\widetilde{T}$}
\cmor((-2,0)(-6,3)(-7,10)(-6,17)(-2,20)) \pright(-10,10){$\widetilde{\phi}_1$}
\cmor((12,0)(16,3)(17,10)(16,17)(12,20)) \pleft(20,10){$\widetilde{\phi}_2$}
\enddc
\end{center}
\caption{Change of coordinates in new smooth structure}\label{comm}
\end{figure}

We must check that $\widetilde{T}$ is $C^1$ at $0$ and $C^\infty$ elsewhere. This suffices because $\widetilde{T}^{-1}$ has the same form as $\widetilde{T}$. The map $T$ is a M\"{o}bius transformation fixing $0$ and is hence of the form
\begin{equation}\label{T}
T(z) = \frac{z}{bz+d},\quad d\neq 0.
\end{equation}
Therefore,
\begin{align*}
\widetilde{T}(z) &= \psi^{-1}\circ T \circ \psi (z)\\
                 &= \psi^{-1}\left(\frac{\psi(z)}{b\psi(z) + d}\right)\\
                 &= \frac{z}{\psi^{-1}(b\psi(z) + d)}
\end{align*}
where the last equality uses property (\ref{multprop}). So we see that $\widetilde{T}$ is $C^1$ at $0$ and $C^\infty$ elsewhere because $\psi$ is $C^1$ at $0$ and $C^\infty$ elsewhere and because $b\psi(z) + d \neq 0$ and $\psi^{-1}$ is $C^\infty$ away from $0$. This proves that we have defined a new $C^1$ SS structure on $S^2$.
\end{proof}

Let $S^2_{new}$ denote $S^2$ with the $C^1$ SS structure given by Lemma \ref{s2c1}. Hence we also get $(S^2_{new})^N$. As described in Section \ref{crcoordinates}, $\overline{\mathcal{M}}_{0,k}$ sits as a submanifold of $(S^2)^N$ by using cross ratios as charts. The next lemma says that this induces a $C^1$ SS structure $\overline{\mathcal{M}}_{0,k}^{new}$. This will be the $C^1$ SS structure of Proposition \ref{dm1}.

\begin{proposition}\label{dmlemma}
Let $\overline{\mathcal{M}}_{0,k}\subset (S^2)^N$ be a submanifold with inclusion given by cross ratio maps. Then $\overline{\mathcal{M}}_{0,k}$ is a $C^1$ SS submanifold of $(S^2_{new})^N$, where $S^2_{new}$ is the $C^1$ SS structure given by Lemma \ref{s2c1}. We will denote $\overline{\mathcal{M}}_{0,k}$ with this new $C^1$ SS structure by $\overline{\mathcal{M}}_{0,k}^{new}$.
\end{proposition}
\begin{proof}
As in the proof of Lemma \ref{s2c1}, we must prove that the transition functions for $\overline{\mathcal{M}}_{0,k}$ associated to the charts on $(S^2_{new})^N$ are $C^1$ SS local diffeomorphisms. We will first study the form of the transition functions by creating a diagram similar to Figure \ref{comm}.

Let $w\in\overline{\mathcal{M}}_{0,k}$. For the normal smooth structure on $S^2$, there is a chart $\phi:V\rightarrow\overline{\mathcal{M}}_{0,k}$ around $w$, where $V\subset\mathbb{C}^{k-3}$ and $\phi$ is given via cross ratios as described earlier. The set $V$ splits as $V=V_{\mathbf{a}}\times V_{\mathbf{b}}\subset \mathbb{C}^{K}\times \mathbb{C}^{k-3-K}$ where $V_{\mathbf{a}}$ are the gluing variables, or $\mathbf{a}$ variables, $V_{\mathbf{b}}$ are the $\mathbf{b}$ variables, and $K$ is the number of nodes in $w$. The $\mathbf{b}$ are away from $0,1,\infty$ and without loss of generality we will assume that all $\mathbf{a}$ variables are around $0$ (rather than possibly $1$ or $\infty$). The charts for $\overline{\mathcal{M}}_{0,k}$ with the new smooth structure $S^2_{new}$ are
\begin{align*}
\widetilde{\phi}:V &\rightarrow \overline{\mathcal{M}}_{0,k}\\
\widetilde{\phi} & = \phi\circ\Psi,
\end{align*}
where $\Psi$ is the reparametrization map
\begin{align*}
\Psi:V=V_{\mathbf{a}}\times V_{\mathbf{b}} &\rightarrow V=V_{\mathbf{a}}\times V_{\mathbf{b}}\\
(\widetilde{\mathbf{a}},\widetilde{\mathbf{b}}) &\mapsto \Psi(\widetilde{\mathbf{a}},\widetilde{\mathbf{b}}):= (\psi(\widetilde{\mathbf{a}}),\widetilde{\mathbf{b}})
\end{align*}
and $\psi$ is the reparametrization map from $(\ref{psireparam})$. We write $\psi(\widetilde{\mathbf{a}})$ to mean
\begin{equation}\label{tildenotation}
\psi(\widetilde{\mathbf{a}}) = \psi(\widetilde{a}_1,\ldots,\widetilde{a}_K) := (\psi(\widetilde{a}_1),\ldots,\psi(\widetilde{a}_K)).
\end{equation}

Thus, in these new local coordinates, the new variables $\widetilde{\mathbf{a}}$ are obtained from the old variables $\mathbf{a}$ by reparametrization: \[\psi(\widetilde{\mathbf{a}}) = \mathbf{a}.\]
The new variables $\widetilde{\mathbf{b}}$ are the same as the old variables $\mathbf{b}$. Similar to the map $\psi$ from $(\ref{psireparam})$, $\Psi$ is $C^1$ SS, that is it is $C^1$ everywhere and $C^\infty$ on the strata determined by the $\mathbf{a}$ variables.

To check that $\overline{\mathcal{M}}_{0,k}$ is a $C^1$ SS submanifold we need to show that the transition functions associated to the new charts are $C^1$ SS local diffeomorphisms. Let $\phi_1,\phi_2:V\rightarrow \overline{\mathcal{M}}_{0,k}$ be two charts for the usual smooth structure on $\overline{\mathcal{M}}_{0,k}$ and let $\widetilde{\phi}_1,\widetilde{\phi}_2$ be the corresponding new charts. Associated to $\phi_1,\phi_2$ is a transition map

\[T = \phi_2^{-1}\circ\phi_1:V\rightarrow V\]
\[T(\mathbf{a},\mathbf{b}) = \big(T_{A}(\mathbf{a},\mathbf{b}),T_{B}(\mathbf{a},\mathbf{b})\big).\]
Associated to $\widetilde{\phi}_1,\widetilde{\phi}_2$ is a transition map

\[\widetilde{T}:V \rightarrow V\]
\[\widetilde{T} :=  \widetilde{\phi}_2^{-1}\circ\widetilde{\phi}_1 = \Psi^{-1}\circ T \circ \Psi.\]
See Figure \ref{comm2} for a diagram of $T,\widetilde{T}$.
\begin{figure}
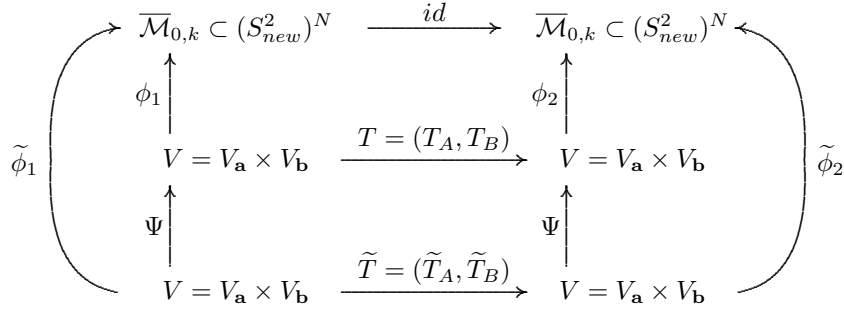

\begin{center}
\mbox{}
\begindc{\commdiag}[5]
\obj(0,20)[A]{$\overline{\mathcal{M}}_{0,k}\subset(S^2_{new})^N$}
\obj(30,20)[B]{$\overline{\mathcal{M}}_{0,k}\subset(S^2_{new})^N$}
\obj(0,10)[C]{$V=V_\mathbf{a}\times V_{\mathbf{b}}$}
\obj(30,10)[D]{$V=V_\mathbf{a}\times V_{\mathbf{b}}$}
\obj(0,0)[E]{$V=V_\mathbf{a}\times V_{\mathbf{b}}$}
\obj(30,0)[F]{$V=V_\mathbf{a}\times V_{\mathbf{b}}$}
\mor(8,20)(22,20){$id$}
\mor(6,10)(24,10){$T = (T_A,T_B)$}
\mor(6,0)(24,0){$\widetilde{T} = (\widetilde{T}_A,\widetilde{T}_B)$}
\mor(-5,10)(-5,20){$\phi_1$}
\mor(25,10)(25,20){$\phi_2$}
\mor(-5,0)(-5,10){$\Psi$}
\mor(25,0)(25,10){$\Psi$}
\cmor((-9,0)(-13,3)(-14,10)(-13,17)(-9,20)) \pright(-16,10){$\widetilde{\phi}_1$}
\cmor((38,0)(42,3)(43,10)(42,17)(38,20)) \pleft(45,10){$\widetilde{\phi}_2$}
\enddc
\end{center}
\caption{Change of coordinates for $\overline{\mathcal{M}}_{0,k}$}\label{comm2}
\end{figure}

We must show that $\widetilde{T}$ is $C^1$ SS, where $V \subset\mathbb{C}^{n-3}$ has stratification given by the $\mathbf{a}$ variables. The fact that $\widetilde{T}^{-1}$ is also $C^1$ SS follows because $\widetilde{T}^{-1}$ has the same form as $\widetilde{T}$. We first note that the transition map $T$ is $C^\infty$, so $T_A,T_B$ are $C^\infty$ functions of $\mathbf{a},\mathbf{b}$. We see that
\begin{align*}
\widetilde{T}(\widetilde{\mathbf{a}},\widetilde{\mathbf{b}}) &= \Psi^{-1}\circ T \circ \Psi(\widetilde{\mathbf{a}},\widetilde{\mathbf{b}})\\
                                                             &= \Psi^{-1}\circ T (\psi(\widetilde{\mathbf{a}}),\widetilde{\mathbf{b}})\\
                                                             &= \Psi^{-1}\Big(T_A\big(\psi(\widetilde{\mathbf{a}}),\widetilde{\mathbf{b}}\big),T_B\big(\psi(\widetilde{\mathbf{a}}),\widetilde{\mathbf{b}}\big)\Big)\\
                                                             &=\Big(\psi\Big(T_A\big(\psi(\widetilde{\mathbf{a}}),\widetilde{\mathbf{b}}\big)\Big),T_B\big(\psi(\widetilde{\mathbf{a}}),\widetilde{\mathbf{b}}\big)\Big).
\end{align*}
Note that $\widetilde{T}_B$, the second component of $\widetilde{T}$, is equal to $T_B(\psi(\widetilde{\mathbf{a}}),\widetilde{\mathbf{b}})$. The map $\widetilde{T}_B$ is $C^1$ SS because $T_B$ is smooth and $\psi$ is $C^1$ SS.

Since $\widetilde{T}_B$ is $C^1$ SS, it remains to argue that the first component $\widetilde{T}_A = \psi\big(T_A(\psi(\widetilde{\mathbf{a}}),\widetilde{\mathbf{b}})\big)$ is $C^1$ SS. In other words, that the new reparametrized gluing parameters are $C^1$ SS functions of $\widetilde{\mathbf{a}},\widetilde{\mathbf{b}}$. To do this we will now move to the specific case of $\overline{\mathcal{M}}_{0,5}$.

Example \ref{m05} describes coordinates on $\overline{\mathcal{M}}_{0,5}$. We will examine this case; this will suffice because transition functions for $\overline{\mathcal{M}}_{0,k}$ can be represented as a composition of transition functions on $\overline{\mathcal{M}}_{0,5}$. We will examine Configurations 1, 2, and 3 from Example \ref{m05}. For simplicity we will consider the case when $x,z$ give coordinates and look at coordinate changes from this. However, the same argument can be made for any choice of coordinates. In what follows we will use the same notation as (\ref{tildenotation}). This means that if $w$ is a coordinate, then $\widetilde{w}$ is the reparametrized coordinate. So if $w$ is a $\mathbf{b}$ variable, $\widetilde{w}=w$ and if $w$ is a $\mathbf{a}$ variable, $\psi(\widetilde{w}) = w$.\\

\noindent\underline{Configuration 1}: In the case of Configuration 1, $x,z$ give coordinates; $z$ is an $\mathbf{a}$ variable and $x$ is a $\mathbf{b}$ variable. Therefore,
\[\widetilde{x} = x,\qquad \psi(\widetilde{z}) = z.\]
For the other coordinates $\widetilde{y},\widetilde{u},\widetilde{v}$,
\[\widetilde{y} = y,\qquad \psi(\widetilde{u}) = u,\qquad \psi(\widetilde{v}) = v.\]
We want to show that $\widetilde{y},\widetilde{u},\widetilde{v}$ are $C^1$ SS functions of $\widetilde{x},\widetilde{z}$.

We have already argued above that $\widetilde{T}_B$ is $C^1$ SS and hence the $\mathbf{b}$ variable $\widetilde{y}$ is a $C^1$ SS function of $\widetilde{x}$ and $\widetilde{z}$.  This can also be seen explicitly using (\ref{m05y}):
\[\widetilde{y} = y = x - xz = \widetilde{x} - \widetilde{x}\psi(\widetilde{z}).\]
The map $\psi$ is $C^1$ SS so $\widetilde{y}$ is a $C^1$ SS function of $\widetilde{x},\widetilde{z}$.

For the $\mathbf{a}$ variable $\widetilde{u}$, we use (\ref{m05u}) to see
\begin{align}
\widetilde{u} &= \psi^{-1}(u)\nonumber\\
              &= \psi^{-1}\left(\frac{z}{zx-x+1}\right)\nonumber\\
              &= \psi^{-1}\left(\frac{\psi(\widetilde{z})}{\psi(\widetilde{z})\widetilde{x}-\widetilde{x}+1}\right)\nonumber\\
              &=\frac{z}{\psi^{-1}(\psi(\widetilde{z})\widetilde{x}-\widetilde{x}+1)},\label{utilde}
\end{align}
where the last equality uses (\ref{multprop}). The variable $\widetilde{z}$ is close to $0$ and $\widetilde{x}$ is away from $1$, so we conclude that $\psi(\widetilde{z})\widetilde{x}-\widetilde{x}+1$ is away from $0$ and hence the denominator of $(\ref{utilde})$ is $C^\infty$ because $\psi^{-1}$ is $C^\infty$ away from $0$.

Similarly, using (\ref{m05v}):
\begin{align}
\widetilde{v} &= \psi^{-1}(v)\nonumber\\
              &= \psi^{-1}\left(\frac{xz}{zx-x+1}\right)\nonumber\\
              &= \psi^{-1}\left(\frac{\widetilde{x}\psi(\widetilde{z})}{\psi(\widetilde{z})\widetilde{x}-\widetilde{x}+1}\right)\nonumber\\
              &=\frac{\psi^{-1}(\widetilde{x})z}{\psi^{-1}\big(\psi(\widetilde{z})\widetilde{x}-\widetilde{x}+1\big)}\label{vtilde}.
\end{align}
The denominator of (\ref{vtilde}) is $C^\infty$ for the same reason as $(\ref{utilde})$. The numerator is easily seen to be $C^\infty$ by noting that $\widetilde{x}$ is away from $0$. Therefore, we see $\widetilde{u},\widetilde{v}$ are $C^1$ SS functions of $\widetilde{x},\widetilde{z}$. In other words, in this case the transition function $\widetilde{T}_A$ is $C^1$ SS. This shows that the transition map $\widetilde{T}$ is $C^1$ SS.\\

\noindent\underline{Configuration 2}: In the case of Configuration 2, $x,z$ give coordinates; both $x,z$ are $\mathbf{a}$ variables so
\[\psi(\widetilde{x}) = x,\qquad \psi(\widetilde{z}) = z.\]
The variables $y,u$ are also $\mathbf{a}$ variables so
\[\psi(\widetilde{y}) = y,\qquad \psi(\widetilde{u}) = u.\]
Just as with Configuration 1, we use $(\ref{m05y})$ and $(\ref{m05u})$ to argue $\widetilde{y},\widetilde{u}$ are $C^1$ SS functions of $\widetilde{x},\widetilde{z}$. First,
\begin{align*}
\widetilde{y} &= \psi^{-1}(y)\\
              &= \psi^{-1}\big(x(1 - z)\big)\\
              &= \psi^{-1}\Big(\psi(\widetilde{x})\big(1 - \psi(\widetilde{z})\big)\Big)\\
              &= \widetilde{x}\Big(\psi^{-1}\big(1 - \psi(\widetilde{z})\big)\Big).
 \end{align*}
We see that $\widetilde{y}$ is a $C^1$ SS function of $\widetilde{x},\widetilde{z}$ because $\widetilde{z}$ is close to $0$ and hence $1-\psi(\widetilde{z})$ is away from zero.

Finally,
\begin{align*}
\widetilde{u} &= \psi^{-1}(u)\\
              &= \psi^{-1}\left(\frac{z}{zx-x+1}\right)\\
              &= \psi^{-1}\left(\frac{\psi(\widetilde{z})}{\psi(\widetilde{z})\psi(\widetilde{x})-\psi(\widetilde{x})+1}\right)\\
              &=\frac{z}{\psi^{-1}\big(\psi(\widetilde{z})\psi(\widetilde{x})-\psi(\widetilde{x})+1\big)}.
\end{align*}
Both $\widetilde{x},\widetilde{z}$ are close to $0$, so $\psi(\widetilde{z})\psi(\widetilde{x})-\psi(\widetilde{x})+1$ is away from $0$ and hence $\widetilde{u}$ is a $C^1$ SS function of $\widetilde{x},\widetilde{z}$. Therefore, we see $\widetilde{y},\widetilde{u}$ are $C^1$ SS functions of $\widetilde{x},\widetilde{z}$. This shows that the transition map $\widetilde{T}$ is $C^1$ SS.\\

\noindent\underline{Configuration 3}: All variables are $\mathbf{b}$ variables and we have previously argued that $\widetilde{T}_B$ is $C^1$ SS.\\

Completing all three configurations, we see that $\widetilde{T}$ is $C^1$ SS and hence $\overline{\mathcal{M}}_{0,k}$ forms a $C^1$ SS submanifold of $(S^2_{new})^N$.
\end{proof}
\begin{remark}\label{smoothforget}
It is trivial to note that for the new smooth structure $\overline{\mathcal{M}}_{k}^{new}$ of Proposition \ref{dmlemma}, the forgetful map
\[\overline{\mathcal{M}}_{k}^{new}\rightarrow\overline{\mathcal{M}}_{k-1}^{new}\]
is $C^1$ SS. Due to the particularly simple form of the reparametrization used to define the new smooth structure, this forgetful map is merely a $C^1$ SS projection. Although simple, this fact is crucial to the proof of the Gromov-Witten axioms in \cite{axiomme}. Other methods of changing the smooth structure on Deligne-Mumford space result in manifolds where the forgetful map is not smooth. For example, \cite[\S 2.3]{hwzdm} uses an exponential gluing profile to define the smooth structure (see Remark \ref{gluingprofile} for more about gluing profiles) and the forgetful map is not smooth in this case. See \cite{wg} for further discussion of the regularity of forgetful maps. The lack of differentiability of forgetful maps has obstructed the proof of the Gromov-Witten axioms using other virtual fundamental cycle methods.
\end{remark}

\section{Gromov-Witten charts}\label{gwcharts}
Let $(M^{2n},\omega)$ be a compact symplectic manifold and let $J$ be a tame almost complex structure. Let $X = \overline{\mathcal{M}}_{0,k}(A,J)$ be the compact space of nodal $J$-holomorphic genus zero stable maps in homology class $A$ with $k$ marked points modulo reparametrization. Let $d=2n+2c_1(A)+2k-6$. We will think of elements of $X$ as equivalence classes $[\Sigma,\mathbf{z},f]$ where $\Sigma$ is a genus zero nodal Riemann surface, $\mathbf{z}$ are $k$ disjoint marked points, and $f:\Sigma\rightarrow M$ is $J$-holomorphic map in homology class $A$.

Sections \ref{coordfree} and \ref{chartsincoords} give a brief description of the basic Gromov-Witten chart near a point $[\Sigma_0,\mathbf{z}_0,f_0]\in X$. \cite[$\S$ 4]{notes} gives a more detailed description of Gromov-Witten charts and verifies the necessary properties. The main result of this section is a proof of Theorem \ref{admitthm} which states that $X$ has a $d$-dimensional $C^1$ SS atlas. A precise description of this atlas is given by Theorem \ref{c1atlasthm}.

\subsection{Coordinate free definition of Gromov-Witten charts}\label{coordfree}
We will first describe the charts in a coordinate free manner. While it is relatively easy to describe the charts in this way, it is difficult to prove anything, so a description in coordinates will be given in Section \ref{chartsincoords}.\\

\noindent \textbf{Isotropy group}: The isotropy group $\Gamma$ is the isotropy of the element $[\Sigma_0,\mathbf{z}_0,f_0]$.\\

\noindent \textbf{Slicing manifold}: Next, we make an auxiliary choice of a \textbf{slicing manifold} $Q$ that is transverse to im$f_0$, disjoint from $f_0(\mathbf{z}_0)$, orientable, and so that the $k$ marked points $\mathbf{z}_0$ together with the $L$ marked points $\mathbf{w}_0:=f_0^{-1}(Q)$ stabilize the domain of $f_0$.\\

\noindent\textbf{Obstruction space}: Let
\begin{equation}\label{delta}
\Delta\subset\overline{\mathcal{M}}_{0,k+L}
\end{equation}
be a small neighborhood of $[\Sigma_0,\mathbf{w}_0,\mathbf{z}_0]\in\overline{\mathcal{M}}_{0,k+L}$ and $\mathcal{C}|_{\Delta}$ denote the universal curve over $\Delta$.

To define the obstruction space, we first choose a map $\lambda:E_0\rightarrow C^\infty(Hom_J^{0,1}(\mathcal{C}|_{\Delta}\times M))$. The space $E_0$ and the map $\lambda$ are chosen to ensure that the domains are cut out transversally near $f_0$ (see (\ref{domainhateqn}) and the discussion afterwards). Then the obstruction space $E$ is defined to be $E:=\Pi_{\gamma\in\Gamma}E_0$, the product of $|\Gamma|$ copies of $E_0$ with elements $\vec{e}:=(e^\gamma)_{\gamma\in\Gamma}$. The isotropy group $\Gamma$ acts on $E$ by permutation
\begin{equation}\label{actonobs}
(\alpha\cdot\vec{e})^{\gamma} = e^{\alpha\gamma}.
\end{equation}
The map $\lambda$ is then extended to be a $\Gamma$-equivariant map on $E$.\\

\noindent \textbf{Defining $\widehat{U}$}: Before defining the domain $U$, we define $\widehat{U}$. For $(0,[\Sigma_0,\mathbf{w}_0,\mathbf{z}_0,f_0])\in X$, $\widehat{U}$ is defined to be a neighborhood of $(0,[\Sigma_0,\mathbf{w}_0,\mathbf{z}_0,f_0])$ in the following space
\begin{equation}\label{domainhateqn}
\widehat{U}\subset\left\{(\vec{e},[\Sigma,\mathbf{w},\mathbf{z},f])~|~\vec{e}\in E,~ \overline{\partial}_Jf = \lambda(\vec{e})|_{\textnormal{graph}f}\right\}.
\end{equation}
Note that $\bar{\partial}_J$ gives a section of the bundle $C^\infty\big(Hom_J^{0,1}(\mathcal{C}|_{\Delta}\times M)\big)$. Standard Fredholm theory implies that while its linearization $d_f(\bar{\partial}_J)$ may not be surjective, it has a finite dimensional cokernel so that one can choose $\lambda(E_0)$ to surject onto it. The solution set $\widehat{U}$ in (\ref{domainhateqn}) is then a smooth manifold. See \cite{notes} for more details on how $E$ and $\lambda$ are chosen.\\

\noindent \textbf{Isotropy group action}: The group $\Gamma$ is the stabilizer of $[\Sigma_0,\mathbf{z}_0,f_0]$. Each $\gamma\in\Gamma$ is uniquely determined by how it permutes the extra marked points $\mathbf{w}_0$. That is, we consider $\Gamma$ to be a subgroup of the symmetric group $S_L$ and $\Gamma$ acts by
\[\mathbf{w}_0\mapsto \gamma\cdot\mathbf{w}_0:=(w^{\gamma(\ell)})_{1\leq\ell\leq L}.\]
There is also an associated action on the nodes $\mathbf{n}_0$.

We can extend this action to an action on $\Delta$ by
\[\delta=[\mathbf{n},\mathbf{w},\mathbf{z}]\mapsto\gamma^{*}(\delta)=[\gamma\cdot\mathbf{n},\gamma\cdot\mathbf{w},\mathbf{z}].\]
Finally, $\Gamma$ (partially) acts on $\widehat{U}$ by
\begin{equation}\label{coordfreeiso}
(\vec{e},[\mathbf{n},\mathbf{w},\mathbf{z},f])\mapsto(\gamma\cdot\vec{e},[\gamma\cdot\mathbf{n},\gamma\cdot\mathbf{w},\mathbf{z},f])
\end{equation}
where the (partially defined) action $\gamma\cdot\vec{e}$ is given by (\ref{actonobs}). It is proved in \cite{notes} that (\ref{coordfreeiso}) preserves solutions to (\ref{domainhateqn}).\\

\noindent \textbf{Domain}: The domain $U$ is obtained from $\widehat{U}$ by making $\widehat{U}$ $\Gamma$-invariant and then cutting down $\widehat{U}$ via the slicing manifold $Q$. By this we mean we may assume $\widehat{U}$ is $\Gamma$-invariant by taking an intersection of the finite orbit of $\widehat{U}$ under $\Gamma$. Then we cut down via the slicing manifold, which is to say that we make $U$ a $\Gamma$-invariant open neighborhood of $(0,[\Sigma_0,\mathbf{w}_0,\mathbf{z}_0,f_0])$ in the following space:
\begin{equation}\label{slice}
U\subset \left\{(\vec{e},[\Sigma,\mathbf{w},\mathbf{z},f])\in\widehat{U}~|~f(\mathbf{w})\subset Q\right\}.
\end{equation}
The subset $\widehat{U}$ will be of use because the analysis requires working in $\widehat{U}$. (See \cite[Remark 4.1.2]{notes} for more details on the necessity of working with $\widehat{U}$.)

The \textbf{section} and \textbf{footprint} maps are given by
\[s\left(\vec{e},[\Sigma,\mathbf{w},\mathbf{z},f]\right) = \vec{e},\qquad \psi\left(\vec{0},[\Sigma,\mathbf{w},\mathbf{z},f]\right) = [\Sigma,\mathbf{z},f].\]

The construction of the sum charts is very similar. For $I\subset\mathcal{I}_{\mathcal{K}}$, denote $\underline{\vec{e}}:=(\vec{e_i})_{i\in I}, \underline{\mathbf{w}}:=(\mathbf{w_i})_{i\in I}$. Then $U_I$ is chosen to be a $\Gamma_I$-invariant open set
\[U_I\subset \left\{(\underline{\vec{e}},[\Sigma,\underline{\mathbf{w}},\mathbf{z},f])~|~\underline{\vec{e}}\in E_I, f(\mathbf{w}_i)\subset Q_i, \overline{\partial}_Jf = \lambda(\underline{\vec{e}})|_{\textnormal{graph}f}\right\}\]
such that it has footprint $F_I=\bigcap_{i\in I}F_i$. As before, $U_I$ is obtained from a $\widehat{U}_I$ by cutting down via the slicing manifolds $\{Q_i\}_{i\in I}$.

In this coordinate free language, the coordinate changes $\mathbf{K}_I\rightarrow\mathbf{K}_J$ are given by choosing an appropriate domain $\widetilde{U}_{IJ}\subset U_J$ and then simply forgetting the components $(\mathbf{w_i})_{i\in(J\setminus I)}$ and the $(e_i)_{i\in (J\setminus I)}$ (which are 0 because $\widetilde{U}_{IJ}\subset s_J^{-1}(E_I)$).

\subsection{Charts in coordinates}\label{chartsincoords}
We will now describe a basic chart in coordinates and express the domain $U$ as the zero set of a transverse section. This is done by describing the domain $\Sigma$ in terms of gluing parameters. More precisely, the elements of the domain $U$ of a basic chart near $[\Sigma_0,\mathbf{z}_0,f_0]\in X$ are of the form $(\vec{e},\mathbf{a},\mathbf{b},\vec{\omega},\vec{\zeta},f)$ where:
\begin{enumerate}[(i)]
\item $\vec{e}\in E$.
\item Let $\mathbf{P}$ be a \textbf{normalization} of $[\Sigma_0,\mathbf{w}_0,\mathbf{z}_0]$, that is a choice of three special points on each component of $\Sigma_0$. A choice of $\mathbf{P}$ determines a parametrization of $\Sigma_0$, that is an identification of each component of $\Sigma_0$ with the standard $S^2$; denote this parametrization by $\Sigma_{\mathbf{P},0}$. Let $(\mathbf{w}_0)_{\mathbf{P}},(\mathbf{z}_0)_{\mathbf{P}},(\mathbf{n}_0)_{\mathbf{P}}$ denote the points of $\mathbf{w}_0,\mathbf{z}_0,$ and nodal points $\mathbf{n}_0$ that are normalized. Then $(\Sigma_{\mathbf{P},0},(\mathbf{w}_0)_{\mathbf{P}},(\mathbf{z}_0)_{\mathbf{P}})$ determines an element $\delta_{\mathbf{P},0}\in\overline{\mathcal{M}}_{0,p_z+p_w}$, where $p_z$ and $p_w$ are the number of points in $(\mathbf{z}_0)_{\mathbf{P}},(\mathbf{w}_0)_{\mathbf{P}}$.

    The parameters $\mathbf{a},\mathbf{b}$ are small and describe a neighborhood of $(\Sigma_{\mathbf{P},0}, (\mathbf{w}_0)_{\mathbf{P}}, (\mathbf{z}_0)_{\mathbf{P}})$ in the Deligne-Mumford space $\overline{\mathcal{M}}_{0,p_z+p_w}$. We will denote such a neighborhood by $\Delta_{\mathbf{P}}$. The parameters $\mathbf{a}$ are gluing parameters and the parameters $\mathbf{b}$ describe the position of the free nodes. Therefore, if $\Sigma_{\mathbf{P},0}$ has $K$ nodes and $\mathbf{P}$ fixes $p_n$ nodal points, then $\mathbf{a}=(a_1,\ldots,a_K)\in B^{2K}$ and $\mathbf{b}=(b_1,\ldots,b_{2K-p_n})\in B^{2(2K-p_n)}$, where $B^\ell$ denotes a small ball around $0$ in $\mathbb{R}^\ell$. For a description of these parameters $\mathbf{a},\mathbf{b}$ via cross ratios of the point $\mathbf{z}_{\mathbf{P}}$,$\mathbf{w}_{\mathbf{P}}$, see Section \ref{c1sscharts}.

    We will denote elements in $\overline{\mathcal{M}}_{0,p_z+p_w}$ nearby $\delta_{\mathbf{P},0}$ by $\delta_{\mathbf{P}}=[\mathbf{n},\mathbf{z}_{\mathbf{P}},\mathbf{w}_{\mathbf{P}}]$. An element $\delta_{\mathbf{P}}$ will have normalized domain $\Sigma_{\mathbf{P},\delta_{\mathbf{P}}}$ and will also be denoted $\Sigma_{\mathbf{P},\mathbf{a},\mathbf{b}}$ so
     \[\Sigma_{\mathbf{P},\delta_{\mathbf{P}}} = \Sigma_{\mathbf{P},\mathbf{a},\mathbf{b}}.\]
     Thus $\Sigma_{\mathbf{P},\delta_{\mathbf{P}}}$ is obtained from $\Sigma_{\mathbf{P},0}$ by gluing nodes and moving nonfixed nodal points in a manner that is precisely described by $\mathbf{a}$ and $\mathbf{b}$.
\item The parameters $\vec{\omega},\vec{\zeta}$ describe the positions of the nonfixed $\mathbf{w},\mathbf{z}$ respectively. Together $\mathbf{a},\mathbf{b},\vec{\omega},\vec{\zeta}$ determine a unique fibre in $\mathcal{C}|_{\Delta}$, where $\Delta$ is as in (\ref{delta}). The construction is so the map $\mathcal{C}|_{\Delta}\rightarrow\mathcal{C}|_{\Delta_{\mathbf{P}}}$ induced by the forgetful map does not collapse any components.

    We will denote stable curves with the full set of marked points by $\delta:=[\mathbf{n},\mathbf{z},\mathbf{w}]=[\Sigma_\delta,\mathbf{z},\mathbf{w}]$. If we parameterize using $\mathbf{P}$, then $\delta$ has normalized domain $\Sigma_{\mathbf{P},\delta_{\mathbf{P}}}$. Note that the stable curve $\delta = [\mathbf{n},\mathbf{w},\mathbf{z}]$ is actually determined by the tuple $[\mathbf{w},\mathbf{z}]$, but we include the nodal points in the notation for clarity because they appear in the normalization.
\item The map $f:\Sigma_{\mathbf{P},\mathbf{a},\mathbf{b}}\rightarrow M$ represents the homology class $A\in H_2(M)$ and is a solution to the equation
    \begin{equation}\label{solution}
    \overline{\partial}_J(f) = \lambda(\vec{e})|_{\textnormal{graph}f} := \sum_{\gamma\in\Gamma}\gamma^*(\lambda(e^\gamma))|_{\textnormal{graph}f},
    \end{equation}
    where $\gamma\in\Gamma$ acts by
    \begin{equation}\label{gammaactlambda}
    \gamma^{*}(\lambda(e^\gamma))|_{(z,f(z))}:=\lambda(e^\gamma)(\phi_{\gamma,\delta}^{-1}(z),f(z))\circ d_z\phi_{\gamma,\delta}^{-1}
    \end{equation}
    and $\phi_{\gamma,\delta}$ is described in (\ref{4121}) below.
\end{enumerate}

The solutions to this equation can be described as the zero set of the section
\begin{align}
F:E\times B^{2(L+k-p_w-p_z)}\times W^{1,p}(\mathcal{C}|_\Delta,M)&\rightarrow L^p\big(Hom_J^{0,1}(\Sigma_{\mathbf{P},\mathbf{a},\mathbf{b}}\times M)\big)\\
F(\vec{e},\mathbf{a},\mathbf{b}, \vec{\omega},\vec{\zeta},f)&=\overline{\partial}_Jf - \lambda(\vec{e})|_{\textnormal{graph}f}\nonumber.
\end{align}
Here $\mathcal{C}|_{\Delta}$ denotes the universal curve over a neighborhood $\Delta$ of $[\Sigma_{\mathbf{P},0},\mathbf{z}_0,\mathbf{w}_0]\in\overline{\mathcal{M}}_{0,k+L}$. It is described in \cite{notes} that this operator has a linearization $dF$ and an appropriate choice of obstruction space ensures that $dF$ is surjective at the center point $(0,0,0,0,0,f_0)$\footnote{There are additional conditions imposed on E. See \cite[$\S 4$]{notes}, in particular condition (*).}. We then define $\widehat{U}$ to be the space of solutions near the center point $(0,0,0,0,0,f_0)$. The domain $U$ of a basic chart about $[\Sigma_0,\mathbf{z}_0,f_0]$ is obtained from $\widehat{U}$ by further cutting down by the slicing manifold $Q$ (as in (\ref{slice})) and is $\Gamma$-invariant. Exactly what type of (smooth) structure $\widehat{U}$ carries depends on the gluing theorem used. In Section \ref{c1gluing}, we prove a ``$C^1$ gluing theorem'', Theorem \ref{mainthm}, which we will show give Gromov-Witten charts a \textit{$C^1$ stratified smooth structure}. This is a structure defined in generality in Section \ref{c1ss}. Section \ref{c1sscharts} will then apply this theory to Gromov-Witten charts.

We can also give explicit formulas for the action of $\Gamma$ and coordinate changes in these coordinates.\\

\noindent \textbf{Isotropy group action:} We will now describe the isotropy group action (\ref{coordfreeiso}) in coordinates. The map $f$ is defined on the normalized domain $\Sigma_{\mathbf{P},\delta}$. Note that $\Sigma_{\mathbf{P},\delta}\neq\Sigma_{\mathbf{P},\gamma^{*}(\delta)}$ because in $\Sigma_{\mathbf{P},\gamma^{*}(\delta)}$ the points whose \textit{new} labels are in $\mathbf{P}$ are normalized. Therefore, the normalized action can be written as
\[(\Sigma_{\mathbf{P},\delta},\mathbf{w},\mathbf{z},f)\mapsto(\Sigma_{\mathbf{P},\gamma^{*}(\delta)},\phi_{\gamma,\delta}^{-1}(\gamma\cdot\mathbf{w}),\phi_{\gamma,\delta}^{-1}(\mathbf{z}),f\circ\phi_{\gamma,\delta})\]
where
\begin{equation}\label{4121}
\phi_{\gamma,\delta}:\Sigma_{\mathbf{P},\gamma^{*}(\delta)}\rightarrow\Sigma_{\mathbf{P},\delta}
\end{equation}
is the unique biholomorphic map that takes the special points $\gamma\cdot\mathbf{n},\gamma\cdot\mathbf{w},\mathbf{z}$ in $\Sigma_{\mathbf{P},\gamma^{*}(\delta)}$ with normalization $\mathbf{P}$ to the corresponding points in $\Sigma_\delta$. Note that $\phi_{\gamma,\delta}$ does not depend on the map $f$ but rather only on the parameters $\mathbf{a},\mathbf{b}$ and $\vec{\omega},\vec{\zeta}$.

These maps $\phi_{\gamma,\delta}$ pull back to partially defined maps $\phi_{\mathbf{P},\gamma,\delta}$ on a fixed surface as follows. The gluing parameters $\mathbf{a},\mathbf{b}$ determine a fibrewise embedding
\begin{align}\label{iota}
\iota_\mathbf{P} &: (\Sigma_{\mathbf{P},0}\setminus \mathcal{N}(nodes)) \times B^{2K}\times B^{4K-2p_n}  \rightarrow \mathcal{C}|_{\Delta_\mathbf{P}}\\
\iota_{\mathbf{P},\mathbf{a},\mathbf{b}} &: (\Sigma_{\mathbf{P},0}\setminus \mathcal{N}(nodes))  \times \{\mathbf{a},\mathbf{b}\} \mapsto \Sigma_{\mathbf{P},\mathbf{a},\mathbf{b}}\setminus\mathcal{N}(nodes)\nonumber
\end{align}
where $\mathcal{N}(nodes)$ denotes some neighborhood of the nodes, varying upon circumstance. So $\iota_{\mathbf{P}}$ identifies $\Sigma_{\mathbf{P},\mathbf{a},\mathbf{b}}\setminus\mathcal{N}(nodes)$ with a \textit{fixed} open subset of $\Sigma_{\mathbf{P},0}$. Then we can define
\[\phi_{\mathbf{P},\gamma,\delta}:(\Sigma_{\mathbf{P},0}\setminus\mathcal{N}(nodes))\rightarrow(\Sigma_{\mathbf{P},0}\setminus\mathcal{N}(nodes))\]
to be the composition
\begin{equation}\label{phip}
\phi_{\mathbf{P},\gamma,\delta}=\iota^{-1}_{\mathbf{P},\mathbf{a},\mathbf{b}}\circ\phi_{\gamma,\delta}\circ\iota_{\mathbf{P},\mathbf{a}',\mathbf{b}'}
\end{equation}
\[\Sigma_{\mathbf{P},0}\setminus\mathcal{N}(nodes)\xrightarrow{\iota_{\mathbf{P},\mathbf{a}',\mathbf{b}'}}\Sigma_{\mathbf{P},\mathbf{a}',\mathbf{b}'}=\Sigma_{\mathbf{P},\gamma^{*},\delta}\xrightarrow{\phi_{\gamma,\delta}}\Sigma_{\mathbf{P},\mathbf{a},\mathbf{b}}=\Sigma_{\mathbf{P},\gamma}\xrightarrow{\iota^{-1}_{\mathbf{P},\mathbf{a},\mathbf{b}}}\Sigma_{\mathbf{P},0}\setminus\mathcal{N}(nodes).\]
We then define the action of $\alpha\in\Gamma$ on $\widehat{U}$ by
\begin{equation}\label{isotopyfor}
\alpha^{*}(\vec{e},\mathbf{a},\mathbf{b},\vec{\omega},\vec{\zeta},f) = (\alpha\cdot\vec{e},\mathbf{a}',\mathbf{b}',\phi^{-1}_{\mathbf{P},\alpha,\delta}(\alpha\cdot\vec{\omega}),\phi^{-1}_{\mathbf{P},\alpha,\delta}(\vec{\zeta}),f\circ\phi_{\alpha,\delta})
\end{equation}
where
\begin{align*}
\delta=[\mathbf{n},\mathbf{w},\mathbf{z}]\in\overline{\mathcal{M}}_{0,k},\\
\Sigma_{\mathbf{P},\delta}:=\Sigma_{\mathbf{P},\mathbf{a},\mathbf{b}},\\
\Sigma_{\mathbf{P},\alpha^{*}(\delta)}=\Sigma_{\mathbf{P},\mathbf{a}',\mathbf{b}'},\\
\phi_{\alpha,\delta}:\Sigma_{\mathbf{P},\mathbf{a}',\mathbf{b}'}\rightarrow\Sigma_{\mathbf{P},\mathbf{a},\mathbf{b}},
\end{align*}
and $\phi_{\alpha,\delta}$ is as in (\ref{4121}). Note that (\ref{isotopyfor}) is well defined because $\vec{\omega}$ consists of points away from the nodes. It is shown in \cite{notes} that $\alpha^{*}(\vec{e},\mathbf{a},\mathbf{b},\vec{\omega},\vec{\zeta},f)$ is a solution to (\ref{solution}) and preserves the slicing conditions as required.\\

\noindent\textbf{Coordinate changes}: We next describe the effect on a single chart of changing the normalization, center, and slicing conditions. These choices were required in the construction of a basic chart. These formulas will be needed to understand the coordinate changes of an atlas. We will construct the sum of two charts in coordinates and describe the coordinate changes. The general construction of sum charts and coordinate changes is argued in the same way. For details of this, refer to \cite{notes}.\\

\noindent\textbf{Change of normalization:} Let $\mathbf{P}_1,\mathbf{P}_2$ be two normalizations. Each determine a domain $U_{\mathbf{P}_i}$. Let $\phi_{\mathbf{P}_2,\mathbf{P}_1}:\Sigma_{\mathbf{P}_2,0}\rightarrow\Sigma_{\mathbf{P}_1,0}$ be the unique biholomorphism that takes the points $\mathbf{n}_0,\mathbf{w}_0,\mathbf{z}_0$ in $\Sigma_{\mathbf{P}_2}$ with labels determined by $\mathbf{P}_1$ to the positions fixed by $\mathbf{P}_1$ in $\Sigma_{\mathbf{P}_1,0}$. Then for each $\delta\in\Delta$, the fibre $\Sigma_\delta$ has two normalizations $\Sigma_{\mathbf{P},\delta} = \Sigma_{\mathbf{P}_i,\mathbf{a}_i,\mathbf{b}_i},i=0,1$. Let $\iota_{\mathbf{P}_i}$ be the maps from $(\ref{iota})$. There is a unique biholomorphism
\[\phi_{\mathbf{P}_2,\mathbf{P}_1,\delta}:\Sigma_{\mathbf{P}_2,\mathbf{a}_2,\mathbf{b}_2}\rightarrow\Sigma_{\mathbf{P}_1,\mathbf{a}_1,\mathbf{b}_1}\]
defined by the equation
\[\phi_{\mathbf{P}_2,\mathbf{P}_1,\delta}:=\iota_{\mathbf{P}_1,\mathbf{a}_1,\mathbf{b}_1}\circ\phi_{\mathbf{P}_2,\mathbf{P}_1}\circ\iota^{-1}_{\mathbf{P}_2,\mathbf{a}_2,\mathbf{b}_2}\]
wherever the right side is defined. Then the change of normalization is given by $\phi_{\mathbf{P}_2,\mathbf{P}_1,\delta}$. Note that $\phi_{\mathbf{P}_2,\mathbf{P}_1,\delta}$ acts on $\vec{\omega},\vec{\zeta}$ because these points lie away from the nodes.

Thus, there is a corresponding change of normalization on the domains $U_{\mathbf{P}_1}\rightarrow U_{\mathbf{P}_2}$ given by
\begin{align}
(\vec{e},\mathbf{a}_1,\mathbf{b}_1,\vec{\omega},\vec{\zeta},f) &\mapsto \phi_{\mathbf{P}_2,\mathbf{P}_1}^{*}(\vec{e},\mathbf{a}_1,\mathbf{b}_1,\vec{\omega},\vec{\zeta},f)\nonumber\\
&:= (\vec{e},\mathbf{a}_2,\mathbf{b}_2,\phi_{\mathbf{P}_2,\mathbf{P}_1}^{-1}(\vec{\omega}),\phi_{\mathbf{P}_2,\mathbf{P}_1}^{-1}(\vec{\zeta}),f\circ\phi_{\mathbf{P}_2,\mathbf{P}_1,\delta}).
\end{align}

\noindent\textbf{Change of center:} Suppose we are given a chart $U_{\mathbf{P}_1}$ constructed using center $\tau_1:=[\Sigma_{01},\mathbf{z}_{01},f_{01}]$, slicing manifold $Q$, and normalization $\mathbf{P}_1$. We will examine the change from center $\tau_1$ to center $\tau_2:=[\Sigma_{02},\mathbf{z}_{02},f_{02}]$, but keeping the same slicing manifold and the normalization $\mathbf{P}_2$ that includes all the nodal points of $\mathbf{P}_1$ that have not been glued together with an appropriate subset of the marked points $\mathbf{w},\mathbf{z}$ fixed by $\mathbf{P}_1$. Denote $\delta_{01} := [\Sigma_{01}, \mathbf{w}_{01}, \mathbf{z}_{01}]$ and similarly $\delta_{02}$.

If $\tau_2$ contains fewer nodes than $\tau_1$, then there is no hope that whole footprint $F$ be covered by a chart based at $\tau_2$. However, in \cite{notes} McDuff describes a $\Gamma$-invariant neighborhood $U_{\mathbf{P}_1|_{\Delta_2}}$ of $\psi^{-1}(F\cap X_{\geq \tau_2})$ in $U_{\mathbf{P}_1}$ that can be represented in the normalization $\mathbf{P}_2$. Here $X_{\geq \tau_2}$ is the set of elements of $X$ containing at least as many nodes as $\tau_2$. Let $\mathbf{a}_i,\mathbf{b}_i$ be parameters and $\iota_{\mathbf{P}_i}$ be maps as in (\ref{iota}) so that the composite $\iota_{\mathbf{P}_2}^{-1}\circ\iota_{\mathbf{P}_1}|_{\Delta_2}$ has the form
\[\iota_{\mathbf{P}_2^{-1}}\circ\iota_{\mathbf{P}_1|_{\Delta_2}}:(\mathbf{a}_1,\mathbf{b}_1,\vec{\omega_1},\vec{\zeta_1})\mapsto(\mathbf{a}_2,\mathbf{b}_2,\vec{\omega_2},\vec{\zeta_2}).\]
Define
\[\phi_\delta:\Sigma_{(\mathbf{P}_2,\delta_{02}),\delta}\rightarrow\Sigma_{(\mathbf{P}_1,\delta_{01}),\delta}\]
to be the biholomorphic map that equals $\iota_{\mathbf{P}_1,\mathbf{a}_1,\mathbf{b}_1}\circ(\iota_{\mathbf{P}_2}\circ\iota_{\mathbf{P}_1|_{\Delta_2}})^{-1}\circ\iota_{\mathbf{P}_2,\mathbf{a}_2,\mathbf{b}_2}^{-1}$ wherever this is defined. Then it is shown in \cite{notes} that change of center on the domains $U_{\mathbf{P}_1|_{\Delta_2}}\rightarrow U_{\mathbf{P}_2}$ is given by
\begin{equation}
(\vec{e},\mathbf{a}_1,\mathbf{b}_1,\vec{\omega_1},\vec{\zeta_1},f)\mapsto (\vec{e},\mathbf{a}_2,\mathbf{b}_2,\vec{\omega_2},\vec{\zeta_2},f\circ\phi_\delta).
\end{equation}

\noindent\textbf{Construction of sum charts:} We now describe the construction of the sum of two charts following \cite{notes}. Suppose that we are given two charts $\mathbf{K}_i$ with overlapping footprints $F_i$. For simplicity we will assume that the center of $\mathbf{K}_2$ is contained in $F_1$. For the general case see \cite[Remark 4.1.5]{notes}. We will construct the sum chart $\mathbf{K}_{12}$ with footprint $F_{12}=F_1\cap F_2$, obstruction space $E_{12}:=E_1\times E_2$, and isotropy group $\Gamma_{12}:=\Gamma_1\times \Gamma_2$. We will construct this sum chart using the coordinates determined by the center of $\mathbf{K}_2$ and the normalization $\mathbf{P}_2$. Similar to the construction of a basic chart, define $\mathcal{W}_{12,\mathbf{P}_2}$ to be tuples $(\vec{e_1},\vec{e_2},\mathbf{a},\mathbf{b},\vec{\omega_1},\vec{\omega_2},\vec{\zeta},f)$ where:
\begin{enumerate}[(i)]
\item $\vec{e_i}\in E_i$.
\item $\mathbf{a}\in B^{2K},\mathbf{b}\in B^{2(2K-p_n)}$.
\item $f:\Sigma_{\mathbf{P}_2,\mathbf{a},\mathbf{b}}\rightarrow M$.
\end{enumerate}
Then define $\widehat{U}_{12,\mathbf{P}_2}$ to a suitable open subset of solution space
\begin{align*}
\widehat{U}_{12,\mathbf{P}_2}\subset &\Big\{(\vec{e_1},\vec{e_2},\mathbf{a},\mathbf{b},\vec{\omega_1},\vec{\omega_2},\vec{\zeta},f)\in\mathcal{W}_{12,\mathbf{P}_2} ~|~  \\
&\bar{\partial}_J(\vec{e_1},\vec{e_2},\mathbf{a},\mathbf{b},\vec{\omega_1},\vec{\omega_2},\vec{\zeta},f) = \sum_{i=1,2}\sum_{\gamma\in \Gamma_i}\gamma^{*}(\lambda_i(e_i^\gamma))|_{\textnormal{graph }f}\Big\}
\end{align*}
with the action $\gamma^{*}$ given by (\ref{gammaactlambda}) and $\phi_{\gamma,\delta_i}$ is as in (\ref{4121}) for
\[\delta_i:=[\Sigma_{\mathbf{P}_2,\mathbf{a},\mathbf{b}},\mathbf{w}_i,\mathbf{z}_i]\]
with $\mathbf{w}_i = \iota_{\mathbf{P}_2,\mathbf{a},\mathbf{b}}(\vec{\omega_i})$ and $\mathbf{z}_i = \iota_{\mathbf{P}_2,\mathbf{a},\mathbf{b}}(\vec{\zeta_i})$.
Elements $\gamma_1\in\Gamma_1$ act on elements of $\widehat{U}_{12,\mathbf{P}_2}$ by simple permutation:
\[\gamma_1^{*}(\vec{e_1},\vec{e_2},\mathbf{a},\mathbf{b},\vec{\omega_1},\vec{\omega_2},\vec{\zeta},f) = (\gamma_1\cdot \vec{e_1},\vec{e_2},\mathbf{a},\mathbf{b},\gamma\cdot\vec{\omega_1},\vec{\omega_2},f).\]
On the other hand, elements $\gamma_2\in\Gamma_2$ act by permutation and renormalization:
\[\gamma_2^{*}(\vec{e_1},\vec{e_2},\mathbf{a},\mathbf{b},\vec{\omega_1},\vec{\omega_2},\vec{\zeta},f) = (\vec{e_1},\gamma_2\cdot\vec{e_2},\mathbf{a}',\mathbf{b}',\phi_{\gamma_2}^{-1}(\vec{\omega_1}),\phi_{\gamma_2}^{-1}(\gamma_2\cdot\omega_2),\phi_{\gamma_2}^{-1}\vec{\zeta},f\circ\phi_{\gamma_2,\delta_2})\]
where $\phi_{\gamma_2}:=\phi_{\mathbf{P}_2,\gamma_2,\delta_2}$ as in (\ref{phip}).

Finally the domain $U_{12,\mathbf{P}_2}$ is obtained from $\widehat{U}_{12,\mathbf{P}_2}$ by cutting down by slicing conditions. To complete the discussion of the sum chart $\mathbf{K}_{12}$ we describe the coordinate changes $\mathbf{K}_i\rightarrow\mathbf{K}_{12}$. The coordinate change $\mathbf{K}_2\rightarrow\mathbf{K}_{12}$ is induced by the projection
\begin{align*}
\rho_{2,12}:U_{12,\mathbf{P}_2}\cap s_{12}^{-1}(E_1)&\rightarrow U_2\\
(\vec{0},\vec{e_2},\mathbf{a},\mathbf{b},\vec{\omega_1},\vec{\omega_2},\vec{\zeta},f)&\mapsto(\vec{e_2},\mathbf{a},\mathbf{b},\vec{\omega_2},\vec{\zeta},f).
\end{align*}
The coordinate change $\mathbf{K}_1\rightarrow\mathbf{K}_{12}$ has domain $\widetilde{U}_{1,12}:=U_{12,\mathbf{P}_2}\cap s_{12}^{-1}(E_1)$ (that is $\widetilde{U}_{1,12}$ has $\vec{e}_1=0$) and is given by first changing the normalization from $\mathbf{P}_2$ to $\mathbf{P}_1$ and then forgetting the components of $\vec{\omega_2}$ to obtain a map $\rho_{1,12}:\widetilde{U}_{1,12}\rightarrow U_{12,\mathbf{P}_1}$.

\subsection{$C^1$ SS Gromov-Witten charts}\label{c1sscharts}
This section will describe how a $C^1$ gluing theorem gives the Gromov-Witten charts described in Sections \ref{coordfree} and \ref{chartsincoords} the structure of a $C^1$ SS Kuranishi atlas.

We will now slightly modify the Gromov-Witten charts $U$ of Section \ref{chartsincoords} by reparametrization in order to achieve $C^1$-differentiability. We will first work with $\widehat{U}$, which is the open set from which we obtain the domain $U$ by making $\widehat{U}$ $\Gamma$-invariant and cutting down by slicing conditions (as in (\ref{slice})).

As in Section \ref{chartsincoords},
\[\widehat{U} \subset E\times B^{2(L+k-p_w-p_z)}\times W^{1,p}(\mathcal{C}|_{\Delta},M)\]
and consists of tuples $(\vec{e},\mathbf{a},\mathbf{b},\vec{\omega},\vec{\zeta},f)$ with $\vec{e}\in E$, $\mathbf{a}$ are gluing parameters, $\mathbf{b}$ are parameters describing the position of free nodes, $\vec{\omega},\vec{\zeta}$ describe the positions of nonfixed $\mathbf{w},\mathbf{z}$, and $f$ is  a map
\begin{equation}\label{reparam}
f:\Sigma_{\mathbf{P},\mathbf{a},\mathbf{b}}\rightarrow M
\end{equation}
that is a solution of $(\ref{solution})$. Thus, $\widehat{U}$ is the zero set of the section
\begin{align}
F:E\times B^{2(L+k-p_w-p_z)}\times W^{1,p}(\mathcal{C}|_\Delta,M)&\rightarrow L^p\big(Hom_J^{0,1}(\Sigma_{\mathbf{P},\mathbf{a},\mathbf{b}}\times M)\big)\\
F(\vec{e},\mathbf{a},\mathbf{b}, \vec{\omega},\vec{\zeta},f)&=\overline{\partial}_Jf - \lambda(\vec{e})|_{\textnormal{graph}f}\nonumber.
\end{align}
By choosing an appropriate obstruction space $E$, as described in Section \ref{coordfree}, we can assume that $dF$ is surjective at $(0,0,0,0,0,f_0)$. Thus, the space of solutions at $\mathbf{a}=\mathbf{b}=0$ is a smooth manifold. Let $W$ denote a precompact open subset around $f_0$ in this manifold (this can be done, for example, by considering maps with a fixed bound on the $L^\infty$ norm of their first derivative; c.f. the beginning of Section \ref{c1gluing}).

We will give $\widehat{U}$ a smooth structure via reparametrized gluing maps. That is, charts from Euclidean space to $\widehat{U}$ are given by
\begin{align}
\nonumber\Phi:E\times B^{2K}\times B^{2(2K-p_n)} &\times B^{2(L-p_w)} \times B^{2(k-p_z)} \times W \rightarrow \widehat{U}\\
(\vec{e},\widetilde{\mathbf{a}},\mathbf{b}, \vec{\omega},\vec{\zeta},\bar{f}) &\mapsto \big(\vec{e},\mathbf{a},\mathbf{b}, \vec{\omega},\vec{\zeta},gl_{\mathbf{a},\mathbf{b}}(\bar{f})\big)\\
\nonumber &=\Big(\vec{e},\psi(\widetilde{\mathbf{a}}),\mathbf{b}, \vec{\omega},\vec{\zeta},gl_{\psi(\widetilde{\mathbf{a}}),\mathbf{b}}(\bar{f})\Big),
\end{align}
where $W$ is the manifold of (nodal) solutions to $(\ref{solution})$ at $\mathbf{a}=\mathbf{b}=0$, $gl_{\mathbf{a},\mathbf{b}}(\bar{f})$ denotes the gluing of nodal map $\bar{f}$ with parameters $\mathbf{a},\mathbf{b}$ and $\psi$ is the reparametrization map from $(\ref{psireparam}),(\ref{tildenotation})$, and we use the notation of Section \ref{smoothstructuresection} so that
 \[\psi(\mathbf{\widetilde{a}}) = \mathbf{a}.\]
We use the classical gluing map of \cite{JHOL}; this map is described in detail in Section \ref{c1gluing}. The main theorem of this section is Theorem \ref{c1atlasthm}, which states that with these reparametrized charts, $\widehat{U}$ is a $C^1$ SS manifold with a $C^1$ SS action of $\Gamma$ and the spaces $U$ are domains for a $C^1$ SS Kuranishi atlas. This proves Theorem \ref{admitthm}. We will use a $C^1$ gluing theorem to prove Theorem \ref{c1atlasthm}. The specific version we will use is stated as Proposition \ref{corc1} below.

\begin{theorem}\label{c1atlasthm}
Let $(M^{2n},\omega,J)$ be a $2n$-dimensional symplectic manifold with a tame almost complex structure $J$. Let $X=\overline{\mathcal{M}}_{0,k}(A,J)$ be the compact space of nodal $J$-holomorphic genus zero stable maps in class $A \in H_2(M)$ with $k$ marked points modulo reparametrization. Let $[f_0]\in X$ and $\widehat{U}$ be the set around $f_0$ as described in Section \ref{chartsincoords}. So $\widehat{U}$ is the zero set of the section
\begin{align}\label{foperator}
F:E\times B^{2(L+k-p_w-p_z)}\times W^{1,p}(\mathcal{C}|_\Delta,M)&\rightarrow L^p\big(Hom_J^{0,1}(\Sigma_{\mathbf{P},\mathbf{a},\mathbf{b}}\times M)\big)\\
F(\vec{e},\mathbf{a},\mathbf{b}, \vec{\omega},\vec{\zeta},f)&=\overline{\partial}_Jf - \lambda(\vec{e})|_{\textnormal{graph}f}\nonumber.
\end{align}
Let $E$ be the obstruction space as described in Section \ref{coordfree}. Then
\begin{enumerate}[(i)]
\item\label{c11} $\widehat{U}$ is a $C^1$ SS manifold with stratification given by the $\mathbf{a}$ variables and $C^1$ SS structure given by charts
\begin{align}
\nonumber\Phi:E\times B^{2K}\times B^{2(2K-p_n)} &\times B^{2(L-p_w)} \times B^{2(k-p_z)} \times W \rightarrow \widehat{U}\\
\label{gluingcharts}(\vec{e},\widetilde{\mathbf{a}},\mathbf{b}, \vec{\omega},\vec{\zeta},\bar{f}) &\mapsto \big(\vec{e},\mathbf{a},\mathbf{b}, \vec{\omega},\vec{\zeta},gl_{\mathbf{a},\mathbf{b}}(\bar{f})\big)\\
\nonumber &=\Big(\vec{e},\psi(\widetilde{\mathbf{a}}),\mathbf{b}, \vec{\omega},\vec{\zeta},gl_{\psi(\widetilde{\mathbf{a}}),\mathbf{b}}(\bar{f})\Big),
\end{align}
where $gl_{\mathbf{a},\mathbf{b}}(\bar{f})$ denotes the gluing of nodal map $\bar{f}$ with parameters $\mathbf{a},\mathbf{b}$ and $W$ is a precompact open subset of the smooth manifold of $W^{1,p}$ solutions to $(\ref{solution})$ for $\mathbf{a}=\mathbf{b}=0$ containing $f_0$.
\item\label{c12} The isotropy group action on $\widehat{U}$, as described in Section \ref{chartsincoords}, is $C^1$ SS.
\item\label{c13} The marked points $\vec{\omega}$ given by the slicing conditions for the domain $U$ are $C^1$ SS in the sense that the map
    \begin{align*}
    (\widetilde{\mathbf{a}},\mathbf{b},f)\mapsto \vec{\omega}\in B^{2(L-p_w)}
    \end{align*}
    is $C^1$ SS.
\item\label{c14} The domain $U$ is a $C^1$ SS submanifold of $\widehat{U}$, $\Gamma$ acts $C^1$ SS on $U$, and coordinate changes are $C^1$ SS maps. Hence, these are charts for a $C^1$ SS Kuranishi atlas $\mathcal{K}$ on the Gromov-Witten moduli space $X$.
\end{enumerate}
\end{theorem}

\begin{proof}[Proof of Theorem \ref{c1atlasthm}]
We must check several things to prove this theorem. We first show that $\widehat{U}$ is a $C^1$ SS manifold with stratification given by the $\mathbf{a}$ variables. We then show that the isotropy group action of $\Gamma$ on $\widehat{U}$ is $C^1$ SS. Then we must prove that $U$ is a $C^1$ SS submanifold of $\widehat{U}$ with a $C^1$ SS action of $\Gamma$. Finally, we show that the coordinate change maps of the atlas $\mathcal{K}$ are $C^1$ SS. These facts together imply that $\mathcal{K}$ is a $C^1$ SS atlas.\\

\noindent\textbf{Proof of (\ref{c11})--Smoothness of $\widehat{U}$}: The space $\widehat{U}$ is defined to be the zero set of the operator $F$ from (\ref{foperator}). As earlier, the linearization $dF$ is surjective at $(0,0,0,0,0,f_0)$. The gluing theorem of \cite{JHOL} implies that $F^{-1}(0)$ is a product of the space of solutions at $\mathbf{a}=\mathbf{b}=0$ with a small neighborhood of $0$ in the parameter space $(\mathbf{a},\mathbf{b})$. Finally, $\widehat{U}$ inherits a $C^1$ SS structure via the charts (\ref{gluingcharts}). This proves (\ref{c11}).\\

\noindent\textbf{Proof of (\ref{c12})--Isotropy group action}: Section \ref{chartsincoords} gives an explicit description of the isotropy group action $\Gamma$ in coordinates. In particular, (\ref{isotopyfor}) describes that action of $\alpha\in\Gamma$ to be
\[\alpha^{*}(\vec{e},\mathbf{a},\mathbf{b},\vec{\omega},\vec{\zeta},f) = (\alpha\cdot\vec{e},\mathbf{a}',\mathbf{b}',\phi^{-1}_{\mathbf{P},\alpha,\delta}(\alpha\cdot\vec{\omega}),\phi^{-1}_{\mathbf{P},\alpha,\delta}(\vec{\zeta}),f\circ\phi_{\alpha,\delta})\]
where \begin{align*}
\delta=[\mathbf{n},\mathbf{w},\mathbf{z}]\in\overline{\mathcal{M}}_{0,k},\\
\Sigma_{\mathbf{P},\delta}:=\Sigma_{\mathbf{P},\mathbf{a},\mathbf{b}},\\
\Sigma_{\mathbf{P},\alpha^{*}(\delta)}=\Sigma_{\mathbf{P},\mathbf{a}',\mathbf{b}'},\\
\phi_{\alpha,\delta}:\Sigma_{\mathbf{P},\mathbf{a}',\mathbf{b}'}\rightarrow\Sigma_{\mathbf{P},\mathbf{a},\mathbf{b}},
\end{align*}
and $\phi_{\alpha,\delta}$ is as in (\ref{4121}). In the local coordinates determined by the charts (\ref{gluingcharts}), the isotropy action is
\[(\vec{e},\widetilde{\mathbf{a}},\mathbf{b},\vec{\omega},\vec{\zeta},\bar{f}) \mapsto (\alpha\cdot\vec{e},\widetilde{\mathbf{a}}',\mathbf{b}',\phi^{-1}_{\mathbf{P},\alpha,\delta}(\alpha\cdot\vec{\omega}),\phi^{-1}_{\mathbf{P},\alpha,\delta}(\vec{\zeta}),\bar{f}')\]
where $gl_{\psi(\widetilde{\mathbf{a}}),\mathbf{b}}(\bar{f})=f, gl_{\psi(\widetilde{\mathbf{a}}'),\mathbf{b}'}(\bar{f}')=f\circ\phi_{\alpha,\delta}.$ (See Figure \ref{uhattrans}.)

The maps $\vec{e}\mapsto\gamma\cdot\vec{e}$ are clearly smooth. Recall that the map $\phi_{\mathbf{P},\alpha,\delta}$ does not depend on the map $f$. We next notice that the map
\[(\vec{e},\widetilde{\mathbf{a}},\mathbf{b},\vec{\omega},\vec{\zeta}) \mapsto (\alpha\cdot\vec{e},\widetilde{\mathbf{a}}',\mathbf{b}',\phi^{-1}_{\mathbf{P},\alpha,\delta}(\alpha\cdot\vec{\omega}),\phi^{-1}_{\mathbf{P},\alpha,\delta}(\vec{\zeta}))\]
is exactly the form of a transition function on $\overline{\mathcal{M}}_{0,k+L}^{new}$ as described in Section \ref{smoothstructuresection}. These maps were shown to be $C^1$ SS.

\begin{figure}
\begin{center}
\mbox{}
\begindc{\commdiag}[5]
\obj(0,20)[A]{$\widehat{U}\ni\big(\vec{e},\mathbf{a},\mathbf{b},\vec{\omega},\vec{\zeta},gl_{\mathbf{a},\mathbf{b}}(\bar{f})\big)$}
\obj(39,20)[B]{$\big(\alpha\cdot\vec{e},\mathbf{a}',\mathbf{b}',\phi^{-1}_{\mathbf{P},\alpha,\delta}(\alpha\cdot\vec{\omega}),\phi^{-1}_{\mathbf{P},\alpha,\delta}(\vec{\zeta}),gl_{\mathbf{a},\mathbf{b}}(\bar{f})\circ\phi_{\alpha,\delta}\big)\in\widehat{U}$}
\obj(2,10)[C]{$(\vec{e},\widetilde{\mathbf{a}},\mathbf{b},\vec{\omega},\vec{\zeta},\bar{f})$}
\obj(43,10)[D]{$(\vec{e},\widetilde{\mathbf{a}}',\mathbf{b}', \vec{\omega}',\vec{\zeta}',\bar{f}')$}
\mor(10,20)(16,20){}[\atleft,\aplicationarrow]
\mor(8,10)(35,10){}[\atleft,\aplicationarrow]
\mor(2,10)(2,20){$\Phi$}[\atleft,\aplicationarrow]
\mor(43,10)(43,20){$\Phi$}[\atleft,\aplicationarrow]
\enddc
\end{center}
\caption{Isotropy group action on $\widehat{U}$ in local coordinates}\label{uhattrans}
\end{figure}

Finally, we must consider the map $\bar{f} \mapsto \bar{f}'$. We first note that both $\bar{f},\bar{f}'$ are elements of $W$, the precompact space of $W^{1,p}$ solutions to $(\ref{solution})$. The space $W$ inherits norms from $W^{1,p}(\Sigma_{\mathbf{P},0},M)$ and also $W^{1,p}\big(\Sigma_{\mathbf{P},0}\setminus \mathcal{N}(nodes),M\big)$ by restriction. These two norms are equivalent because $W$ is finite dimensional and all norms on finite dimensional spaces are equivalent. Therefore, it suffices to work on $(\Sigma_0\setminus \mathcal{N}(nodes))$ and we denote the restriction of $W$ by $W_N\subset W^{1,p}\big(\Sigma_{\mathbf{P},0}\setminus \mathcal{N}(nodes),M\big)$.

Let $W_N^0$ denote a small open neighborhood of $f_0$ in $W_N$. Proposition \ref{corc1} below will let us analyze the map
\begin{align}
\nonumber B^{2K}\times B^{2(2K-p_n)} \times W_N^0 &\rightarrow B^{2K}\times B^{2(2K-p_n)}\times W_N^0\\
\label{c1ssiso}(\widetilde{\mathbf{a}},\mathbf{b},\bar{f}) &\mapsto (\widetilde{\mathbf{a}}',\mathbf{b}',\bar{f}').
\end{align}
To do this we use the gluing maps $gl_{\mathbf{a},\mathbf{b}}$ and the maps
\[\iota_{\mathbf{P},\mathbf{a},\mathbf{b}} : (\Sigma_{\mathbf{P},0}\setminus \mathcal{N}(nodes))\rightarrow (\Sigma_{\mathbf{P},\mathbf{a},\mathbf{b}}\setminus\mathcal{N}(nodes))\]
from (\ref{iota}) that identify domains.

\begin{proposition}\label{corc1}
Under the hypotheses above, the map
\begin{align*}
B^{2K}\times B^{2(2K-p_n)} \times W_N^0 &\rightarrow W_N^0\\
(\widetilde{\mathbf{a}},\mathbf{b},\bar{f}) &\mapsto gl_{\psi(\widetilde{\mathbf{a}}),\mathbf{b}}(\bar{f})\circ\iota_{\mathbf{P},\psi(\widetilde{\mathbf{a}}),\mathbf{b}}
\end{align*}
is $C^1$ SS.
\end{proposition}
\begin{proof}[Proof of Proposition \ref{corc1}]
Proposition \ref{corc1} follows from the $C^1$ gluing results of Section \ref{c1gluing}. The technical component is Theorem \ref{mainthm}. The special case of Proposition \ref{corc1} for two regular $J$-holomorphic spheres meeting at a node is stated and proved as Corollary \ref{corc1easy}. For generalizing Corollary \ref{corc1easy}, see Remark \ref{extensions}.
\end{proof}

Continuing to consider (\ref{c1ssiso}), by following definitions and Figure \ref{uhattrans} we see
\begin{align*}
\bar{f}' &= gl_{\mathbf{a}',\mathbf{b}'}^{-1}(f\circ\phi_{\alpha,\delta})\\
&=gl_{\mathbf{a}',\mathbf{b}'}^{-1}\big(gl_{\mathbf{a},\mathbf{b}}(\bar{f})\circ\phi_{\alpha,\delta}\big)\\
&=gl_{\mathbf{a}',\mathbf{b}'}^{-1}\big(gl_{\mathbf{a},\mathbf{b}}(\bar{f})\circ\iota_{\mathbf{P},\mathbf{a},\mathbf{b}}\circ\iota_{\mathbf{P},\mathbf{a},\mathbf{b}}^{-1}\circ\phi_{\alpha,\delta}\circ\iota_{\mathbf{P},\mathbf{a}',\mathbf{b}'}\circ\iota_{\mathbf{P},\mathbf{a}',\mathbf{b}'}^{-1}\big)\\
&=gl_{\psi(\widetilde{\mathbf{a}}'),\mathbf{b}'}^{-1}\big(gl_{\psi(\widetilde{\mathbf{a}}),\mathbf{b}}(\bar{f})\circ\iota_{\mathbf{P},\psi(\widetilde{\mathbf{a}}),\mathbf{b}}\circ\iota_{\mathbf{P},\psi(\widetilde{\mathbf{a}}),\mathbf{b}}^{-1}\circ\phi_{\alpha,\delta}\circ\iota_{\mathbf{P},\psi(\widetilde{\mathbf{a}}'),\mathbf{b}'}\circ\iota_{\mathbf{P},\psi(\widetilde{\mathbf{a}}'),\mathbf{b}'}^{-1}\big).
\end{align*}
So (\ref{c1ssiso}) can be written as
\begin{align}
\label{provec1ss}&(\widetilde{\mathbf{a}},\mathbf{b},\bar{f})\mapsto\\
\nonumber&\Big(\widetilde{\mathbf{a}}',\mathbf{b}',gl_{\psi(\widetilde{\mathbf{a}}'),\mathbf{b}'}^{-1}\big(gl_{\psi(\widetilde{\mathbf{a}}),\mathbf{b}}(\bar{f})\circ\iota_{\mathbf{P},\psi(\widetilde{\mathbf{a}}),\mathbf{b}}\circ\iota_{\mathbf{P},\psi(\widetilde{\mathbf{a}}),\mathbf{b}}^{-1}\circ\phi_{\alpha,\delta}\circ\iota_{\mathbf{P},\psi(\widetilde{\mathbf{a}}'),\mathbf{b}'}\circ\iota_{\mathbf{P},\psi(\widetilde{\mathbf{a}}'),\mathbf{b}'}^{-1}\big)\Big)
\end{align}
We will prove (\ref{provec1ss}) is $C^1$ SS by examining it in pieces. First, note that (\ref{provec1ss}) is smooth in the $\bar{f}$ direction by \cite[Proposition 10.5.4]{JHOL}. Also we have already noted above that the map
\begin{equation}\label{c1diffeo}
(\widetilde{\mathbf{a}},\mathbf{b})\mapsto(\widetilde{\mathbf{a}}',\mathbf{b}')
\end{equation}
is a $C^1$ SS local diffeomorphism. The map
\begin{align*}
B^{2K}\times B^{2(2K-p_n)} \times W_N^0 &\rightarrow W_N^0\\
(\widetilde{\mathbf{a}},\mathbf{b},\bar{f}) &\mapsto gl_{\psi(\widetilde{\mathbf{a}}),\mathbf{b}}(\bar{f})\circ\iota_{\mathbf{P},\psi(\widetilde{\mathbf{a}}),\mathbf{b}}
\end{align*}
is $C^1$ SS by Proposition \ref{corc1}. Next, the composition
\begin{align*}
B^{2K}\times B^{2(2K-p_n)} &\rightarrow W^{1,p}\Big(\big(\Sigma_{\mathbf{P},0}\setminus \mathcal{N}(nodes)\big),\big(\Sigma_{\mathbf{P},0}\setminus \mathcal{N}(nodes)\big)\Big)\\
(\widetilde{\mathbf{a}},\mathbf{b}) &\mapsto \iota_{\mathbf{P},\psi(\widetilde{\mathbf{a}}),\mathbf{b}}^{-1}\circ\phi_{\alpha,\delta}\circ\iota_{\mathbf{P},\psi(\widetilde{\mathbf{a}}'),\mathbf{b}'}
\end{align*}
does not depend on $\bar{f}$ and is $C^1$ SS. Finally, we note that Proposition \ref{corc1} implies that
\[(\widetilde{\mathbf{a}},\mathbf{b},\bar{f}) \mapsto gl_{\psi(\widetilde{\mathbf{a}}),\mathbf{b}}(\bar{f})\circ\iota_{\mathbf{P},\psi(\widetilde{\mathbf{a}}),\mathbf{b}}\]
is $C^1$ SS. Consider the map
\begin{equation}\label{c1diffeq}
(\widetilde{\mathbf{a}},\mathbf{b},\bar{f}) \mapsto (\widetilde{\mathbf{a}},\mathbf{b},gl_{\psi(\widetilde{\mathbf{a}}),\mathbf{b}}(\bar{f})\circ\iota_{\mathbf{P},\psi(\widetilde{\mathbf{a}}),\mathbf{b}}).
\end{equation}
The map (\ref{c1diffeq}) is a $C^1$ SS local diffeomorphism because it is of the form
\begin{equation}\label{biga}
(A,F)\mapsto (A,\varphi_A(F))
\end{equation}
where for each $A$, $\varphi_A$ is a local diffeomorphism that varies $C^1$ SS with $A$. Therefore, (\ref{biga}) has an inverse that we will denote
\[(A,f)\mapsto(A,gl_A^{-1}(f)).\]
In this notation, (\ref{c1diffeq}) has a $C^1$ SS inverse given by
\[(\widetilde{\mathbf{a}},\mathbf{b},g) \mapsto \big(\widetilde{\mathbf{a}},\mathbf{b},gl^{-1}_{\psi(\widetilde{\mathbf{a}}),\mathbf{b}}(g\circ\iota^{-1}_{\mathbf{P},\psi(\widetilde{\mathbf{a}}),\mathbf{b}})\big).\]
Thus, because (\ref{c1diffeo}) is also a $C^1$ SS local diffeomorphism,
\[(\widetilde{\mathbf{a}},\mathbf{b},\bar{f}) \mapsto (\widetilde{\mathbf{a}}',\mathbf{b}',gl_{\psi(\widetilde{\mathbf{a}}'),\mathbf{b}'}(\bar{f})\circ\iota_{\mathbf{P},\psi(\widetilde{\mathbf{a}}'),\mathbf{b}'})\]
is a $C^1$ SS local diffeomorphism. Therefore, we see that
\begin{align*}
B^{2K}\times B^{2(2K-p_n)}\times W_N^0 &\rightarrow B^{2K}\times B^{2(2K-p_n)}\times W_N^0\\
(\widetilde{\mathbf{a}},\mathbf{b},g) &\mapsto \Big(\widetilde{\mathbf{a}}',\mathbf{b}',gl_{\psi(\widetilde{\mathbf{a}}'),\mathbf{b}'}^{-1}\big(g\circ\iota_{\mathbf{P},\psi(\widetilde{\mathbf{a}}'),\mathbf{b}'}^{-1}\big)\Big)
\end{align*}
is $C^1$ SS. Putting all of these facts together we conclude that (\ref{provec1ss}) is a $C^1$ SS local diffeomorphism and hence the action of $\Gamma$ on $\widehat{U}$ is $C^1$ SS. This proves (\ref{c12}).\\

\noindent\textbf{Proof of (\ref{c13})--Smoothness of slicing conditions}: To impose the slicing conditions, we fix the domain of $f$ and require the extra marked points to lie in the slicing manifold $Q$. More precisely,
\[f\circ \iota_{\mathbf{P},\mathbf{a},\mathbf{b}}(\omega^\ell)\in Q.\]
If $f$ is close enough to $f_0$, then $f\pitchfork Q$ and for every point in $f_0^{-1}(Q)$, there is a unique close point in $f^{-1}(Q)$. By the implicit function theorem, if $f\pitchfork Q$ then the map
\begin{align*}
W^{k,p}(S^2,M)&\rightarrow S^2\\
f&\mapsto f^{-1}(Q)
\end{align*}
is $C^\ell$ if $k>\ell+\frac{2}{p}$ so $W^{k,p}(S^2)\subset C^{\ell}(S^2)$. For more discussion of this issue, see \cite[$\S 3$]{mwfund}.

Proposition \ref{corc1} implies that the map
\begin{align*}
B^{2K}\times B^{2(2K-p_n)} &\rightarrow W_N^0\\
(\widetilde{\mathbf{a}},\mathbf{b}) &\mapsto gl_{\psi(\widetilde{\mathbf{a}}),\mathbf{b}}(\bar{f})\circ\iota_{\psi(\widetilde{\mathbf{a}}),\mathbf{b}}
\end{align*}
is $C^1$ SS as a function of $(\widetilde{\mathbf{a}},\mathbf{b})$ to the space $W_N^0$ equipped with the $W^{k,p}$-norm, $k\geq2$ (because $W_N^0$ is finite dimensional, all norms are equivalent). Hence, for a fixed $\bar{f}\in W_N^0$, we conclude that the map
\begin{align*}
B^{2K}\times B^{2(2K-p_n)} &\rightarrow S^2\\
(\widetilde{\mathbf{a}},\mathbf{b}) &\mapsto \left(gl_{\psi(\widetilde{\mathbf{a}}),\mathbf{b}}(\bar{f})\circ\iota_{\psi(\widetilde{\mathbf{a}}),\mathbf{b}}\right)^{-1}(Q)
\end{align*}
is $C^1$ SS. We also clearly get $C^\infty$ differentiability from varying $\bar{f}$. This proves (\ref{c13}).\\

\noindent\textbf{Proof of (\ref{c14})--Smoothness of U and structural maps}: The domain of our charts, $U$, is obtained from $\widehat{U}$ by making $\widehat{U}$ $\Gamma$-invariant and imposing slicing conditions (as in (\ref{slice})). We need to show that $U$ is a $C^1$ SS submanifold of $\widehat{U}$. We have shown that $\Gamma$ acts $C^1$ SS on $\widehat{U}$, so $\widehat{U}$ can be assumed to be $\Gamma$-invariant by taking an intersection of the finite orbit of $\widehat{U}$ under $\Gamma$.

We have shown in (\ref{c13}) that the marked points vary $C^1$ SS. By \cite[Proposition 4.1.4]{mwfund} the slicing conditions are transverse. Therefore, we conclude that $U$ is a $C^1$ SS submanifold of $\widehat{U}$. It is also easily seen that $\Gamma$ acts $C^1$ SS on $U$.

Finally, to complete the proof of (\ref{c14}), we consider coordinate changes. Section \ref{chartsincoords} gives an explicit description of the coordinate changes in coordinates. The coordinate change maps have the same form as the isotropy group action and hence are $C^1$ SS by the same reasoning. This completes the proof of (\ref{c14}) and hence Theorem \ref{c1atlasthm}.
\end{proof}

We will now show that Theorem \ref{c1atlasthm} and Proposition \ref{dmlemma} are compatible and prove Proposition \ref{dm1}. Proposition \ref{dm1} is implied by the stronger Lemma \ref{forgetlemma} below. In fact, the domain $U$ of a basic chart carries \textit{two} important forgetful maps that are useful in applications. The first is the forgetful map $\pi_1:U\rightarrow\overline{\mathcal{M}}_{0,k+L}^{new}$ that takes an element $(\vec{e},\mathbf{a},\mathbf{b},\vec{\omega},\vec{\zeta},f)$ to the underlying domain of $f$ with all marked points $\mathbf{w},\mathbf{z}$ (this does not collapse any spheres because the extra $L$ points were chosen to stabilize the domain). The other forgetful map that is useful is $\pi_0:U\rightarrow\overline{\mathcal{M}}_{0,k}^{new}$ which takes an element $(\vec{e},\mathbf{a},\mathbf{b},\vec{\omega},\vec{\zeta},f)$ to the stabilization of the domain $f$ with only marked points $\mathbf{z}$. With the modification to the charts for $U$ made by (\ref{gluingcharts}) it is no longer clear that $\pi_0$ and $\pi_1$ are smooth maps. The following lemma verifies that they are.

\begin{lemma}\label{forgetlemma}
Let $U$ be a Gromov-Witten Kuranishi chart with $C^1$ SS structure given by Theorem \ref{c1atlasthm}. Let $\mathcal{M}_{0,m}^{new}$ denote genus zero Deligne-Mumford space with the modified smooth structure given by Proposition \ref{dmlemma}. Then
\begin{align*}
\pi_1:U&\rightarrow\overline{\mathcal{M}}_{0,k+L}^{new}\\
\pi_0:U&\rightarrow\overline{\mathcal{M}}_{0,k}^{new}
\end{align*}
are $C^1$ SS.
\end{lemma}
\begin{proof}
Given Theorem \ref{c1atlasthm} and Proposition \ref{dmlemma}, the proof of Lemma \ref{forgetlemma} is simple. The map $\pi_1$ is easily seen to be $C^1$ SS by examining the map in local coordinates. It is also easily seen that $\pi_0$ is $C^1$ SS because it is the composition of $\pi_1$ with that forgetful map $\overline{\mathcal{M}}_{k+L}^{new}\rightarrow\overline{\mathcal{M}}_{k}^{new}$. The way the smooth structure on $\overline{\mathcal{M}}_{0,m}^{new}$ was defined in Proposition \ref{dmlemma}, it is clear that the forgetful map $\overline{\mathcal{M}}_{0,m}^{new}\rightarrow\overline{\mathcal{M}}_{0,m-1}^{new}$ is $C^1$ SS, c.f. Remark \ref{smoothforget}.
\end{proof}

\begin{remark}\label{gluingprofile}
The gluing maps (\ref{gluingcharts}) are used to define a $C^1$ SS structure on the domain $U$ (and the corresponding structural maps). In some literature this is referred to as a choice of \emph{gluing profile}. In $(\ref{gluingcharts})$ a gluing profile of $\psi:x\mapsto x|x|^{p-1}$, where $p>2$, is used. Other authors have different choices\footnote{\cite[Corollary 10.9.3]{JHOL} is slightly incorrect in their choice of a gluing profile. In their notation, $\beta(s)$ should be chosen to be $\beta(s) = s^\nu, \nu>2$, in order to achieve $C^1$ differentiability.}. We do not claim that the choice of $(\ref{reparam})$ is canonical.\hfill$\Diamond$
\end{remark}

\section{$C^1$ gluing}\label{c1gluing}

The goal of this section is to prove a ``$C^1$ gluing theorem" for two $J$-holomorphic curves.

\subsection{Overview and statement of main results}\label{gluingoverview}
Given a regular almost complex structure $J$, the gluing map $\iota_k^{\delta,R}$ is defined for two genus zero $J$-holomorphic curves $(u^0,u^\infty)$ such that $u^{0}(0) = u^{\infty}(0)$ and
\[\|du^0\|_{L^\infty}\leq k,\qquad \|du^{\infty}\|_{L^\infty}\leq k,\]
where $k$ is some constant. (The notation $\iota_k^{\delta,R}$ is chosen to agree with \cite{JHOL}, but we also use $gl$ to denote the gluing map in Section \ref{c1sscharts} and later in this section.) The gluing of these two curves, $\iota_k^{\delta,R}(u^0,u^\infty)$, is a $J$-holomorphic curve that converges in the Gromov topology to $(u^0,u^\infty)$ as $R\rightarrow\infty$. It is defined for $0<\delta<\delta_0$ and $R>\frac{1}{\delta_0\delta}$, where $\delta_0$ is a small constant depending only on $k$. It is described in \cite{JHOL} the various conditions that $\delta_0$ must satisfy. We will think of fixing $k$ and a sufficiently small $\delta$ and varying $R$. This section studies the derivative of the map $\iota_k^{\delta,R}$ with respect to $R$. Before stating the main results, we introduce some more notation.

We will consider $u^0, u^\infty$ as having fixed domains $S_0,S_{\infty}$ and identify $S_i$, $i\in\{0,\infty\}$, with $S_i = S^2 = \mathbb{C}\cup\{\infty\}$. The domain of the glued curve $\iota_k^{\delta,R}(u^0,u^\infty)$, denoted by $\Sigma_{R,\delta}$ (or just $\Sigma_R$), is also diffeomorphic to $S^2$, but rather than being identified with $\mathbb{C}\cup\{\infty\}$, has two charts
\begin{equation}\label{twocharts}
T_i := \left\{|z| > \frac{2\delta}{R}\right\} \subset S_i.
\end{equation}
These charts are identified over the regions
\begin{equation}\label{gluingregions}
\left\{\frac{2\delta}{R} < |z| < \frac{1}{2\delta R}\right\} \subset T_0, \left\{\frac{2\delta}{R} < |z| < \frac{1}{2\delta R}\right\} \subset T_{\infty}
\end{equation}
via the map
\begin{align}\label{phineck}
\phi_R: T_0 &\rightarrow T_\infty\\
\nonumber z &\mapsto \frac{1}{R^2z}.
\end{align}
Therefore, we have
\begin{align}
\label{sigmar}\Sigma_R &= T_0 \cup_{\phi} T_\infty\\
\nonumber &= \left\{|z| > \frac{2\delta}{R}\right\} \bigcup_{\phi} \left\{|z| > \frac{2\delta}{R}\right\}.
\end{align}
We will use the following terminology for regions in $\Sigma_R$: The region
\[\left\{\frac{2\delta}{R} < |z| < \frac{1}{2\delta R}\right\} \subset T_0, T_\infty \subset \Sigma_R\]
on which $T_0, T_\infty$ are identified is called the \textbf{inner neck} of $\Sigma_R$. The region
\[\left\{z \left| z\in T_0, \frac{2\delta}{R} < |z| < \frac{1}{\delta R}\right.\right\}\bigcup_{\phi_R}\left\{z \left| z\in T_0, \frac{2\delta}{R} < |z| < \frac{1}{\delta R}\right.\right\}\]
is called the \textbf{neck}. The complement of the inner neck in the neck is called the \textbf{outer neck}. We will sometimes abuse terminology and also use outer neck to refer to $\left\{|z| > \frac{1}{2\delta R}\right\}$, i.e. the complement of the inner neck in all of $T_0$ or $T_\infty$.

\begin{remark}
This is a different setup than that of \cite{JHOL}. For every $R$, they identify the domain of their glued curve with $\mathbb{C}$ by an $R$-dependent identification that involves rescaling the $S_\infty$ factor. This has the advantage of simplifying notation and enables one to work on a fixed domain $\mathbb{C}$. Additionally, they make their norms symmetric in $u^0, u^\infty$ so all estimates they do are done on just one half. However, as pointed out by Kenji Fukaya, taking an $R$ derivative in their setting is not symmetric. Our construction is completely symmetric with respect to $u^0,u^\infty$; this makes the notation more complicated, but the estimates much easier. Although our setup is different from \cite{JHOL}, it is easy to see that their estimates continue to hold in our setting as well. \hfill $\Diamond$
\end{remark}

We can then restrict the glued map $\iota^{\delta,R}(u^0,u^\infty)$ to each chart $T_i$ to obtain
\begin{align*}
\iota^{\delta,R}_i(u^0,u^\infty): T_i \rightarrow M, \qquad i=0,\infty.
\end{align*} (We will use subscript $``i"$ to denote the restriction to the chart $T_i$ throughout this section.) The main result of this section is a bound on the $R$ derivative of the glued map in each chart.

\begin{theorem}\label{mainthm}
Under the hypotheses above, let $\iota^{\delta,R}$ denote the gluing map, $\iota^{\delta,R}_i$ its restriction to the chart $T_i$, $p>2$, and $(u^0,u^\infty)$ be a fixed pair of $J$-holomorphic spheres meeting at a point. Then for $R$ sufficiently large,
\[\left\|\left.\frac{d}{dR}\right|_{R=R_0}\iota^{\delta,R}_i(u^0,u^\infty)\right\|_{W^{1,p}} \leq C\frac{1}{R_0}\frac{1}{(\delta R_0)^{2/p}}\]
where $C$ is a constant that depends only on $p,\omega,k,J$ and $\|\cdot\|_{W^{1,p}}$ denotes an appropriate $W^{1,p}$ norm in the chart $T_i$ (depending on $R$); this norm is described in Section \ref{nsm}.
\end{theorem}
\noindent The same result immediately follows for complex gluing parameters (in which case the regions $T_i$ from (\ref{twocharts}) are identified by a map $\phi$ that depends on $R$ and also twists by an additional angle parameter $\theta$).

The primary use of Theorem \ref{mainthm} is Corollary \ref{corc1easy}. Before stating the corollary we need some preliminaries, including translating between the language of \cite{JHOL}, the language of this section, and the language of Section \ref{gwcharts}. In the analysis of the gluing theorem, one wants to work with a fixed domain as $R$ varies. This section will focus on analyzing the simple case of two curves meeting at a node to better illuminate the main ideas. However, we will now describe a more general language for doing that is also used in Section \ref{gwcharts}.


Let $T$ be the stable tree with two components and two marked points on each component. Then we can think of $(u^0,u^\infty)$ as an element of the moduli space $\mathcal{M}_{0,T}(A^{0,\infty};J)$ of $J$-holomorphic maps with four marked points, modelled over $T$ and representing the homology classes $A^0$ and $A^\infty$. An equivalence class in $\mathcal{M}_{0,T}$ has a unique representative $[\mathbf{u},\mathbf{z}]$ such that
\[z^0=0,\quad z_1^0=1,\quad z_2^0=\infty,\quad z^\infty = 0,\quad z_3^\infty=e^{i\theta},\quad z_4^\infty = \infty\]
where $z_1^0,z_2^0$ are the marked points on one node, $z_3^\infty,z_4^\infty$ are the marked points on the other node, and $z^0,z^\infty$ are the nodal points. Let $\Sigma_0$ denote the parameterized domain of $[\mathbf{u},\mathbf{z}]$.

Define the element $[gl_\lambda(\mathbf{u}),\mathbf{z}_\lambda]\in\mathcal{M}_{0,4}(A;J)$ by
\[gl_\lambda(\mathbf{u}):=\iota_k^{\delta,R}(u^0,u^\infty),\qquad \mathbf{z}_\lambda:=(1,\infty,\lambda, 0)\]
where $\lambda = \frac{e^{i\theta}}{R^2}$. Define $[gl_0(\mathbf{u}),\mathbf{z}_0] = [\mathbf{u},\mathbf{z}]\in\overline{\mathcal{M}}_{0,4}(A;J)$. With this, the map
\begin{align*}
B_\rho\times\mathcal{M}_{0,T} &\rightarrow \overline{\mathcal{M}}_{0,4}(A;J)\\
(\lambda,[\mathbf{u},\mathbf{z}]) &\mapsto [gl_\lambda(\mathbf{u}),\mathbf{z}_\lambda]=\iota_k^\lambda([\mathbf{u},\mathbf{z}])
\end{align*}
is well defined for $\rho<(\delta\delta_0)^2$. The \textbf{gluing parameter}
\[\lambda = w(z_4,z_1,z_2,z_3)\]
is exactly the cross ratio of the marked points of the element $[gl_\lambda(\mathbf{u}),\mathbf{z}_\lambda]$ in the image of the gluing map. Let $\Sigma_\lambda$ denote the domain of $[gl_\lambda(\mathbf{u}),\mathbf{z}_\lambda]$.

Let $\Delta$ denote a neighborhood of $\Sigma_0$ in $\overline{\mathcal{M}}_{0,4}$ and $\mathcal{C}|_{\Delta}$ denote the universal curve over $\Delta$. We can identify the domains $\Sigma_\lambda\setminus \mathcal{N}(nodes)$ with a \textit{fixed} open subset $\Sigma_0\setminus \mathcal{N}(nodes)$ of $\Sigma_0$. We will denote this identification by a map $\kappa$ (compare with (\ref{iota})).
\begin{align}
\label{kappa}\kappa &: (\Sigma_0\setminus \mathcal{N}(nodes)) \times B \rightarrow \mathcal{C}|_{\Delta}\\
\nonumber\kappa_{\lambda} &: (\Sigma_0\setminus \mathcal{N}(nodes))  \times \{\lambda\} \mapsto \Sigma_{\lambda}\setminus\mathcal{N}(nodes).\nonumber
\end{align}
Here $B\subset \mathbb{C}$ denotes a neighborhood of $0$.

In the language of \cite{JHOL}, $\lambda=\frac{e^{i\theta}}{R^2}$ and $\Sigma_0=S^2\vee S^2 = (\mathbb{C}\cup\{\infty\})\vee (\mathbb{C}\cup\{\infty\})$. In \cite{JHOL}, $\kappa_{\lambda}$ involves a rescaling, but in our formulation $\kappa_{\lambda}$ identifies $\Sigma_0\setminus \mathcal{N}(nodes)$ with $\Sigma_R \setminus \mathcal{N}(nodes)$ by the charts $T_i$ for $\Sigma_R$ as described above.

This identification allows us to work on a \textit{fixed} domain $\Sigma_0\setminus \mathcal{N}(nodes)$ and we can consider the map
\begin{align}\label{glmap}
B &\rightarrow W^{1,p}(\Sigma_0\setminus \mathcal{N}(nodes),M)\\
\lambda &\mapsto gl_\lambda(\mathbf{u})\circ\kappa_\lambda.
\end{align}
The main result is that after an appropriate reparametrization, $(\ref{lambdareparam})$, this map is $C^1$.

\begin{corollary}\label{corc1easy}
Under the hypotheses of Theorem \ref{mainthm}, let $p>2$, $B\subset \mathbb{C}$ denote a small neighborhood of $0$, and for $\lambda\in B$ define
\begin{equation}\label{lambdareparam}
\widetilde{\lambda} = |\lambda|^{\frac{1}{p}-1}\lambda.
\end{equation}
Let $\mathbf{u}:\Sigma_0\rightarrow M$ be a fixed parameterized nodal curve and $gl_\lambda(\mathbf{u}):\Sigma_\lambda\rightarrow M$ its gluing with gluing parameter $\lambda$. Let $\kappa$ denote the map that identifies $(\Sigma_0\setminus \mathcal{N}(nodes))$  and the domains $\Sigma_{\lambda}\setminus\mathcal{N}(nodes)$ as in (\ref{kappa}).
Then the map
\begin{align}
\noindent B &\rightarrow W^{1,p}(\Sigma_0\setminus \mathcal{N}(nodes),M)\\
\label{glkmap}\widetilde{\lambda} &\mapsto gl_{\lambda}(\mathbf{u})\circ\kappa_{\lambda} = gl_{|\widetilde{\lambda}|^{p-1}\widetilde{\lambda}}(\mathbf{u})\circ\kappa_{|\widetilde{\lambda}|^{p-1}\widetilde{\lambda}}
\end{align}
is $C^1$.
\end{corollary}

The proof of Corollary \ref{corc1easy} is contained in Section \ref{corproof} and after appealing to Theorem \ref{mainthm}, it amounts to a calculus exercise.

\begin{remark}\label{extensions}[Extensions of $C^1$ gluing theorem]
Theorem \ref{mainthm} and Corollary \ref{corc1easy}, as stated, only prove differentiability of the gluing map with respect to the gluing parameter in the simplest case of two regular curves intersecting at one point. For applications, of course, one is interested in considering both nonregular curves (that is to say with some nonzero obstruction bundle) and curves with multiple nodes. We remark here that both of these situations can be dealt with by techniques standard in the literature. For curves with a nonvanishing obstruction bundle, $E\neq0$, one uses the graph construction of Gromov and considers the graph of the curve $S^2\rightarrow S^2\times M$ which will be a regular $J$-holomorphic curve for an appropriate $J$ (that incorporates $e\in E$). For more discussion of this, see \cite[Remark 4.1.2]{notes} and \cite[Chapter 6.2]{JHOL}.

For curves with multiple nodes, we note that estimates required for the construction of the gluing map as well as the estimates of Theorem \ref{mainthm} are local in nature and therefore one can construct a gluing map for a curve with multiple nodes by localizing around each node via more charts and a partition of unity. With multiple nodes, one can also generalize Corollary \ref{corc1easy} to arbitrary gluing parameters. The generalization needed for this thesis is stated as Proposition \ref{corc1}.
\hfill$\Diamond$
\end{remark}

Section \ref{construction} contains an overview of the construction of the gluing map and presents the basic strategy for proving Theorem \ref{mainthm}. Sections \ref{nsm} and \ref{asd} contain the analytic setup and prove Theorem \ref{mainthm} assuming some technical lemmas. The proofs of the technical lemmas are contained in Section \ref{derivativeestimates}. Corollary \ref{corc1easy} is proved in Section \ref{corproof}.
\subsection{Review of the construction of the gluing map}\label{construction}
Following \cite{JHOL}, the construction of the gluing map is as follows: First a pregluing $u^{R}(z)$ of $(u^0,u^\infty)$ is defined (its exact definition is given in (\ref{pregluing})). The pregluing $u^R(z)$ is not a $J$-holomorphic curve, but rather an approximate $J$-holomorphic curve. It is then proven that there is a vector field $\widetilde{\xi}^R$ along $u^R(z)$ such that $\exp_{u^R(z)}(\widetilde{\xi}^R)$ is a $J$-holomorphic curve. We then define
\[\iota^{\delta,R}(u^0,u^\infty)(z) := \exp_{u^R(z)}(\widetilde{\xi}^R).\] The existence of such a $\widetilde{\xi}^R$ is provided by a version of the implicit function theorem. We recall the specific version needed in Proposition \ref{ift}.
\begin{remark}[Choices in the gluing construction]\label{choices}
As mentioned above, we will think of $R$ as the variable of the gluing construction and fix a $\delta$ sufficiently small. There are many other choices that are made in the course of the gluing construction. The constants $p>2$ and $k>0$ only effect the choice of constant $\delta_0$, but do not effect the gluing map.

Together, the symplectic form $\omega$ and almost complex structure $J$ determine a metric $g$ from which we define our norms that will be used the construction. This metric in turn will determine the complement of Ker$D_{u}$ in which the image of the right inverse $Q_u$ will be defined (see Section \ref{nsm} and (\ref{qt})). The construction also relies on a choice of cutoff function $\rho$ in (\ref{rho}).

These are the only choices that are made and that is due to the fact that we are working in genus zero case. In the genus zero case, in contrast to the higher genus setting, one has canonical coordinates near the nodes. One would expect all results of this section to continue to hold in the higher genus setting, but would need to provide more details. \hfill $\Diamond$
\end{remark}
\begin{proposition}\cite[Proposition A.3.4]{JHOL}\label{ift}
Let $X$ and $Y$ be Banach spaces, and $f:X\rightarrow Y$ be a continuously differentiable map. We will let $d_0f:X\rightarrow Y$ denote the derivative of $f$ at $0\in X$. Suppose that the following conditions hold:
\begin{enumerate}[(a)]
\item $d_0f:X\rightarrow Y$ is surjective.
\item $d_0f$ has a (bounded linear) right inverse $Q: Y\rightarrow X$.
\end{enumerate}
Let $\delta,c$ be positive constants such that
\begin{itemize}
\item $\|Q\|\leq c$.
\item $\|x\|<\delta \quad\Longrightarrow\quad \|d_xf-d_0f\|\leq\frac{1}{2c}.$
\end{itemize}
Finally assume that
\begin{enumerate}[(c)]
\item $\|f(0)\|\leq\frac{\delta}{4c}$.
\end{enumerate}
Then there is a unique $x^*\in X$ such that
\[f(x^*) = 0,\quad x^*\in \operatorname{Im} Q,\quad \|x^*\|<\delta\]
and $x^*$ also satisfies
\[\|x^*\|\leq 2c\|f(0)\|.\]
\end{proposition}
\begin{remark}[On the proof of Proposition \ref{ift}]Proposition \ref{ift} is proved using a Newton process. Specifically, it is shown that
\[f(x) = 0,\quad x\in \operatorname{Im} Q\qquad\Longleftrightarrow\qquad -Qf(x) + Q(d_0f(x))=x.\]
Then it is shown that $-Qf(x) + Q(d_0f(x))$ has a fixed point via the contraction mapping principle. \hfill$\Diamond$
\end{remark}
We will need the following general lemma about contraction mappings.
\begin{lemma}\label{fpderivativebound}
Suppose that $\phi_t:X\rightarrow X$ satisfies
\[\|\phi_t(x) -\phi_t(y)\| < \frac{1}{2}\|x-y\|\]
for all $x,y\in X$ and sufficiently small $t$. Let $x_t^*$ be the unique fixed point of $\phi_t$. Suppose also that
\begin{equation}\label{fixedpointder}
\|\phi_t(x_0^*) - \phi_0(x_0^*)\|\leq Kt.
\end{equation}
Then $\|x_t^*-x_0^*\|\leq 2Kt$.
Thus, if the map $t\mapsto x_t^*$ is known to be differentiable, then
\[\left\|\left.\frac{d}{dt}\right|_{t=0}x_t^*\right\| \leq 2K.\]
\end{lemma}
\begin{proof}
The Newton process tells us that $x_t^* = \lim_{n\rightarrow\infty}\phi_t^n(x_0^*)$. First,
\[\|\phi_t(x_0^*) - x_0^*\| = \|\phi_t(x_0^*) - \phi_0(x_0^*)\| \leq Kt\]
using hypothesis (\ref{fixedpointder}). Then
\[\|\phi_t^2(x_0^*) - \phi_t(x_0^*)\| < \frac{1}{2}Kt\]
using the contraction property of $\phi_t$. Similarly,
\[\|\phi_t^{n+1}(x_0^*) - \phi_t^n(x_0^*)\| < \frac{1}{2^n}Kt.\]
Adding these inequalities over all $n$ we get the desired result.
\end{proof}
We now aim to estimate $\frac{d}{dR}\iota^{\delta,R}$. The rough strategy is to apply Proposition \ref{ift} to a family of maps and estimate their derivative using Lemma \ref{fpderivativebound}. Important to note is that Lemma \ref{fpderivativebound} applies to a family of maps on a \textit{fixed} Banach space. Describing the setup in which this is the case is delicate and is the subject of the next section.

\subsection{Norms, spaces, and maps}\label{nsm}
This section contains the appropriate setup of norms, spaces, and maps in order to describe the gluing map as a parameterized implicit function theorem problem on a fixed Banach space. We begin by fixing a regular $\omega$-compatible almost complex structure $J$. Together with $\omega$ this determines a metric on $M$. On $S^2 = \mathbb{C}\cup\{\infty\}$, we will use the round metric
\begin{equation}\label{roundmetric}
\frac{1}{(1+|z|^2)^2}(ds^2+dt^2).
\end{equation}
We now use this metric to define a metric on $\Sigma_R$, the domain of the glued curve (see (\ref{sigmar})). Recall from Section \ref{gluingoverview} that $\Sigma_R$ has two charts $T_0,T_\infty = \{|z| > \frac{2\delta}{R}\} \subset S^2$. These two charts were identified over a region called the \textit{inner neck}, $\{\frac{2\delta}{R} < |z| < \frac{1}{2\delta R}\}$, and this identification was via the map $\phi_R$ from (\ref{phineck}). The complement of the inner neck is called the outer neck.

Take a partition of unity $\{\psi_0, \psi_\infty\}$ subordinate to the cover $T_0, T_\infty$ of $\Sigma_R$ and that is symmetric with respect to the involution $\phi_R$. Put the round metric (\ref{roundmetric}) on each chart $T_0, T_\infty$ and then use the partition of unity $\{\psi_0, \psi_\infty\}$ to define a metric on $\Sigma_R$.

Therefore, this metric agrees with the round metric (\ref{roundmetric}) on the outer neck for each chart, $\left\{|z| > \frac{1}{2 \delta R}\right\} \subset T_i$ On the inner neck, the metric is defined by an interpolation of the metrics on each chart using the partition of unity. This metric induces $W^{k,p}$ norms on $\Sigma_R$; the restriction of this norm to a chart $T_i$ is the norm $||\cdot||_{W^{1,p}}$ in Theorem \ref{mainthm}.


In this section we will use $||\cdot||_{W^{k,p}}$ to denote the norm described above on $\Sigma_R$ or the charts $T_i$, where which one we are using will be clear from context. Note that these norms depend on $R$, but this will be suppressed in the notation as it will be clear to which $R$ we are referring.  The key property of this norm on $\Sigma_R$ is that it is symmetric with respect to the charts $T_0,T_\infty$. Therefore, henceforth we will prove estimates by examining one chart $T_0$ and consider $|z|\leq\frac{1}{R}$. Note that while the formulation is different, this norm is the same as that considered by \cite{JHOL}.

Next we define the pregluing $u^R(z)$ of the curves $u^0,u^\infty$. We will now fix the curves $u^0,u^\infty$; the independence of the results on the curves $u^0,u^\infty$ follows from the fact that the set of pairs $(u^0,u^\infty)$ satisfying
\[u^0(0)=u^\infty(0),\quad \|du^0\|_{L^\infty}\leq k,\quad \|u^\infty\|_{L^\infty}\leq k\]
is compact.

We will use the connection
\[\widetilde{\nabla}:=\nabla - \tfrac{1}{2}J(\nabla J)\]
on $(TM,J)$ and $(T^{*}M,J)$, where $\nabla$ is the Levi-Civita connection. The connection $\widetilde{\nabla}$ preserves $(0,1)$ forms and is used for reasons necessary later. However, $\widetilde{\nabla}$ has torsion.

Let $x_0:=u^0(0)=u^\infty(0)$. Define
\[\zeta^0(z),\zeta^\infty(z)\in T_{x_0}M\]
to be the unique tangent vectors of norm less than the injectivity radius of $M$ such that
\begin{equation}\label{zeta0def}
u^0(z) = \exp_{x_0}(\zeta^0(z)),\quad u^\infty(z) = \exp_{x_0}(\zeta^\infty(z))
\end{equation}
where the exponential map is defined using $\widetilde{\nabla}$.
Next fix a cutoff function $\rho:\mathbb{C}\rightarrow[0,1]$ such that
\begin{equation}\label{rho}
\rho(z) = \begin{cases}
0,& |z|\leq 1\\
1,& |z|\geq 2.
\end{cases}
\end{equation}
Then define the pregluing to be
\[u^R: \Sigma_R \rightarrow M\]
\begin{equation}\label{pregluing}
u^R(z):=\begin{cases}
u^{\infty}(z),& z\in\left\{\frac{2}{\delta R}\leq|z|\right\} \subset T_\infty\\
\exp_{x_0}\left(\rho(\delta Rz)\zeta^\infty(z)\right),& z\in\left\{\frac{1}{\delta R}\leq|z|\leq\frac{2}{\delta R}\right\} \subset T_\infty\\
x_0=u^\infty(0),& z\in\left\{\frac{2\delta}{R}\leq|z|\leq\frac{1}{\delta R}\right\} \subset T_\infty\\
x_0=u^0(0),& z\in\left\{\frac{2\delta}{R}\leq|z|\leq\frac{1}{\delta R}\right\} \subset T_0\\
\exp_{x_0}\left(\rho(\delta Rz)\zeta^0(z)\right),& z\in\left\{\frac{1}{\delta R}\leq|z|\leq\frac{2}{\delta R}\right\} \subset T_0\\
u^{0}(z),& z\in\left\{\frac{2}{\delta R}\leq|z|\right\} \subset T_0
\end{cases}\end{equation}
where $T_0,T_\infty$ are the charts for $\Sigma_R$ from (\ref{sigmar}). Again exponentiation is defined using $\widetilde{\nabla}.$\footnote{It is worth noting that \cite{JHOL} define their pregluing using the Levi-Civita connection. However, none of their estimates rely on the connection used being torsion free. The linearized $\bar{\partial}$ operator (as defined in (\ref{linearizeddelbar})) is defined using $\widetilde{\nabla}$, as are other aspect of our problem, so our choice seems the most natural.} Using the terminology of Section \ref{gluingoverview}, note that the pregluing is constant on the neck of $\Sigma_R$; in particular, it is constant on a neighborhood of the inner neck. The inner neck was also exactly where the $W^{1,p}$ norm on $\Sigma_R$ was defined using a partition of unity interpolation.

We will now fix $\delta$ and consider any $R_0>\delta\delta_0$. Note that the pregluing is defined to be symmetric with respect to $u^0,u^\infty$. We will be interested in estimating $\left.\frac{d}{dR}\right|_{R=R_0}\iota^{\delta,R}(u^0,u^\infty)$ on each chart $T_i$. (We are studying an $R$ parameterized families of maps, so the $R$ derivative only makes sense when we fix a parametrization on the domain, that is restrict to a chart.)

We will work on the Banach spaces
\[W^{1,p}_{u^{R_0}}:=W^{1,p}(\Sigma_{R_0},(u^{R_0})^*TM),\quad L^p_{u^{R_0}}:=L^p(\Sigma_{R_0},\Lambda^{0,1}\otimes_J(u^{R_0})^*TM).\]
We have fixed $R_0$, so these will be the fixed Banach spaces on which we apply Proposition \ref{ift} and Lemma \ref{fpderivativebound}.

Finally, we describe the maps $f_t, Q_t$ to which we will apply the implicit function theorem. Given $v\in T_xM$, define the map
\begin{equation}\label{psit}
\Psi_{v,x}:T_xM \rightarrow T_{\textnormal{exp}_x{v}}M
\end{equation}
to be parallel transport along the geodesic defined by $v$; we will sometimes suppress $x$. The geodesic and parallel transport are taken with respect to the Hermitian connection $\widetilde{\nabla}$. Note that $\Psi_{0,x}=$ id$:T_xM\rightarrow T_xM$, so $\Psi_{v,x}$ is uniformly invertible for small $v$.

We now wish to measure the difference between the pregluings $u^{R_0}(z)$ and $u^{R_0+t}(z)$ for $t$ small. However, these curves have different domains, $\Sigma_{R_0}$ and $\Sigma_{R_0 + t}$ respectively. Thus, we first identify these domain curves and then calculate our estimates on a fixed domain. Roughly, away from the inner neck the curves are identified by the identity map and there is an interpolation on the inner neck (where the pregluing is constant). More precisely, for each $t$ choose a diffeomorphism
\[\chi_t: \left\{|z| > \frac{2\delta}{R_0}\right\} \rightarrow \left\{|z| > \frac{2\delta}{R_0 + t}\right\}\]
such that $\chi_t(z) = z$ for $z\in\left\{|z| > \frac{1}{2 \delta R_0}\right\}$. In this way we identify each of the two charts $T_0, T_\infty$ for $\Sigma_{R_0}$ and $\Sigma_{R_0 + t}$ (see (\ref{twocharts})). Note that we choose $\chi_t$ to be the identity map outside of the inner neck $\left\{\frac{2\delta}{R} < |z| < \frac{1}{2\delta R}\right\}\subset T_0, T_\infty$. We can also choose $\chi_t$ such that its derivatives are uniformly bounded for $t$ near $0$. By choosing $\chi_t$ symmetric with respect to $\phi_{R_0}$, the map that identifies the two charts over the neck to form $\Sigma_{R_0}$ (see (\ref{phineck})), $\chi_t$ induces a diffeomorphism
\begin{equation}\label{chidiffeo}
\chi_t: \Sigma_{R_0}\rightarrow \Sigma_{R_0 + t}.
\end{equation}
By construction, in each chart $T_i = \left\{|z| > \frac{2\delta}{R_0}\right\}$, we have $\chi_t(z) = z$ on the region $\left\{|z| > \frac{1}{2\delta R_0}\right\}$ (this includes the outer neck, which is where the pregluing map is potentially nonconstant).

Define $\zeta_t$ to be the unique small vector field along $u^{R_0}(z)$ such that
\begin{equation}\label{zetat}
u^{R_0+t}\big(\chi_t(z)\big) = \exp_{u^{R_0}(z)}(\zeta_t(z)).
\end{equation}
For $t>0$ and small, in the chart $T_0$, we claim that $\zeta_t$ is given by:
\begin{equation}\label{zetat2}
\zeta_t(z) = \begin{cases}
0,& z\in\left\{|z|\geq\frac{2}{\delta R_0}\right\}\subset T_0\\
\Big(\rho\big(\delta(R_0+t)z\big)-\rho(\delta R_0z)\Big)\Psi_{\rho(\delta R_0z)\zeta^0(z),x_0}\big(\zeta^0(z)\big),& z\in\left\{\frac{1}{\delta (R_0+t)}\leq|z|\leq\frac{2}{\delta R_0}\right\}\subset T_0\\
0,& z\in\left\{\frac{1}{R_0}\leq|z|\leq\frac{1}{\delta (R_0+t)}\right\}\subset T_0.
\end{cases}
\end{equation}
To see this, first note that on the region $\left\{\frac{1}{R_0}\leq|z|\leq\frac{1}{\delta (R_0+t)}\right\}$, $u^{R_0}(z) = u^{R_0+t}\big(\chi_t(z)\big) = x_0$. In particular, $\zeta_t(z) = 0$ when $\chi_t(z) \neq z$. Next, we note that on the region $\left\{|z|\geq \frac{1}{\delta R_0}\right\}\subset T_0$ (outside the neck), $\chi_t(z) = z$ and both $u^{R_0}(z)$ and $u^{R_0+t}(z)$ are defined by (\ref{pregluing}) to be exponentiation of a scaling of $\zeta^0(z)$ (see Figure \ref{pregfig}). We then use the general fact that for $\lambda_1,\lambda_2\in\mathbb{R}$ and $v\in T_xM$
\[\exp_x\big((\lambda_1+\lambda_2)v\big) = \exp_{\exp_{x_0}(\lambda_1 v)}\big(\Psi_{v,x}(\lambda_2v)\big).\]
We are also using the fact that the tangent vector field of the geodesic $s\mapsto\exp_{x_0}\big(\zeta^0(z)s\big)$ is also parallel with respect to $\widetilde{\nabla}$ and hence $\Psi_{\rho(\delta R_0z)\zeta^0,x_0}(\zeta^0(z))$ is tangent to the geodesic. In particular, (\ref{zetat2}) shows that the size of $\zeta_t$ changes with $t$, but its direction does not. A similar description can be given for $t<0$ (and also on $T_\infty$ by symmetry).

\begin{figure}
\begin{center}
\includegraphics[scale=0.5]{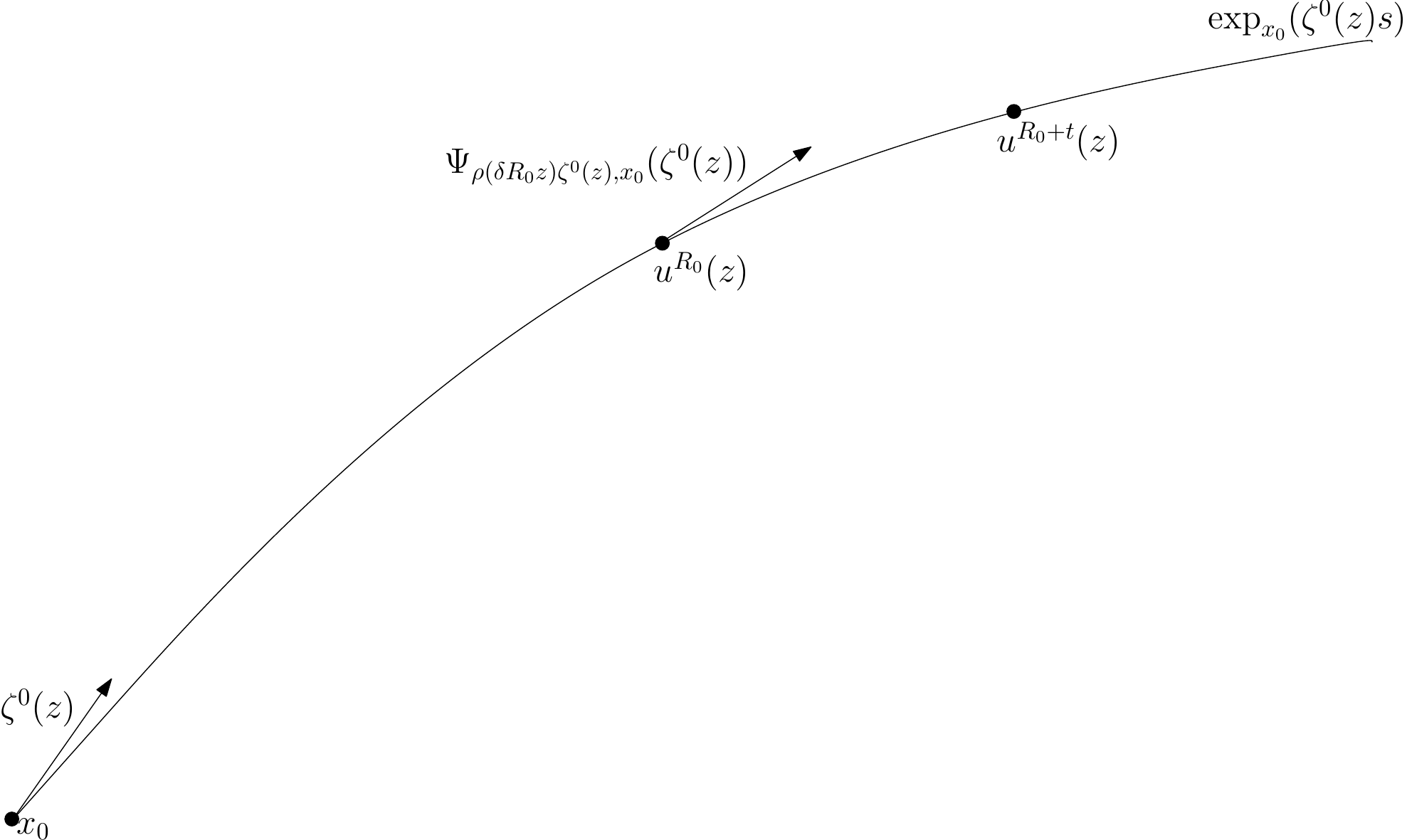}
\caption{The pregluings $u^{R_0}$ and $u^{R_0+t}$ for $\frac{1}{\delta (R_0+t)}\leq|z|\leq\frac{2}{\delta R_0}$}\label{pregfig}
\end{center}
\end{figure}

Let
\begin{equation}\label{phit}
\Phi_t:T_{u^{R_0}(z)}M\rightarrow T_{u^{R_0+t}\big(\chi_t(z)\big)}M
\end{equation}
denote parallel transport along the geodesic $\zeta_t$ with respect to the connection $\widetilde{\nabla}$. In terms of the notation above, $\Phi_t=\Psi_{\zeta_t,u^{R_0}(z)}$.

Recall that if $\chi_t(z) \neq z$, then $z$ is in the inner neck. However, the pregluing is constant in the inner neck, so $u^{R_0}(z) = u^{R_0+t}\left(\chi_t(z)\right)$. In particular, on the inner neck
\begin{equation}\label{phioninner}
\Phi_t : T_{x_0}M \rightarrow T_{x_0}M
\end{equation}
is the identity map.

We will abuse notation by always using $\Phi_t$ regardless of whether its input is $W^{1,p}$ vector fields or $L^p$ $(0,1)$ forms.

We finally give the family of maps $f_t, Q_t$ to which we apply Proposition \ref{ift}, the implicit function theorem. Define the function
\begin{equation}\label{ft}
f_t = \Phi_t^{-1}\circ \mathcal{F}_{u^{R_0+t}}\circ \Phi_t:W^{1,p}_{u^{R_0}}\rightarrow L^p_{u^{R_0}}
\end{equation}
where
\begin{equation}\label{nonlineardelbar}
\mathcal{F}_{u^{R}}(\xi) = \Psi_\xi^{-1}\left(\bar{\partial}_J(\exp_{u^R}(\xi))\right) : W^{1,p}_{u^{R_0}}\rightarrow L^p_{u^{R_0}}.
\end{equation}
(This is exactly the operator considered in \cite{JHOL}.) In particular, in the inner neck
\begin{equation}\label{nonlineardelbarinner}
\mathcal{F}_{u^{R}}(\xi) = \Psi_\xi^{-1}\left(\bar{\partial}_J(\exp_{x_0}(\xi))\right).
\end{equation}

By \cite[Remark 6.7.3]{JHOL},
\begin{equation}\label{linearizeddelbar}
\mathcal{F}_{u^R}(0) = \bar{\partial}_J(u^R),\quad d(\mathcal{F}_{u^R}(0)) = D_{u^R}
\end{equation}
where $D_u$ is the linearized Cauchy-Riemann operator. Similarly, we define
\begin{equation}\label{qt}
Q_t = \Phi_t^{-1}\circ Q_{u^{R_0+t}}\circ \Phi_t:L^p_{u^{R_0}}\rightarrow W^{1,p}_{u^{R_0}}
\end{equation}
where $Q_{u^R}$ is the bounded right inverse for $D_{u^R}$ constructed in \cite[$\S 10.5$]{JHOL}.

In the next section we will verify the hypotheses of Proposition \ref{ift} are satisfied for $f_t, Q_t$. However, now note that if $f_t(\xi)=0$, then $\exp_{u^{R_0+t}(z)}(\Phi_t(\xi))$ is exactly $\iota^{\delta,R_0+t}(u^0,u^\infty)$. Therefore, roughly speaking, the family of maps $f_t$ on the fixed space $W^{1,p}_{u^{R_0}}$ gives the gluing map for $u^{R_0 + t}$. We will make this more precise in the next section.

\subsection{Hypotheses of the Implicit Function Theorem}\label{asd}
In this section we verify that the hypotheses of Proposition \ref{ift} are satisfied for the functions $f_t$ and $Q_t$ defined in (\ref{ft}) and (\ref{qt}). As mentioned in the last section, first note that if $f_t(\xi)=0$, then $\exp_{u^{R_0+t}(z)}(\Phi_t(\xi))$ is exactly $\iota^{\delta,R_0+t}(u^0,u^\infty)$.

Now $d_0f_t$, the derivative of $f_t$ at $\xi=0$, is given by $d_0f_t = \Phi_t^{-1}\circ D_{u^{R_0+t}}\circ\Phi_t$ because $\Phi_t$ is linear and $d_0\Phi_t=$ id. The almost complex structure $J$ was assumed to be regular, so $d_0f_t$ is surjective. This verifies condition $(a)$ of Proposition \ref{ift}.

Also, $Q_t$ is a right inverse of $d_0f_t$ because $Q_{u^R}$ is a right inverse for $D_{u^R}$. This is verifies condition $(b)$.

It is proved in \cite[Proposition 3.5.3]{JHOL}, that for $p>2$ and for every constant $c_0>0$, there exists $c_1>0$ such that
\[\|du\|_{L^p}\leq c_0,\quad\|\xi\|_{L^\infty}\leq c_0\quad\Longrightarrow\quad \|d_\xi f_0-d_0f_0\|\leq c_1\|\xi\|_{W^{1,p}}\]
where $\|\cdot\|$ denotes operator norm. Choose $c_0 > \|Q_{u^{R_0}}\|$. \cite[Equation (10.5.8)]{JHOL} asserts that $\|Q_{u^{R_0}}\|$ is bounded independent of $R_0$ and hence the same is true for $\|Q_t\|$ with $t$ sufficiently small. It follows that choosing $\delta$ such that $c_1\delta < \frac{1}{2c_0}$ and $t$ sufficiently small,
\[\|\xi\|_{W^{1,p}}\leq\delta\quad\Longrightarrow\quad\|d_\xi f_t - d_0f_t\|_{W^{1,p}}\leq\frac{1}{2\|Q_t\|}.\]

\cite[Lemma 10.3.2]{JHOL} states that there is a constant $c_0=c_0(c,p)$ such that
\[\|\bar{\partial}_J(u^R)\|_{L^p}\leq c_0(\delta R)^{-2/p}.\]
Hence if $R_0$ is chosen large enough, then
\[\|f_0(0)\| = \left\|\bar{\partial}_J(u^{R_0})\right\|<\frac{\delta}{4\|Q_{u^{R_0}}\|}.\]
Using this, we see that
\[\|f_t(0)\|\leq\frac{\delta}{4\|Q_t\|}\]
for $t$ sufficiently small. This verifies the final condition $(c)$ of Proposition \ref{ift}.

Therefore, the hypotheses of Proposition \ref{ift} are satisfied with $c=c_0$ and we get $\xi_t^*\in W^{1,p}(\Sigma_{R_0},(u^{R_0})^*TM)$ such that $f_t(\xi_t^*) = 0$. Moreover, we get
\begin{equation}\label{sizeofxi}
\|\xi_t^*\|_{W^{1,p}}\leq 2c_0\|f_t(0)\| \leq 2c_0^2(\delta(R_0+t))^{-2/p}.
\end{equation}
As mentioned at the beginning of this section, we have $\exp_{u^{R_0+t}(z)}(\Phi_t(\xi_t^*))=\iota^{\delta,R_0+t}(u^0,u^\infty)$.

We are interested in understanding the derivative of this map in a chart, that is
\[\left.\frac{d}{dt}\right|_{t=0}\iota_i^{\delta,R_0+t}(u^0(z),u^\infty(z)) := \left.\frac{d}{dt}\right|_{t=0}\left(\left.\iota^{\delta,R_0+t}(u^0(z),u^\infty(z))\right|_{T_i}\right).\]
This map is known to be smooth in the gluing variable (c.f. \cite[Theorem 10.1.2]{JHOL}).

\begin{remark}\cite[Remark 10.5.5]{JHOL}\label{E}
Let $M$ be a compact Riemannian manifold. Then there are two smooth families of endomorphisms $E_1(x,\xi),E_2(x,\xi):T_xM\rightarrow T_{\exp_x(\xi)}M$ defined by
\[\frac{d}{dt}\exp_{\gamma(t)}(v(t)) = E_1(\gamma(t),v(t))\dot{\gamma} + E_2(\gamma(t),v(t))\widetilde{\nabla}_tv(t).\]
$E_j(x,0) =$ id, so $E_j(x,\xi)$ is uniformly invertible for sufficiently small $\xi$.\hfill$\Diamond$
\end{remark}

\begin{proposition}
Theorem \ref{mainthm} holds if
\begin{equation}\label{propeqn1}
\left\|\left.\frac{d}{dt}\right|_{t=0}u^{R_0+t}_i(z)\right\|_{W^{1,p}}\leq C\frac{1}{R_0}\frac{1}{(\delta R_0)^{2/p}}
\end{equation}
and
\begin{equation}\label{propeqn2}
\left\|\left.\frac{d}{dt}\right|_{t=0}(\xi_t^*)|_{T_i}\right\|_{W^{1,p}}\leq C\frac{1}{R_0}\frac{1}{(\delta R_0)^{2/p}}
\end{equation}
where $C$ depends only on $p,M,J,\omega$. Here $u^{R_0+t}_i$ denotes the restriction of $u^{R_0+t}$ to the chart $T_i$ and $\|\cdot\|_{W^{1,p}}$ is restriction of the norm on $\Sigma_{R_0}$ described in Section \ref{nsm}.
\end{proposition}
\begin{proof}
Using the notation of Remark \ref{E}, we see that
\begin{equation}
\begin{split}
\left.\frac{d}{dt}\right|_{t=0}\iota_i^{\delta,R_0+t}(u^0,u^\infty)(z) &= E_1(u^{R_0}_i(z),\xi_0^*)\left(\left.\frac{d}{dt}\right|_{t=0}u_i^{R_0+t}(z)\right) \\
           &+ E_2(u_i^{R_0}(z),\xi_0^*)\left.\widetilde{\nabla}_t\right|_{t=0}(\Phi_t(\xi_t^*)).
\end{split}
\end{equation}
Note that $E_j(x,\xi)$ is uniformly invertible for $\xi$ sufficiently small. Also, taking $R_0$ to be large enough, $\xi_0^*$ can be made arbitrarily small. So Theorem \ref{mainthm} holds if (\ref{propeqn1}) holds and
\begin{equation}\label{termtwo}
\left\|\left.\widetilde{\nabla}_t\right|_{t=0}(\Phi_t(\xi_t^*))\right\|_{W^{1,p}}\leq C\frac{1}{R_0}\frac{1}{(\delta R_0)^{2/p}}.
\end{equation}
For this estimate, we will use the following useful remark.
\begin{remark}\label{parallelofphi}
Observe that $\left.\widetilde{\nabla}_t\right|_{t=0}(\Phi_t(x))=0$ because as noted earlier $(\ref{zetat2})$ shows that the vector $\zeta_t$ from (\ref{zetat}) used to define to $\Phi_t$ does not change direction and hence defines the same geodesic for all $t$. Thus, $\Phi_t$ is parallel transport along the same geodesic for all $t$. Moreover, also as noted earlier, this geodesic is parallel with respect to the connection $\widetilde{\nabla}$ used to defined $\Phi_t$. Hence its covariant derivative is 0.\hfill$\Diamond$
\end{remark}
By Remark \ref{parallelofphi} and the fact that $d_0\Phi_t=$ id, we see that in the chart $T_i$,
\begin{equation}\label{abc}
\left.\widetilde{\nabla}_t\right|_{t=0}(\Phi_t(\xi_t^*)) = \left.\frac{d}{dt}\right|_{t=0}(\xi_t^*)|_{T_i}.
\end{equation}
This proves the proposition.
\end{proof}
\subsection{Derivative estimates}\label{derivativeestimates}
In this section we will fix $u^0,u^\infty$ and prove the estimates (\ref{propeqn1}) and (\ref{propeqn2}); this in turn proves Theorem \ref{mainthm}. The following are the two key lemmas.
\begin{lemma}\label{w1pur}
Let $u^R(z):\Sigma_R\rightarrow M$ denote the pregluing of $u^0$ and $u^\infty$ as defined in (\ref{pregluing}) using the cutoff function $\rho$ from (\ref{rho}). Let $u^{R_0}_i$ denote the restriction of $u^{R_0}$ to the chart $T_i$. Then
\[\left\|\left.\frac{d}{dt}\right|_{t=0}u_i^{R_0+t}(z)\right\|_{W^{1,p}} \leq C\frac{1}{R_0}\frac{1}{(\delta R_0)^{2/p}}\]
where $C$ depends only on $\rho$ and the metric on $M$.
\end{lemma}
\begin{lemma}\label{mainlemma}
Let $f_t:W^{1,p}_{u^{R_0}}\rightarrow L^p_{u^{R_0}}$ be the map from (\ref{ft}) and let $d_0f_t:W^{1,p}_{u^{R_0}}\rightarrow L^p_{u^{R_0}}$ denote its derivative at $0$. Let $Q_t$ be the right inverse of $d_0f_t$ from (\ref{qt}). Let $\xi_0^{*}\in C^\infty(\Sigma_{R_0},(u^{R_0})^{*}TM)$ denote the vector field along $u^{R_0}$ given by Proposition \ref{ift} applied to the map $f_0 = \mathcal{F}_{u^{R_0}}$ so  $f_0(\xi_0^*)=0$. Then
\[\left\|\left.\frac{d}{dt}\right|_{t=0}\left.\Big(-Q_tf_t(\xi_0^*) + Q_t\big(d_0f_t(\xi_0^*)\big)\Big)\right|_{T_i}\right\|_{W^{1,p}}\leq C\frac{1}{R_0}\frac{1}{(\delta R_0)^{2/p}}.\]
\end{lemma}
Lemma \ref{w1pur} exactly provides the estimate (\ref{propeqn1}). For (\ref{propeqn2}), recall that Proposition \ref{ift} constructs $\xi_t^*$ as a fixed point of the contraction mapping
\[-Q_tf_t(x)+Q_t(d_0f_t(x)).\]
Combining Lemmas \ref{mainlemma} and \ref{fpderivativebound}, we see that (\ref{propeqn2}) holds.

The rest of this section is devoted to proofs of Lemmas \ref{w1pur} and \ref{mainlemma}.

\begin{proof}[Proof of Lemma \ref{w1pur}]
Let $t>0$ and restrict to the chart $T_0$ so all of our analysis will be done on the region $\{|z|\geq1/R_0\}\subset T_0$. (The same analysis holds for $t<0$.) We can restrict to $T_0$ by the symmetry of our construction. The only region in which $\left.\frac{d}{dt}\right|_{t=0}u_0^{R_0+t}(z)$ is nonzero in the annulus $\frac{1}{\delta R_0}\leq|z|\leq\frac{2}{\delta R_0}$ (recall $u_0^{R_0}$ is the restriction of the pregluing $u^{R_0}$ to the chart $T_0$.) On this annulus,
\begin{align}
\nonumber\left.\frac{d}{dt}\right|_{t=0}u_0^{R_0+t}(z) &= \left.\frac{d}{dt}\right|_{t=0}\bigg(\exp_{x_0}\Big(\rho\big(\delta (R_0+t)z\big)\enskip\zeta^0(z)\Big)\bigg)\\
\nonumber                                            &=\left(d_{\rho(\delta R_0z)\zeta^0(z)}\exp_{x_0}\right)\left(\left.\frac{d}{dt}\right|_{t=0}\rho(\delta R_0z+\delta tz)\zeta^0(z)\right)\\
\label{pregluingeq}                                            &=\left(d_{\rho(\delta R_0z)\zeta^0(z)}\exp_{x_0}\right)\Big(\delta\big(z\cdot\nabla\rho(\delta R_0z)\big)\zeta^0(z)\Big)
\end{align}
where in the last line we are thinking of $\delta$ as a constant, $z,\nabla\rho(\delta R_0z)\in\mathbb{C}=\mathbb{R}^2$, and $\cdot$ denotes the usual dot product.
So
\begin{equation}\label{5411}
\left\|\left.\frac{d}{dt}\right|_{t=0}u_0^{R_0+t}(z)\right\|_{L^\infty}  \leq C_1\|\delta\left(z\cdot\nabla\rho(\delta R_0z)\right)\zeta^0(z)\|_{L^\infty}
\end{equation}
where $C_1$ depends only on the metric. We now make some observations that will be useful in estimating the right hand side of (\ref{5411}) and in subsequent estimates.
\begin{lemma}\label{zetaob}
Let $\zeta^0(z)\in T_{x_0}M$ denote the vector defined by (\ref{zeta0def}). Let $\rho:\mathbb{C}\rightarrow\mathbb{C}$ denote the cutoff function from (\ref{rho}). Below, by \textit{constant} we mean a constant independent of $z,\rho,\zeta^0,\delta,R_0$. Then on the annulus $\frac{1}{\delta R_0}\leq|z|\leq\frac{2}{\delta R_0}$ the following hold.
\begin{enumerate}[(1)]
\item\label{ob1} The size of $\zeta^0(z)$ is bounded by a constant times $\frac{1}{\delta R_0}$.
\item\label{ob2} The size of the first derivatives of $\zeta^0(z)$ are uniformly bounded
\end{enumerate}
\end{lemma}
\begin{proof}
For \textit{(\ref{ob1})} note that $\zeta^0(0) = 0$ and $|z|\leq\frac{2}{\delta R_0}$. Statement \textit{(\ref{ob2})} holds simply by the smoothness of $\zeta^0(z)$.
\end{proof}
We can now use Lemma \ref{zetaob} to see
\begin{equation}\label{5412}
\|\delta\left(z\cdot\nabla\rho(\delta R_0z)\right)\zeta^0(z)\|_{L^\infty} \leq C_2\frac{1}{R_0}
\end{equation}
because $\nabla\rho$ is uniformly bounded and $\delta$ and the $\frac{1}{\delta R_0}$ cancel out to leave just a factor of $\frac{1}{R_0}$ and $C_2$ is independent of $\delta$. Therefore combining (\ref{5411}) and (\ref{5412}) and the fact that the integral for this estimate is over a region of area at most $\frac{4\pi}{(\delta R_0)^{2}}$, we find
\[\left\|\left.\frac{d}{dt}\right|_{t=0}u_0^{R_0+t}(z)\right\|_{L^p} \leq C_3\frac{1}{R_0}\frac{1}{(\delta R_0)^{2/p}}.\]

We now need to estimate the $L^p$ norm of the first derivatives of $\left.\frac{d}{dt}\right|_{t=0}u_0^{R_0+t}(z)$ in $z$ directions. We take the same approach as earlier: We uniformly bound and then integrate over the disc $|z|\leq\frac{2}{\delta R_0}$. By (\ref{pregluingeq}),
\[\left.\frac{d}{dt}\right|_{t=0}u_0^{R_0+t}(z)=\left(d_{\rho(\delta R_0z)\zeta^0(z)}\exp_{x_0}\right)\Big(\delta\big( z\cdot\nabla\rho(\delta R_0z)\big)\zeta^0(z)\Big).\]
We will briefly consider a general situation. We define the function
\[G(v,w) = (d_v\exp_x)(w),\qquad\qquad x\in M, v\in T_xM,w\in T(T_xM).\]
Note that
\[G(v,\cdot)\textnormal{ is linear}\quad\textnormal{and}\quad G(0,w) = \textnormal{id}.\]
Hence $G(v,w)$ is uniformly invertible for small $v$. We need to take the derivative of $G(v,w)$ with respect to $s_1,s_2$ where $z=(s_1,s_2)\in\mathbb{R}^2$ and both $v$ and $w$ depend on $z$. Using the chain rule we see that
\begin{equation}\label{derivativeofexp}
\frac{\partial}{\partial s_1}G(v,w) = \frac{\partial G}{\partial v}(v,w)\left(\frac{\partial v}{\partial s_1}\right) + \frac{\partial G}{\partial w}(v,w)\left(\frac{\partial w}{\partial s_1}\right)
\end{equation}
where $\frac{\partial}{\partial v},\frac{\partial}{\partial w}$ denotes differentials in those variables.
We now specialize to the case we are interested in, namely
\[v=\rho(\delta R_0z)\zeta^0(z),\quad w=\delta\big( z\cdot\nabla\rho(\delta R_0z)\big)\zeta^0(z).\]

First let us consider first derivatives of $v$. We claim they are uniformly bounded. They are a sum of two terms. One is first derivatives of $\rho(\delta R_0z)$ (which grow like $\delta R_0)$ multiplied by $\zeta^0(z)$. By Lemma \ref{zetaob}, $\zeta^0(z)$ is bounded by a constant times $\frac{1}{\delta R_0}$. Hence these first type terms are uniformly bounded. The second type of terms are $\rho(\delta R_0z)$ multiplied by derivatives of $\zeta^0(z)$. By Lemma \ref{zetaob} these are uniformly bounded as well.

Next consider first derivatives of $w$. We claim they are bounded by a constant times $1/R_0$. They can also be expressed as a sum of two terms. One is first derivatives of $\delta z\cdot\nabla\rho(\delta R_0z)$ multiplied by $\zeta^0(z)$. The first derivatives of $\delta z\cdot\nabla\rho(\delta R_0z)$ are bounded by a constant times $\delta$. By Lemma \ref{zetaob}, $\zeta^0(z)$ is bounded by a constant times $\frac{1}{\delta R_0}$. Therefore, the first type of terms are bounded by a constant times $1/R_0$. The second term is $\delta z\cdot\nabla\rho(\delta R_0z)$ multiplied by the first derivatives of $\zeta^0(z)$. Then $\delta z\cdot\nabla\rho(\delta R_0z)$ is bounded by a constant times $1/R_0$ because $|z|\leq\frac{2}{\delta R_0}$ and the first derivatives of $\zeta^0(z)$ are uniformly bounded by Lemma \ref{zetaob}. So the second type of term is also bounded by a constant times $1/R_0$. Therefore the first derivatives of $\left(\delta z\cdot\nabla\rho(\delta R_0z)\right)\zeta^0(z)$ are bounded by  a constant times $\frac{1}{R_0}$.

Finally, note that $v$ is bounded by a constant times $1/R_0$ by Lemma \ref{zetaob}.

Combining the above facts and the fact that when $z=0$, we have $v=w=0$ and hence
\[\frac{\partial G}{\partial v}(v(0),w(0)) = 0,\] we see that first derivatives are uniformly bounded by $1/R_0$.
Again we integrate over a region of area at most $\frac{4\pi}{(\delta R_0)^{2}}$. This proves the lemma.
\end{proof}
\begin{proof}[Proof of Lemma \ref{mainlemma}]
As in the proof of Lemma \ref{w1pur}, we will restrict to the chart $T_0$ and the region $\{|z| \geq 1/R_0\} \subset T_0$, which can be done by symmetry. We will omit the $|_{T_0}$ throughout this proof for simplicity of notation. Applying the chain rule, we get
\begin{align*}
&\left.\frac{d}{dt}\right|_{t=0}\Big(-Q_tf_t(\xi_0^*) + Q_t\big(d_0f_t(\xi_0^*)\big)\Big)\\
=&-\left.\frac{d}{dt}\right|_{t=0}Q_t(f_0(\xi_0^*)) - Q_0\left(\left.\frac{d}{dt}\right|_{t=0}\big(f_t(\xi_0^*)\big)\right) \\&+ \left.\frac{d}{dt}\right|_{t=0}\Big(Q_t\big(d_0f_0(\xi_0^*)\big)\Big)+ Q_0\left(\left.\frac{d}{dt}\right|_{t=0}\big(d_0f_t(\xi_0^*)\big)\right)\\
=&-Q_{u^{R_0}}\left(\left.\frac{d}{dt}\right|_{t=0}\big(f_t(\xi_0^*)\big)\right) + \left.\frac{d}{dt}\right|_{t=0}\Big(Q_t\big(D_{u^{R_0}}(\xi_0^*)\big)\Big) + Q_0\left(\left.\frac{d}{dt}\right|_{t=0}\big(d_0f_t(\xi_0^*)\big)\right)\numberthis\label{threeterms}
\end{align*}
where the second equality uses the fact that $f_0(\xi_0^*)=0$ by definition of $\xi_t^*$ and also uses the definitions of $f_t$ and $Q_t$ from (\ref{ft}) and (\ref{qt}).

Recall from (\ref{phioninner}) that on the inner neck (where the charts $T_0$ and $T_\infty$ are identified) the map $\Phi_t$ used to define $f_t, Q_t$ is the identity. Thus,
\[f_t(\xi) = \Psi_{\xi}^{-1}(\bar{\partial}\exp_{x_0}(\xi))\]
as in (\ref{nonlineardelbar}) and (\ref{nonlineardelbarinner}). Hence, $d_0f_t$ is the usual linearized Cauchy-Riemann operator from (\ref{linearizeddelbar}) and $Q_t$ is its right inverse.

The rest of this section will be estimating the three terms of (\ref{threeterms}). Each term will be estimated in a step below.\\

\noindent\underline{Step 1}: We first focus on the first term of (\ref{threeterms}). As remarked in Section \ref{asd}, $\|Q_{u^{R_0}}\|$ is bounded independently of $R_0$. Thus, we focus on
\[\left\|\left.\frac{d}{dt}\right|_{t=0}(f_t(\xi_0^{*}))\right\|_{L^p}.\]
Using Remark \ref{parallelofphi}, we see that
\begin{equation}\label{first}
\left\|\left.\frac{d}{dt}\right|_{t=0}(f_t(\xi_0^{*}))\right\|_{L^p}=\left\|\left.\widetilde{\nabla}_t\right|_{t=0}\left(\Psi_{\Phi_t(\xi_0^{*})}^{-1}\Big(\bar{\partial}_J\big(\exp_{u^{R_0+t}}(\Phi_t(\xi_0^{*}))\big)\Big)\right)\right\|_{L^p}
\end{equation}
where $\Phi_t$ is defined in (\ref{phit}).

\begin{remark}\label{Psi}
Let $\Psi$ be the map defined in (\ref{psit}). Covariantly differentiate the identity
\[\Psi_{v_t}\circ\Psi_{v_t}^{-1}(w_t) = w_t\]
and use the fact that $\Psi_v$ is uniformly invertible for small $v$ to see that
\[\left\|\left.\widetilde{\nabla}_t\right|_{t=0}\big(\Psi_{v_t}^{-1}(w_t)\big)\right\| \leq K\left(\left\|\left.\widetilde{\nabla}_t\right|_{t=0}w_t\right\| + \left\|\Psi_{\left.\frac{d}{dt}\right|_{t=0}v_t}(w_0)\right\|\right)\hfill\]
for any norm $\|\cdot\|$. \hfill$\Diamond$
\end{remark}

In order to estimate (\ref{first}), we can apply Remark \ref{Psi} with
\[v_t=\Phi_t(\xi_0^{*}),\quad w_t=\bar{\partial}\left(\exp_{u^{R_0+t}(z)}(\Phi_t(\xi_0^{*}))\right),\quad \|\cdot\|=\|\cdot\|_{L^p}.\]
First,
\[\Psi_{\left.\frac{d}{dt}\right|_{t=0}v_t}(w_0) = \Psi_{\left.\frac{d}{dt}\right|_{t=0}v_t}\left(\bar{\partial}\left(\exp_{u^{R_0}(z)}(\xi_0^{*})\right)\right) = 0\]
by the definition of $\xi_0^{*}$. Therefore, it suffices to estimate
\[\left\|\left.\widetilde{\nabla}_t\right|_{t=0}w_t\right\|_{L^p}=\left\|\left.\frac{d}{dt}\right|_{t=0}\left(\bar{\partial}\left(\exp_{u^{R_0+t}(z)}(\Phi_t(\xi_0^{*}))\right)\right)\right\|_{L^p}.\]
Using the notation of Remark \ref{E}, this is
\begin{align*}
&D_{u^{R_0}(z)}\left(E_1(u^{R_0}(z),\xi_0^{*})\left(\left.\frac{d}{dt}\right|_{t=0}u^{R_0+t}(z)\right) + E_2(u^{R_0}(z),\xi_0^{*})\left.\widetilde{\nabla}_t\right|_{t=0}(\Phi_t(\xi_0^{*}))\right)\\
=&D_{u^{R_0}(z)}\left(E_1(u^{R_0}(z),\xi_0^{*})\left(\left.\frac{d}{dt}\right|_{t=0}u^{R_0+t}(z)\right)\right)
\end{align*}
where the second equality uses Remark \ref{parallelofphi}. As noted in Remark \ref{E}, the map $E_1(x,\xi)$ is uniformly invertible for $\xi$ sufficiently small and hence this is the case for $\xi_0^*$ if $R_0$ is chosen large enough. Lemma \ref{w1pur} provides the estimate for $\left\|\left.\frac{d}{dt}\right|_{t=0}u^{R_0+t}(z)\right\|_{W^{1,p}}$. Next, recall the following formula for the operator $D_u$ from \cite[Proposition 3.1.1]{JHOL}:
\begin{equation}\label{equationford}
D_u\xi = \frac{1}{2}\left(\nabla\xi + J(u)\nabla\xi\circ j_\Sigma\right) - \frac{1}{2}J(u)(\nabla_\xi J)(u)\partial_J(u).
\end{equation}
By \cite[Lemma 10.3.2]{JHOL}, $\left\|u^R(z)\right\|_{W^{1,\infty}}$ is bounded independently of $R$. Thus, the operator norm of $D_{u^{R}}$ is also bounded independently of $R$. Therefore,
\[\left\|\left.\frac{d}{dt}\right|_{t=0}\left(\bar{\partial}\left(\exp_{u^{R_0+t}(z)}(\Phi_t(\xi_0^*))\right)\right)\right\|_{L^p}\leq C'\left\|\left.\frac{d}{dt}\right|_{t=0}u^{R_0+t}(z)\right\|_{W^{1,p}}\leq C_4\frac{1}{R_0}\frac{1}{(\delta R_0)^{2/p}}.\]
Thus, the first term of (\ref{threeterms}) has the desired bound, that is
\begin{equation}\label{firstterm}
\left\|Q_{u^{R_0}}\left(\left.\frac{d}{dt}\right|_{t=0}\big(f_t(\xi_0^*)\big)\right)\right\|_{W^{1,p}} \leq C\frac{1}{R_0}\frac{1}{(\delta R_0)^{2/p}}.
\end{equation}

\noindent\underline{Step 2}: Next, we bound the $W^{1,p}$ norm of the second term of (\ref{threeterms}).
\begin{align*}
\left.\frac{d}{dt}\right|_{t=0}\Big(Q_t\big(D_{u^{R_0}}(\xi_0^*)\big)\Big) &= \left.\frac{d}{dt}\right|_{t=0}\Big(\Phi_t^{-1}\circ Q_{u^{R_0+t}}\circ \Phi_t\big(D_{u^{R_0}}(\xi_0^*)\big)\Big)\\
&=\left.\widetilde{\nabla}_t\right|_{t=0}\Big(Q_{u^{R_0+t}}\circ \Phi_t\big(D_{u^{R_0}}(\xi_0^*)\big)\Big)
\end{align*}
Examining the proof of \cite[Lemma 10.6.2]{JHOL}, we see that
\[\left\|\left.\widetilde{\nabla}_t\right|_{t=0}\Big(Q_{u^{R_0+t}}\circ \Phi_t\big(D_{u^{R_0}}(\xi_0^*)\big)\Big)\right\|_{W^{1,p}} \leq C_5\left\|\left.\frac{d}{dt}\right|_{t=0}u^{R_0+t}\right\|_{W^{1,\infty}}\|\Phi_t(D_{u^{R_0}}(\xi_0^*))\|_{L^p}.\]
Lemma \ref{w1pur} was proved exactly by showing
\[\left\|\left.\frac{d}{dt}\right|_{t=0}u^{R_0+t}\right\|_{W^{1,\infty}}\leq C_2\frac{1}{R_0}.\]
Finally we see
\[\|\Phi_t(D_{u^{R_0}}(\xi_0^*))\|_{L^p} \leq C_6\|\xi_0^*\|_{W^{1,p}}\leq C_7(\delta R_0)^{-2/p}\]
where $C_6,C_7$ are constants independent of $R_0$; the second inequality uses (\ref{sizeofxi}). Putting the last two equations together we see that the second term of (\ref{threeterms}) has the desired bound, that is
\begin{equation}\label{secondterm}
\left\|\left.\frac{d}{dt}\right|_{t=0}\Big(Q_t\big(D_{u^{R_0}}(\xi_0^*)\big)\Big)\right\|_{W^{1,p}}\leq C_8\frac{1}{R_0}\frac{1}{(\delta R_0)^{2/p}}.
\end{equation}

\noindent\underline{Step 3}: Finally, we bound the $W^{1,p}$ norm of the third term of (\ref{threeterms}). Again, $\|Q_{u^{R}}\|$ is bounded independently of $R$, so we need to bound $\left\|\left.\frac{d}{dt}\right|_{t=0}(d_0f_t(\xi_0^*))\right\|_{L^p}.$
\begin{align*}
&\left.\frac{d}{dt}\right|_{t=0}(d_0f_t(\xi_0^*))\\ =& \left.\frac{d}{dt}\right|_{t=0}\Big(\Phi_t\circ D_{u^{R_0+t}}\big(\Phi_t(\xi_0^*)\big)\Big)\\
=&\left.\widetilde{\nabla}_t\right|_{t=0}\Big(D_{u^{R_0+t}}\big(\Phi_t(\xi_0^*)\big)\Big)\\
=&\frac{1}{2}\left.\widetilde{\nabla}_t\right|_{t=0}\left(\nabla(\Phi_t(\xi_0^*)) + J(u^{R_0+t})\nabla\Phi_t(\xi_0^*)\circ j_{\Sigma} - J(u^{R_0+t})(\nabla_{\Phi_t(\xi_0^*)}J)(u^{R_0+t})\partial_J(u^{R_0+t})\right).
\end{align*}
This gives us three summands we need to bound. The relevant quantities we need to bound are
\begin{equation}\label{del}
\left\|\left.\widetilde{\nabla}_t\right|_{t=0}\partial_J(u^{R_0+t})\right\|_{L^p}
\end{equation}
\begin{equation}\label{del3}
\left\|\left.\widetilde{\nabla}_t\right|_{t=0}\nabla\big(\Phi_t(\xi_0^*)\big)\right\|_{L^p}.
\end{equation}
We first focus on (\ref{del}). This is supported on the annulus $\frac{1}{\delta R_0}\leq|z|\leq\frac{2}{\delta R_0}$ where
\begin{equation}\label{del2}
\partial_J\left(u^{R_0+t}\right) = \partial_J\Big(\exp_x\big(\rho\big(\delta(R_0+t)z\big)\zeta^0(z)\big)\Big).
\end{equation}
Similar to the proof of Lemma \ref{w1pur}, we bound the $L^p$ norm in (\ref{del}) by uniformly bounding and then integrating. The right hand side of (\ref{del2}) can be expressed as a sum of two terms. The first involves the first derivatives of $\rho(\delta (R_0+t) z)$ \big(which grow like $\delta (R_0+t)$\big) multiplied by $\zeta_0(z)$. Taking the derivative of this with respect to $t$, we see that these derivatives are uniformly bounded by a constant times $\delta$ times the size of $\zeta^0(z)$. By Lemma \ref{zetaob}, the size of $\zeta^0(z)$ is bounded by a constant times $1/\delta R_0$. Therefore the first terms are bounded by a constant times $1/R_0$. The second term involves $\rho\big(\delta(R_0+t)z\big)$ multiplied by the first derivatives of $\zeta^0(z)$. Similar to the proof of Lemma \ref{w1pur}, the derivative of such terms with respect to $t$ is uniformly bounded by a constant times $\frac{1}{R_0}$. Therefore, all terms are bounded by a constant times $\frac{1}{R_0}$ and since we are integrating over a region of area at most $\frac{4\pi}{(\delta R_0)^2}$, we see that the $L^p$ norm $(\ref{del})$ is bounded by $\frac{1}{R_0}\frac{1}{(\delta R_0)^{2/p}}$. The same bound holds for $(\ref{del2})$ by similar reasoning. Therefore, the $W^{1,p}$ norm of the third term of (\ref{threeterms}) has the desired bound, that is
\begin{equation}\label{thirdterm}
\left\|Q_0\left(\left.\frac{d}{dt}\right|_{t=0}\big(d_0f_t(\xi_0^*)\big)\right)\right\|_{W^{1,p}}\leq C_9\frac{1}{R_0}\frac{1}{(\delta R_0)^{2/p}}.
\end{equation}

We finally, note that we worked entirely in the chart $T_0$, but exactly the same argument applies to $T_\infty$ by the symmetry of our construction. Therefore together, (\ref{firstterm}), (\ref{secondterm}), and (\ref{thirdterm}) prove Lemma \ref{mainlemma}.
\end{proof}
\subsection{Proof of Corollary \ref{corc1easy}}\label{corproof}
\begin{proof}
It suffices to consider real gluing parameters $\lambda$ and prove $C^1$ differentiability at $\lambda=0$. When $\lambda$ is real, $|\lambda|^{p-1}\lambda = \lambda^p$. The gluing theorem in \cite{JHOL} states that the gluing map is $C^\infty$ for $\lambda\neq0$. To prove that $\widetilde{\lambda} = \lambda^{1/p}$ is the appropriate gluing parameter, we will first prove that the gluing parameter $s = R^{-\frac{2}{p}}$ gives bounded derivatives away from zero. This is a simple application of the chain rule. Let $h(s)$ represent the gluing map (\ref{glkmap}) with gluing parameter $s$.
\[\left|\frac{d}{ds}h(s)\right| = \left|\frac{dh}{dR}\frac{dR}{ds}\right| = \left|\frac{\frac{dh}{dR}}{\frac{ds}{dR}}\right| \leq \left|C\frac{R^{-\frac{2}{p}-1}}{R^{-\frac{2}{p}-1}}\right| = C.\]
where the last inequality uses Theorem \ref{mainthm}. Thus, the gluing map $h(s)$ is continuous for all $s$ (by the gluing theorem of \cite{JHOL}) and smooth with bounded derivative for $s\neq0$. Therefore, reparametrizing to $s^{1+\varepsilon}$, gives a $C^1$ function with respect to $s$ that has derivative equal to zero at $s=0$. The proof of this for a function $h:\mathbb{R}\rightarrow\mathbb{R}$ is an easy calculus exercise that we do below. The general result follows easily.

Consider $h:\mathbb{R}\rightarrow\mathbb{R}$ that is continuous everywhere and smooth with derivative bounded by $C$ away from zero, let $\varepsilon>0$, and suppose we are given $\eta>0$. Then
\[\frac{|h(s^{1+\varepsilon})-h(0)|}{s} \leq C\frac{s^{1+\varepsilon}}{s}\]
by applying the mean value theorem to the function $h(x^{1+\varepsilon})$ on the interval $[0,s]$. Therefore, since $\varepsilon>0$, by choosing $|s|$ small enough we can ensure $\frac{|h(s^{1+\varepsilon})-h(0)|}{s} \leq \eta$ and hence $h(s^{1+\varepsilon})$ is differentiable with respect to $s$ at $s=0$ and has derivative equal to zero.

Finally, $\lambda=R^{-2}e^{i\theta}$ corresponds to the coordinates obtained by cross ratios as described in Section \ref{crcoordinates}. Thus we can choose $\varepsilon>0$ so that $1+\varepsilon = p$ and hence
\[\widetilde{\lambda}|\widetilde{\lambda}|^{\varepsilon} = R^{-2(1+\varepsilon)/p}e^{i\theta} = R^{-2}e^{i\theta} = \lambda.\]
\end{proof}

\bibliographystyle{amsalpha}
\bibliography{research}

\providecommand{\bysame}{\leavevmode\hbox to3em{\hrulefill}\thinspace}
\providecommand{\MR}{\relax\ifhmode\unskip\space\fi MR }
\providecommand{\MRhref}[2]{%
  \href{http://www.ams.org/mathscinet-getitem?mr=#1}{#2}
}
\providecommand{\href}[2]{#2}
\begin{thebibliography}{FOOO09}

\bibitem[Cas]{axiomme}
Robert Castellano, \emph{Genus zero {G}romov-{W}itten axioms via {K}uranishi
  atlases}, arXiv:1601.04048.

\bibitem[FO99]{fo}
Kenji Fukaya and Kaoru Ono, \emph{Arnold conjecture and {G}romov-{W}itten
  invariants}, Topology \textbf{38} (1999), 933--1048.

\bibitem[FOOO]{fooo12}
Kenji Fukaya, Yong-Geun Oh, Hiroshi Ohta, and Kaoru Ono, \emph{Technical detail
  on {K}uranishi structure and virtual fundamental chain}, arXiv:1209.4410.

\bibitem[FOOO09]{fooo}
\bysame, \emph{Lagrangian intersection theory, anomaly and obstruction, parts
  {I} and {II}}, AMS/IP Studies in Advanced Mathematics, Providence, RI and
  Somerville, MA, 2009.

\bibitem[GW11]{wg}
Eduardo Gonzalez and Chris Woodward, \emph{Deformations of sympletic vortices},
  Annals of Global Analysis and Geometry \textbf{39} (2011), 45--82.

\bibitem[HWZ]{hwz1}
Helmut Hofer, Krzysztof Wysocki, and Eduard Zehnder, \emph{Applications of
  polyfold theory {I}: {T}he polyfolds of {G}romov-{W}itten theory},
  arXiv:1107.2097.

\bibitem[HWZ08]{hwzdm}
\bysame, \emph{Deligne-mumford-type spaces with a view towards symplectic field
  theory}, http://www.cims.nyu.edu/~hofer/polyfolds/lecture7-9.pdf, 2008,
  [Online; accessed 31-March-2016].

\bibitem[KM94]{axioms}
Maxim Kontsevich and Yuri Manin, \emph{Gromov-{W}itten classes, quantum
  cohomology, and enumerative geometry}, Communications in Mathematical Physics
  \textbf{164} (1994), 525--562.

\bibitem[Knu83]{knudsen}
Finn~F. Knudsen, \emph{The projectivity of moduli spaces of stable curves
  {II}}, Functional Analysis and its Applications \textbf{19} (1983), 161--199.

\bibitem[LS]{lisheng}
An-Min Li and Li~Sheng, \emph{The exponential decay of gluing maps and
  {G}romov-{W}itten invariants}, arXiv:1506.0633.

\bibitem[LT98]{litian}
Jun Li and Gang Tian, \emph{Virtual moduli cycles and {G}romov-{W}itten
  invariants for general symplectic manifolds}, Topics in symplectic
  4-manifolds (Cambridge, MA), International Press, 1998, pp.~47--83.

\bibitem[McD]{notes}
Dusa McDuff, \emph{Notes on {K}uranishi atlases}, arXiv:1411.4306.

\bibitem[MS12]{JHOL}
Dusa McDuff and Dietmar Salamon, \emph{J-holomorphic curves and symplectic
  topology}, American Mathematical Society, Providence, RI, 2012.

\bibitem[MWa]{mwfund}
Dusa McDuff and Katrin Wehrheim, \emph{The fundamental class of smooth
  {K}uranishi atlases with trivial isotropy}, arXiv:1508.01560.

\bibitem[MWb]{mwiso}
\bysame, \emph{Smooth {K}uranishi atlases with isotropy}, arXiv:1508.01556.

\bibitem[MWc]{mwtop}
\bysame, \emph{The topology of {K}uranishi atlases}, arXiv:1508:01844.

\bibitem[Par]{pardon}
John Pardon, \emph{An algebraic approach to virtual fundamental cycles on
  moduli spaces of {J}-holomorphic curves}, arXiv:1309.2370.

\bibitem[Rua99]{ruanvir}
Yongbin Ruan, \emph{Virtual neighborhoods and pseudoholomorphic curves},
  Turkish J. Math (1999), 161--231.

\bibitem[Sie99]{siebert}
Bernd Siebert, \emph{Symplectic {G}romov-{W}itten invariants}, New Trends in
  Algebraic Geometry (Cantanese, Peters, Reid, and Hulek, eds.), Cambridge
  University Press, 1999, pp.~375--424.

\end{thebibliography}

\end{document}